\documentclass[11pt, a4paper, twoside]{article}
\usepackage{amsfonts}
\usepackage{mathrsfs}
\usepackage{amsmath, amsthm}
\usepackage{pgf,tikz}
\usepackage{amssymb}
\usepackage{fancyhdr}
\usepackage[colorlinks, linkcolor=black, anchorcolor=black, citecolor=black]{hyperref}
\usepackage{bbm}
\usepackage[shortlabels]{enumitem}
\usepackage{verbatim}

 \usepackage{todonotes}
\usetikzlibrary{patterns}
\numberwithin{equation}{section}

\allowdisplaybreaks

\begin{document}

\fancyhead[OR]{\thepage}

\renewcommand{\headrulewidth}{0pt}
\renewcommand{\thefootnote}{\fnsymbol {footnote}}

\theoremstyle{plain} 
\newtheorem{thm}{\indent\sc Theorem}[section] 
\newtheorem{lem}[thm]{\indent\sc Lemma}
\newtheorem{cor}[thm]{\indent\sc Corollary}
\newtheorem{prop}[thm]{\indent\sc Proposition}
\newtheorem{claim}[thm]{\indent\sc Claim}
\theoremstyle{definition} 
\newtheorem{dfn}[thm]{\indent\sc Definition}
\newtheorem{rem}[thm]{\indent\sc Remark}
\newtheorem{ex}[thm]{\indent\sc Example}
\newtheorem{notation}[thm]{\indent\sc Notation}
\newtheorem{assertion}[thm]{\indent\sc Assertion}
%
%
\numberwithin{equation}{section}
\renewcommand{\proofname}{\indent\sc Proof.} 
\def\C{\mathbb{C}}
\def\R{\mathbb{R}}
\def\Rn{{\mathbb{R}^n}}
\def\M{\mathbb{M}}
\def\N{\mathbb{N}}
\def\Q{{\mathbb{Q}}}
\def\Z{\mathbb{Z}}
\def\F{\mathcal{F}}
\def\L{\mathcal{L}}
\def\S{\mathcal{S}}
\def\supp{\operatorname{supp}}
\def\essi{\operatornamewithlimits{ess\,inf}}
\def\esss{\operatornamewithlimits{ess\,sup}}
\def\dlim{\displaystyle\lim}

\fancyhf{}

\fancyhead[EC]{W. Li, H. Wang, Y. Zhai}

\fancyhead[EL]{\thepage}

\fancyhead[OC]{Maximal functions associated with curvatures}

\fancyhead[OR]{\thepage}

\renewcommand{\headrulewidth}{0pt}
\renewcommand{\thefootnote}{\fnsymbol {footnote}}

\title{\textbf{Sparse domination and $L^{p} \rightarrow L^{q}$ estimates for maximal functions associated with curvature}
\footnotetext {This work is supported by Natural Science Foundation of China (No.11601427);
 China Postdoctoral Science Foundation (No. 2021M693139); the Fundamental Research Funds for the Central Universities (No. E1E40109); ERC project FAnFArE no. 637510, the region Pays de la Loire and CNRS.}
\footnotetext {{}{2000 \emph{Mathematics Subject
 Classification}: 42B20, 42B25.}}
\footnotetext {{}\emph{Key words and phrases}: Maximal function, vanishing curvature, sparse domination, $L^{p} \rightarrow L^{q}$ estimate, weighted inequalities. } } \setcounter{footnote}{0}
\author{
Wenjuan Li, Huiju Wang, Yujia Zhai}

\date{}
\maketitle

\begin{abstract}
In this paper, we study maximal functions along some finite type curves and hypersurfaces. In particular, various impacts of non-isotropic dilations are considered. Firstly, we provide a generic scheme that allows us to deduce the sparse domination bounds for global maximal functions under the assumption that the corresponding localized maximal functions satisfy the $L^{p}$ improving properties. Secondly, for the localized  maximal functions with non-isotropic dilations of  curves and hypersurfaces whose curvatures vanish to finite order at some points, we establish the $L^{p}\rightarrow L^{q}$ bounds $(q >p)$. As a corollary, we obtain the weighted inequalities for the corresponding global maximal functions, which generalize the known unweighted estimates.
\end{abstract}

\tableofcontents

\section{Introduction}
The spherical maximal function
\begin{equation} \label{sphMax}
 \sup_{t >0}\left|\int_{\mathcal{S}^{n-1}}f(y-tx)d\sigma(x)\right|,
\end{equation}
where $d\sigma$ is normalized surface measure on the sphere $S^{n-1}$, was first studied by Stein  \cite{stein3} in 1976. The sharp $L^{p}$-estimate was established by Stein \cite{stein3} for $p>n/(n-1)$ when $n \ge 3$, and later by  Bourgain \cite{JB} for $p> 2$ when $n = 2$. Then many authors turned to the study of generalizations of the spherical maximal function, i.e. the sphere is replaced by a more general smooth hypersurface in $ {\mathbb{R}}^n$. In 2019, Lacey \cite{Lacey} proved the sparse bound for the spherical maximal functions \eqref{sphMax} and obtained the weighted inequalities involving weights of Muckenhoupt and reverse H\"older type. In \cite{Lacey}, the sparse domination for the global spherical maximal functions mainly follows from the stopping-time decomposition and the $L^p$ improving properties of the appropriate localized spherical maximal functions.
Both the lacunary and full spherical maximal functions were considered in \cite{Lacey}, whose localized versions correspond to
\begin{equation}\label{local operator}
 \mathcal{M}_Ef(y)=\sup_{t \in E}\left|\int_{\mathcal{S}^{n-1}}f(y-tx)d\sigma(x)\right|,
\end{equation}
with $E$ being a single point set or a interval respectively.

For generic subset $E\subset [1,2]$, Anderson-Hughes-Roos-Seeger \cite{AHRS} considered the spherical maximal operators with sets $E$ of dilates intermediate between the two above extreme cases. The bounds for such spherical maximal operators will depend on Minkowski dimension and the upper Assouad dimension of the set $E$. Using the same stopping-time decomposition as \cite{Lacey}, they derived the sparse domination bound and weighted estimates for this class of spherical maximal operators. Later, Roos-Seeger \cite{Roos} gave a deep characterization of the closed convex sets which can occur as closure of the sharp $L^p$ improving region of $\mathcal{M}_E$ for some $E$. This region relies on the Minkowski dimension of $E$, and the Assouad spectrum of $E$.

A natural generalization of the spherical maximal functions is to characterize the $L^p$-bounded
ness properties of the maximal operator associated to hypersurface where the Gaussian curvature at some points is allowed to vanish. Related works can be found in Iosevich \cite{I},  Sogge-Stein \cite{ss}, Sogge \cite{Sogge}, Cowling-Mauceri \cite{cw1,cw2}, Nagel-Seeger-Wainger \cite{nsw},  Iosevich-Sawyer \cite{ios1,ios2}, Iosevich-Sawyer-Seeger \cite{Iss}, Ikromov-Kempe-M\"{u}ller \cite{IKMU} and references therein. Nevertheless, the weighted estimates for maximal operators associated with hypersurfaces with vanishing curvature are far beyond satisfaction.


Meanwhile, maximal operators defined by averages  over curves or surfaces with non-isotropic dilations have also been extensively considered. In 1970, in the study of a problem related
to Poisson integrals for symmetric spaces, Stein raised the question as to when the operator $\mathcal{M}_{\gamma}$ defined by
\begin{equation*} \label{eq:stein}
\mathcal{M}_{\gamma}f(x)=\sup_{h>0}\frac{1}{h}\int_{0}^h|f(x-\gamma(t))|dt,
\end{equation*}
where $\gamma(t)=(A_1t^{a_1},A_2t^{a_2},\cdots,A_nt^{a_n})$ and $A_1,A_2,\cdots,A_n$ are real, $a_i>0$, is bounded on $L^p(\mathbb{R}^n)$.
Nagel-Riviere-Wainger \cite {NRW} showed that the $L^p$-boundedness of $\mathcal{M}_{\gamma}$ holds for  $p>1$ for the special case $\gamma(t)=(t,t^2)$ in $\mathbb{R}^2$ and Stein \cite{stein} for  homogeneous curves in $\mathbb{R}^n$.  For maximal functions $\mathcal{M}$ associated with non-isotropic dilations in higher dimensions, one can see the works by Greenleaf \cite{Greenleaf}, Sogge-Stein \cite{ss}, Iosevich-Sawyer \cite{ios2}, Ikromov-Kempe-M\"{u}ller \cite{IKMU}, Zimmermann \cite{Zimmermann}. More information can be found in \cite{WL} and references therein.

In \cite{WL}, the first author of this paper established $L^p$-estimates for the maximal function related to the hypersurface $\Gamma=\{(x_{1}, x_{2},\Phi(x_1,x_2))\in \mathbb{R}^3:(x_1,x_2)\in \Omega\subset \mathbb{R}^2, 0\in\Omega\}$ with associated dilations $\delta_t(x)=(t^{a_1}x_1,t^{a_2}x_2,t^{a_3}x_3)$, defined by
 \begin{equation}\label{Eq1.4}
\sup_{t >0}\left|\int_{\mathbb{R}^2}f(y-\delta_t(x_1,x_2,\Phi(x_1,x_2)))\eta(x)dx\right|,
\end{equation}
where $\eta$ is somooth fucntion supported in a sufficiently small neighborhood $U\subset\Omega$ of the origin, and $\Phi(x_1,x_2)\in C^{\infty}(\Omega)$. When $2a_{2}\neq a_{3}$ and the hypersurface $\Gamma$ has one non-vanishing principle curvature, one can see that the local smoothing estimate from \cite{mss2,mss}
will imply the boundedness for the corresponding maximal function. Employing the isometric transformation, this result can also be generalized to the maximal functions along smooth hypersurfaces of finite type with order $d$, $d  >2$, but $da_{2} \neq a_{3}$ is required in the  dilation. When $da_{2} = a_{3}$ and $\Phi(x_1, x_2) = x_2^d(1+(\mathcal{O}(x_2^m)))$ near the origin, where $d \geq 2, m \geq 1$,
by freezing operator, she observed the necessity to estimate a family of corresponding Fourier integral operators in $\mathbb{R}^2\times \mathbb{R}$ which fail to satisfy the ``cinematic curvature" condition  uniformly. This means that celebrated local smoothing estimates could not be directly applied there. In fact, the similar problem has also arised in the study of  maximal functions associated with the curve $\gamma(x)=(x,x^d\phi(x))$ and associated dilations $(t,t^d)$, i.e.,
\begin{equation}\label{Eq1.3}
\sup_{t >0}\left|\int_{\mathbb{R}}f(y_1-tx,y_2-t^dx^d\phi(x))\eta(x)dx\right|,
\end{equation}
 where $\eta(x)$ is supported in a sufficiently small neighborhood of the origin, and $\phi(x)$ satisfies (\ref{phi}) below.
In \cite{WL},  new ideas were established to obtain $L^{4}$-estimate for Fourier integral operators which fail to satisfy the ``cinematic curvature" condition  uniformly, and to develop the sharp $L^p$-estimates for the maximal function (\ref{Eq1.3}).

The goal of our paper is to establish the weighted $L^p$-inequalities for maximal operators associated to curves and hypersurfaces where the Gaussian curvatures at some points are allowed to vanish. For example, we are able to develop the weighted $L^p$- estimates for $(1)$ maximal operators along generic curves of finite type in dimension $2$; $(2)$ maximal operators related to homogeneous curves studied by Stein \cite{stein}; $(3)$ most maximal operators associated with non-isotropic dilations considered in \cite{WL}.

Since the sparse domination bounds imply the weighted estimates (see \cite{Lacey}) in a straightforward fashion, the main methodology of this paper consists of \textbf{two} parts:
\begin{enumerate}
\item \label{reduction}
\textit{a generic scheme that relates global maximal operators to the corresponding localized maximal operators and employs the stopping-time decompositions developed by Lacey \cite{Lacey}, which
reduces the sparse domination bounds to the $L^p$ improving property of the localized maximal operators;
\item
the study of $L^p \rightarrow L^q$ regularity property of the localized maximal operators, particularly the ones which correspond to Fourier integral operators that fail to satisfy the ``cinematic curvature" condition uniformly.}
\end{enumerate}

An immediate observation from our methodology is that the sparse domination bounds and thus the weighted $L^p$-estimates for a global maximal operator can be reduced to the validation of the $L^p \rightarrow L^q$ regularity property of the corresponding localized maximal operator. We noticed that the articles \cite{Digo, BJS} described a very general and precise connection between sparse bounds and $L^p$ improving properties, but in an isotropic setting. However, our paper is mainly about non-isotropic cases.

In the next few subsections, we will state our theorems on sparse domination (Subsection \ref{SD}) and $L^p \rightarrow L^q$ regularity properties (Subsection \ref{regularity}), and deduce from those two main ingredients, the weighted $L^p$-estimates for global maximal operators of interest (Subsection \ref{weighted}).  

\subsection{Sparse domination} \label{SD}
This subsection is devoted to the derivation of the sparse domination bound for a generic global maximal operator assuming the validity of a \textit{local property} that will be specified. Such \textit{local property} is a consequence of the $L^p \rightarrow L^q$ regularity property, as will be shown in Subsection \ref{sec:con:lemma}.

Let $\Omega \in \mathbb{R}^{k}$ be an open neighborhood of the origin. Suppose that $\Gamma$ is a $k$-dimensional surface in $\mathbb{R}^n$ for $1 \leq k \leq n$ which is parametrized by
$$\Gamma=\{(\Phi_1(x), \ldots, \Phi_{n-1}(x), \Phi_n(x)) \in \mathbb{R}^n: x \in \Omega\subset \mathbb{R}^{k}\},$$
where $\Phi_i$ are smooth maps for $1 \leq i \leq n$.
 Denote by $\delta_t$ the non-isotropic dilations in $\mathbb{R}^n$ given by
\begin{equation}\label{dilation:generic}
\delta_t(x)=(t^{a_1}x_1,, \ldots, t^{a_{n-1}}x_{n-1}, t^{a_n}x_n),
\end{equation}
for $a_1, \ldots, a_n >0$.

Let $\eta: \mathbb{R}^{k} \rightarrow \mathbb{R}$ be a smooth function supported on a sufficiently small neighborhood $U \subseteq \Omega$ of the origin. We define the average associated to the hypersurface parametrized by $\vec{\Phi}:=(\Phi_i)_{1 \leq i \leq n}$ as
\begin{equation*}
A^{\vec{\Phi}}_{t} f(y):=  \int_{\mathbb{R}^{k}}f\left(y-\delta_t\left(\Phi_1(x), \ldots, \Phi_n(x)\right)\right)\eta(x)dx.
\end{equation*}

The corresponding maximal function can then be defined as
\begin{equation} \label{max:Phi}
\mathcal{M}_{\vec{\Phi}}f(y):=\sup_{t>0} |A^{\vec{\Phi}}_{t}f(y)|.
\end{equation}

We would like to remark that the full spherical maximal operator in \cite{Lacey} corresponds to the case when $\Phi_i: \Omega \subseteq \mathbb{R}^{n-1}\rightarrow \mathbb{R}$ for $1 \leq i \leq n$ and $\Phi_i(x) = x_i$ for $1 \leq i \leq n-1$, $\Phi_n(x) = \left(1-(x_1^2+\ldots + x_{n-1}^2)\right)^{\frac{1}{2}}$. The dilation $\delta_t$ is defined in \eqref{dilation:generic} with $a_1=\ldots = a_n = 1$. Another closely related class of operators is the maximally truncated Hilbert transforms along well curved curves (\cite{stein2}), one of which can be expressed as \eqref{max:Phi}
with the dilation $\delta_t$ \eqref{dilation:generic} satisfying $0 < a_1 < \ldots < a_n < \infty$ and $\Phi_i: \Omega \subseteq \mathbb{R} \rightarrow \mathbb{R}$ for $1 \leq i \leq n$ such that $\Phi_i(x) = |x|^{a_i}$ with the same $a_i$s.
The singular integral variant of this example is shown by Cladek-Ou \cite{cou} to satisfy the sparse domination bound and the weighted estimates.

A simple but important observation is that the maximal function defined in terms of a generic non-isotropic dilation $\delta_t$ \eqref{dilation:generic}  can be reduced to a maximal function with the dilation
\begin{equation} \label{dilation:restrict}
\tilde{\delta}_t(x):= (t^{b_1}x_1, \ldots, t^{b_n}x_n)
\end{equation}
 satisfying the condition that
\begin{equation} \label{dilation:geq1}
b_j \geq 1, \text{ for all } 1\leq j \leq n.
\end{equation}
Suppose that $a_{j_0} = \min_{1 \leq j \leq n} a_j >0$. Then one can let $\tilde{t} := t^{a_{j_0}}$ and the supreme can be taken over $\tilde{t}$ so that
\begin{align*}
M_{\vec{\Phi}} f(y) = & \sup_{t >0} \left|\int_{\mathbb{R}^{k}}f\left(y-\left(t^{a_1}\Phi_1(x), \ldots, t^{a_n}\Phi_n(x)\right)\right)\eta(x)dx\right| \\
= & \sup_{\tilde{t} >0} \left|\int_{\mathbb{R}^{k}}f\left(y-\left(\tilde{t}^{\frac{a_1}{a_{j_0}}}\Phi_1(x), \ldots , \tilde{t}^{\frac{a_n}{a_{j_0}}}\Phi_n(x)\right)\right)\eta(x)dx\right|,
\end{align*}
where $b_j := \frac{a_j}{a_{j_0}} \geq 1$ for all $1 \leq j \leq n $ by the minimality of $a_{j_0}$. From now on, we will therefore assume without loss of generality that the dilation satisfies the condition \eqref{dilation:geq1}, which allows a separation of scales that will become crucial later on.

We will now describe the important \textit{local property}, namely local continuity property, which gives rise to the sparse bounds. The operator of consideration is defined as
\begin{equation} \label{max:local}
\tilde{\mathcal{M}}_{\vec{\Phi}} f := \sup_{1 \leq t \leq 2} A^{\vec{\Phi}}_{t} f,
\end{equation}
which reflects the "local" behavior since the local supreme is inspected instead of the global supreme. \vskip .1in
 \noindent
 \textbf{\textit{Local Continuity Property.}} We say that the operator $\tilde{\mathcal{M}}_{\vec{\Phi}}$ has the \textit{local continuity property}
 if there exists a convex polygon $\mathcal{L}_n$ satisfying the following two conditions:
 \begin{enumerate}
 \item
 $\mathcal{L}_n$ has non-empty interior;
 \item
  for all $(\frac{1}{p}, \frac{1}{q})$ in the interior of $\mathcal{L}_n$, there is some $\epsilon = \epsilon(n,p,q)$ such that for all $|z| \leq 1$,
 \begin{equation} \label{cont:scale1}
 \| \sup_{1 \leq t \leq 2} |A^{\vec{\Phi}}_{t} f - \tau_z A^{\vec{\Phi}}_{t} f| \|_q \lesssim |z|^\epsilon \|f\|_p,
 \end{equation}
 where $\tau_z F(y) := F(y-z)$ denotes the translation of $F$ by $y$.
\end{enumerate}

We will state the sparse bounds established upon the local continuity property. We set $\mathcal{L}_n'$ be the dual range of exponents to $\mathcal{L}_n$, which is defined as $\mathcal{L}_n' := \{(\frac{1}{p}, \frac{1}{q'}):(\frac{1}{p},1-\frac{1}{q'}) \in \mathcal{L}_n\}$.

Heuristically speaking, the sparse bounds capture the sparseness of the geometric objects where the  ``main mass" is distributed. In particular, the geometrical objects one encounters are generalized cubes defined as follows.
\begin{dfn}
Let $\delta$ denote the dilation specified in \eqref{dilation:restrict}. Let $\mathcal{R}$ denote the collection of all axes-parallel hyperrectangles. Then the collection of $\delta$-cubes with dyadic size is defined as
\begin{equation*}
\mathcal{Q}^\delta := \{Q \in \mathcal{R}: l_1(Q) = 2^{\lceil kb_1 \rceil}, \ldots, l_n(Q) = 2^{\lceil kb_n \rceil}, \text{ for some } k \in \mathbb{Z}\},
\end{equation*}
where $l_j(Q)$ denote the $j$-th side-length of $Q$ and $b_j \geq 1$ for $1 \leq j \leq n$.
\end{dfn}
\begin{rem} \label{sep:scale}
One crucial property implied by the condition $b_j \geq 1$ is the separation of scales: for any $k \neq k' \in \mathbb{Z}$,
\begin{equation*}
|k b_j - k' b_j| = |k-k'|b_j \geq b_j \geq 1,
\end{equation*}
or equivalently, for any $l_j(Q)$, there exists a unique $k \in \mathbb{Z}$ such that $2^{\lceil kb_j \rceil} = l_j(Q)$. As a consequence, given $l_j(Q)$ for some $1 \leq j \leq n$, one can determine $k$ and thus all other side-lengths of $Q$.
\end{rem}

\begin{dfn}
We say that a collection of $\delta$-cubes $\mathcal{S}$ is \textit{sparse} if there are sets $\{E_S \subseteq S: S \in \mathcal{S} \}$ such that they are pairwise disjoint and $|E_S| > \frac{1}{4}|S|$ for all $S \in \mathcal{S}$.
\end{dfn}

\begin{thm} \label{thm:sparse}
Let $\mathcal{M}_{\vec{\Phi}}$ denote the maximal function defined in \eqref{max:Phi}. Suppose that the corresponding local operator $\tilde{\mathcal{M}}_{\vec{\Phi}}$ \eqref{max:local} satisfies the local continuity property in the range $\mathcal{L}_n$. Then for all bounded compactly supported functions  $f$, $g$ and for any $(\frac{1}{p}, \frac{1}{q'}) \in \mathcal{L}'_n$, there exists a constant $C < \infty$ such that
\begin{equation*}
| \langle \mathcal{M}_{\vec{\Phi}} f, g \rangle | \leq C \sup_{\mathcal{S}} \Lambda_{\mathcal{S},p,q'}(f,g),
\end{equation*}
where the supreme is taken over all possible sparse collections of $\delta$-cubes.
\end{thm}

In the paper, we will restrict our focus on proving sparse domination bounds (and thus weighted inequalities) for for maximal operators over hypersurfaces, namely when $\Phi_i: \Omega \subseteq \mathbb{R}^{n-1} \rightarrow \mathbb{R}$ for $1 \leq i \leq n$, $\Phi_i(x)= x_i$ for $1 \leq i \leq n-1$ and $\Phi_n =: \Phi$ is a smooth function. We will explicitly define this class of maximal operators for further references.

Let
\begin{equation*}
A_{t} f(y):= \int_{\mathbb{R}^{n-1}}f\left(y-\delta_t\left(x', \Phi(x')\right)\right)\eta(x')dx'.
\end{equation*}

The corresponding maximal function can thus be defined as
\begin{equation} \label{max:Phi}
\mathcal{M}_{}f(y):=\sup_{t>0}|A^{}_{t}f(y)|.
\end{equation}

\subsection{$L^p \rightarrow L^q$ estimates and continuity lemmas for localized maximal functions associated with curvature} \label{regularity}
The $L^p$-improving estimate ($L^p\rightarrow L^q$ estimate) for the local  maximal operator $\mathcal{M}_E$  defined by (\ref{local operator}) have been extensively studied, where $E $ is a subset of $[1,2]$.
In one extreme case when $E=\{\textmd{single point}\}$, the corresponding maximal operator reduces to an averaging operator, for which Littman \cite{Littman} has proved the $L^p\rightarrow L^q$ boundedness if and only if $(\frac{1}{p},\frac{1}{q})$ belongs to the closed triangle with corners $(0,0), (1,1),(\frac{n}{n+1},\frac{1}{n+1})$.
For the other extreme case when $E=[1,2]$,  Schlag \cite{WS1} showed that in $\mathbb{R}^2$, $\mathcal{M}_E$ is actually bounded in the interior of the triangle with vertices $(0,0)$, $(1/2,1/2)$ and $(2/5,1/5)$. This result was obtained using "combinatorial method" in \cite{WS1}. Based on some local smoothing estimates, an alternative proof was given by Schlag-Sogge \cite{WS2} later. Schlag-Sogge \cite{WS2} also established $L^{p} \rightarrow L^{q}$ estimates for the local maximal functions of hypersurfaces satisfying ``cinematic curvature" condition  in $\mathbb{R}^{n}$ for $n\geq 3$. It is worthwhile to mention that, using bilinear cone restriction estimate, Lee \cite{SL1} improved the local smoothing estimate in \cite{WS2} and attained the endpoint estimate for the local circular maximal function in $\mathbb{R}^{n}$ for $n\geq 2$. Anderson-Hughes-Roos-Seeger \cite{AHRS} considered the spherical maximal operators with sets $E$ of dilates intermediate between the above two extreme cases. Let $\beta$ and $\gamma$ be the upper Minkowski dimension and the upper Assouad dimension of the set $E$ respectively, where $0\leq\beta\leq \gamma\leq 1$ if $n\geq 3$ and $0\leq\beta\leq \gamma\leq 1/2$ if $n=2$. Then $\mathcal{M}_E$ is bounded in the interior of the closed convex hull of the points $Q_1=(0,0),Q_2=(\frac{n-1}{n-1+\beta},(\frac{n-1}{n-1+\beta}),Q_3=(\frac{n-\beta}{n-\beta+1},\frac{1}{n-\beta+1}),Q_4=(\frac{n(n-1)}{n^2+2\gamma-1},\frac{n-1}{n^2+2\gamma-1})$ and the line segment connecting $Q_1$ and $Q_2$, with $Q_1$ included and $Q_2$ excluded. Furthermore, they obtained the sparse domination for the global maximal operators defined by $\sup_{k\in \mathbb{Z}}\sup_{t \in E}\left|\int_{\mathcal{S}^{n-1}}f(y-2^ktx)d\sigma(x)\right|$, which covers the corresponding result in \cite{Lacey}. However, it still remains open whether the $L^p\rightarrow L^q$ estimates hold for the maximal operators related to hypersurfaces where the Gaussian curvatures at some points are allowed to vanish.

In this subsection, we state some $L^p\rightarrow L^q$  estimates for some local maximal operators associated with anistropic dilations of curves in $\mathbb{R}^{2}$ or hypersurfaces in $\mathbb{R}^{3}$ where the Gaussian curvatures at some points are allowed to vanish to finite order at some points. Then the corresponding  continuity lemmas are given as inferences.

\subsubsection{Localized maximal functions associated with  curves  in $\mathbb{R}^{2}$}\label{subsection1.1}
According to the previous results from \cite{WS1,WS2}, the local circular maximal operator (when $n=2$ and $E= [1, 2]$ in equality (\ref{local operator})) is $L^p\rightarrow L^q$ bounded if  $(\frac{1}{p},\frac{1}{q})$ belongs to
\begin{equation}\label{delta0}
\Delta_{0} := \{(\frac{1}{p},\frac{1}{q}): \frac{1}{2p} < \frac{1}{q} \le \frac{1}{p} , \frac{1}{q}  > \frac{3}{p} -1\} \cup \{(0, 0)\}.
\end{equation}
It is clear that the circle has non-vanishing curvature. In this subsection, we consider the $L^p\rightarrow L^q$ estimates for local maximal operators along curves of finite type $d$ ($d \ge 2$) at the origin. For convenience, we define the following regions of boundedness exponents that will be referred to later on:
\begin{equation}\label{delta1}
\Delta_{1} := \{(\frac{1}{p},\frac{1}{q}): \frac{1}{2p} < \frac{1}{q} \le \frac{1}{p} , \frac{1}{q}  > \frac{3}{p} -1, \frac{1}{q} > \frac{1}{p} - \frac{1}{d+1}\} \cup \{(0, 0)\};
\end{equation}
\begin{equation}\label{delta2}
\Delta_{2} := \{(\frac{1}{p},\frac{1}{q}): \frac{1}{2p} < \frac{1}{q} \le \frac{1}{p} , \frac{1}{q}  > \frac{d+1}{p} -1\} \cup \{(0, 0)\};
\end{equation}
\begin{equation}\label{delta3}
\Delta_{3} := \{(\frac{1}{p},\frac{1}{q}): \frac{1}{2p} < \frac{1}{q} \le \frac{1}{p} , \frac{1}{q}  > \frac{2}{p} -1, \frac{1}{q} > \frac{1}{p} - \frac{1}{d+1}\} \cup \{(0, 0), (1, 1)\}.
\end{equation}

\begin{figure}
\centering

\tikzset{
pattern size/.store in=\mcSize, 
pattern size = 5pt,
pattern thickness/.store in=\mcThickness, 
pattern thickness = 0.3pt,
pattern radius/.store in=\mcRadius, 
pattern radius = 1pt}
\makeatletter
\pgfutil@ifundefined{pgf@pattern@name@_sa2perfxd}{
\pgfdeclarepatternformonly[\mcThickness,\mcSize]{_sa2perfxd}
{\pgfqpoint{0pt}{-\mcThickness}}
{\pgfpoint{\mcSize}{\mcSize}}
{\pgfpoint{\mcSize}{\mcSize}}
{
\pgfsetcolor{\tikz@pattern@color}
\pgfsetlinewidth{\mcThickness}
\pgfpathmoveto{\pgfqpoint{0pt}{\mcSize}}
\pgfpathlineto{\pgfpoint{\mcSize+\mcThickness}{-\mcThickness}}
\pgfusepath{stroke}
}}
\makeatother
\tikzset{every picture/.style={line width=0.75pt}} 

\begin{tikzpicture}[x=0.75pt,y=0.75pt,yscale=-1,xscale=1]

\draw  [color={rgb, 255:red, 255; green, 255; blue, 255 }  ,draw opacity=1 ][fill={rgb, 255:red, 74; green, 144; blue, 226 }  ,fill opacity=0.3 ] (136.85,1677.98) -- (44.28,1726.15) -- (162.76,1606.93) -- cycle ;
\draw   (167.68,1602.07) .. controls (167.68,1600.25) and (166.23,1598.78) .. (164.43,1598.78) .. controls (162.63,1598.78) and (161.17,1600.25) .. (161.17,1602.07) .. controls (161.17,1603.89) and (162.63,1605.37) .. (164.43,1605.37) .. controls (166.23,1605.37) and (167.68,1603.89) .. (167.68,1602.07) -- cycle ;
\draw   (141.04,1674.22) .. controls (141.04,1672.4) and (139.58,1670.92) .. (137.78,1670.92) .. controls (135.98,1670.92) and (134.53,1672.4) .. (134.53,1674.22) .. controls (134.53,1676.04) and (135.98,1677.51) .. (137.78,1677.51) .. controls (139.58,1677.51) and (141.04,1676.04) .. (141.04,1674.22) -- cycle ;
\draw  [fill={rgb, 255:red, 0; green, 0; blue, 0 }  ,fill opacity=1 ] (41.41,1724.93) .. controls (41.41,1726.76) and (42.76,1728.25) .. (44.43,1728.25) .. controls (46.1,1728.25) and (47.46,1726.76) .. (47.46,1724.93) .. controls (47.46,1723.09) and (46.1,1721.61) .. (44.43,1721.61) .. controls (42.76,1721.61) and (41.41,1723.09) .. (41.41,1724.93) -- cycle ;
\draw [line width=1.5]    (44.28,1726.15) -- (162.92,1604.26) ;
\draw  [dash pattern={on 4.5pt off 4.5pt}]  (47.67,1498.33) -- (278.86,1499.27) ;
\draw  [dash pattern={on 4.5pt off 4.5pt}]  (278.86,1499.27) -- (280.72,1723.52) ;
\draw    (44.44,1724.93) -- (298.25,1725.4) ;
\draw [shift={(300.25,1725.4)}, rotate = 180.11] [color={rgb, 255:red, 0; green, 0; blue, 0 }  ][line width=0.75]    (10.93,-3.29) .. controls (6.95,-1.4) and (3.31,-0.3) .. (0,0) .. controls (3.31,0.3) and (6.95,1.4) .. (10.93,3.29)   ;
\draw    (44.28,1726.15) -- (45.49,1481) ;
\draw [shift={(45.5,1479)}, rotate = 90.28] [color={rgb, 255:red, 0; green, 0; blue, 0 }  ][line width=0.75]    (10.93,-3.29) .. controls (6.95,-1.4) and (3.31,-0.3) .. (0,0) .. controls (3.31,0.3) and (6.95,1.4) .. (10.93,3.29)   ;
\draw  [draw opacity=0][fill={rgb, 255:red, 74; green, 144; blue, 226 }  ,fill opacity=0.3 ] (476.3,1681.19) -- (376.57,1726.94) -- (496.9,1614.34) -- cycle ;
\draw  [color={rgb, 255:red, 74; green, 144; blue, 226 }  ,draw opacity=0.13 ][pattern=_sa2perfxd,pattern size=6pt,pattern thickness=2.25pt,pattern radius=0pt, pattern color={rgb, 255:red, 245; green, 166; blue, 35}] (496.9,1614.34) -- (487.26,1646.09) -- (427.18,1701.28) -- (374.12,1727.12) -- cycle ;
\draw [line width=1.5]    (376.57,1725.16) -- (496.9,1611.68) ;
\draw    (376.57,1726.94) -- (376.1,1484) ;
\draw [shift={(376.09,1482)}, rotate = 89.89] [color={rgb, 255:red, 0; green, 0; blue, 0 }  ][line width=0.75]    (10.93,-3.29) .. controls (6.95,-1.4) and (3.31,-0.3) .. (0,0) .. controls (3.31,0.3) and (6.95,1.4) .. (10.93,3.29)   ;
\draw   (488.79,1650.73) .. controls (488.79,1649.04) and (487.34,1647.67) .. (485.54,1647.67) .. controls (483.74,1647.67) and (482.28,1649.04) .. (482.28,1650.73) .. controls (482.28,1652.42) and (483.74,1653.8) .. (485.54,1653.8) .. controls (487.34,1653.8) and (488.79,1652.42) .. (488.79,1650.73) -- cycle ;
\draw   (433.79,1702.06) .. controls (433.79,1700.35) and (432.31,1698.95) .. (430.48,1698.95) .. controls (428.66,1698.95) and (427.18,1700.35) .. (427.18,1702.06) .. controls (427.18,1703.78) and (428.66,1705.18) .. (430.48,1705.18) .. controls (432.31,1705.18) and (433.79,1703.78) .. (433.79,1702.06) -- cycle ;
\draw   (496.76,1609.64) .. controls (496.76,1608.1) and (498.09,1606.84) .. (499.74,1606.84) .. controls (501.38,1606.84) and (502.71,1608.1) .. (502.71,1609.64) .. controls (502.71,1611.19) and (501.38,1612.44) .. (499.74,1612.44) .. controls (498.09,1612.44) and (496.76,1611.19) .. (496.76,1609.64) -- cycle ;
\draw  [fill={rgb, 255:red, 0; green, 0; blue, 0 }  ,fill opacity=1 ] (379.87,1724.71) .. controls (379.87,1723) and (378.39,1721.6) .. (376.57,1721.6) .. controls (374.74,1721.6) and (373.26,1723) .. (373.26,1724.71) .. controls (373.26,1726.43) and (374.74,1727.83) .. (376.57,1727.83) .. controls (378.39,1727.83) and (379.87,1726.43) .. (379.87,1724.71) -- cycle ;
\draw    (376.57,1724.71) -- (641.5,1722.07) ;
\draw [shift={(643.5,1722.05)}, rotate = 179.43] [color={rgb, 255:red, 0; green, 0; blue, 0 }  ][line width=0.75]    (10.93,-3.29) .. controls (6.95,-1.4) and (3.31,-0.3) .. (0,0) .. controls (3.31,0.3) and (6.95,1.4) .. (10.93,3.29)   ;
\draw  [dash pattern={on 4.5pt off 4.5pt}]  (377.5,1503) -- (616.1,1502.45) ;
\draw  [dash pattern={on 4.5pt off 4.5pt}]  (616.5,1724) -- (616.1,1502.45) ;
\draw    (436,1708) -- (463.79,1731.82) -- (470.21,1739.47) ;
\draw [shift={(471.5,1741)}, rotate = 229.98] [color={rgb, 255:red, 0; green, 0; blue, 0 }  ][line width=0.75]    (10.93,-3.29) .. controls (6.95,-1.4) and (3.31,-0.3) .. (0,0) .. controls (3.31,0.3) and (6.95,1.4) .. (10.93,3.29)   ;
\draw    (490.79,1650.73) -- (533.5,1651.94) ;
\draw [shift={(535.5,1652)}, rotate = 181.62] [color={rgb, 255:red, 0; green, 0; blue, 0 }  ][line width=0.75]    (10.93,-3.29) .. controls (6.95,-1.4) and (3.31,-0.3) .. (0,0) .. controls (3.31,0.3) and (6.95,1.4) .. (10.93,3.29)   ;
\draw   (479.55,1684.49) .. controls (479.55,1682.67) and (478.1,1681.19) .. (476.3,1681.19) .. controls (474.5,1681.19) and (473.04,1682.67) .. (473.04,1684.49) .. controls (473.04,1686.31) and (474.5,1687.78) .. (476.3,1687.78) .. controls (478.1,1687.78) and (479.55,1686.31) .. (479.55,1684.49) -- cycle ;

\draw (147,1788.5) node [anchor=north west][inner sep=0.75pt]    {$\Delta _{0}$};
\draw (351,1472.5) node [anchor=north west][inner sep=0.75pt]    {$\frac{1}{q}$};
\draw (303.81,1703.75) node [anchor=north west][inner sep=0.75pt]    {$\frac{1}{p}$};
\draw (140.72,1664.57) node [anchor=north west][inner sep=0.75pt]    {$\left(\frac{2}{5} ,\frac{1}{5}\right)$};
\draw (166.77,1589.25) node [anchor=north west][inner sep=0.75pt]    {$\left(\frac{1}{2} ,\frac{1}{2}\right)$};
\draw (46.26,1724.28) node [anchor=north west][inner sep=0.75pt]    {$( 0,0)$};
\draw (275.79,1471.03) node [anchor=north west][inner sep=0.75pt]    {$( 1,1)$};
\draw (17.84,1465.15) node [anchor=north west][inner sep=0.75pt]    {$\frac{1}{q}$};
\draw (375.41,1726.61) node [anchor=north west][inner sep=0.75pt]    {$( 0,0)$};
\draw (599.79,1476.03) node [anchor=north west][inner sep=0.75pt]    {$( 1,1)$};
\draw (500.77,1583.25) node [anchor=north west][inner sep=0.75pt]    {$\left(\frac{1}{2} ,\frac{1}{2}\right)$};
\draw (469.87,1722.82) node [anchor=north west][inner sep=0.75pt]    {$\left(\frac{2}{d+1} ,\frac{1}{d+1}\right)$};
\draw (534,1630) node [anchor=north west][inner sep=0.75pt]    {$\left(\frac{d}{2d+2} ,\frac{d-2}{2d+2}\right)$};
\draw (640.81,1699.75) node [anchor=north west][inner sep=0.75pt]    {$\frac{1}{p}$};
\draw (436,1791) node [anchor=north west][inner sep=0.75pt]    {$\Delta _{1} \ for\ \textcolor[rgb]{0.29,0.56,0.89}{\mathbf{2\leq d\leq 4}}\textcolor[rgb]{0.29,0.56,0.89}{\ } and\ \textcolor[rgb]{0.95,0.63,0.09}{\mathbf{d\geq 5}} \ $};
\draw (481.72,1668.57) node [anchor=north west][inner sep=0.75pt]    {$\left(\frac{2}{5} ,\frac{1}{5}\right)$};

\end{tikzpicture}

\caption{${\color{black}{\Delta_1}} = {\color{black}\Delta_0}$ for $2 \leq d\leq 4$ and ${\color{black}{\Delta_1}} \subsetneq {\color{black}\Delta_0}$ for $d\geq5$}
\end{figure}

\begin{figure}
\centering

\tikzset{
pattern size/.store in=\mcSize, 
pattern size = 5pt,
pattern thickness/.store in=\mcThickness, 
pattern thickness = 0.3pt,
pattern radius/.store in=\mcRadius, 
pattern radius = 1pt}
\makeatletter
\pgfutil@ifundefined{pgf@pattern@name@_7c8epya2v}{
\pgfdeclarepatternformonly[\mcThickness,\mcSize]{_7c8epya2v}
{\pgfqpoint{0pt}{-\mcThickness}}
{\pgfpoint{\mcSize}{\mcSize}}
{\pgfpoint{\mcSize}{\mcSize}}
{
\pgfsetcolor{\tikz@pattern@color}
\pgfsetlinewidth{\mcThickness}
\pgfpathmoveto{\pgfqpoint{0pt}{\mcSize}}
\pgfpathlineto{\pgfpoint{\mcSize+\mcThickness}{-\mcThickness}}
\pgfusepath{stroke}
}}
\makeatother

 
\tikzset{
pattern size/.store in=\mcSize, 
pattern size = 5pt,
pattern thickness/.store in=\mcThickness, 
pattern thickness = 0.3pt,
pattern radius/.store in=\mcRadius, 
pattern radius = 1pt}
\makeatletter
\pgfutil@ifundefined{pgf@pattern@name@_t8htwyv0t}{
\pgfdeclarepatternformonly[\mcThickness,\mcSize]{_t8htwyv0t}
{\pgfqpoint{0pt}{-\mcThickness}}
{\pgfpoint{\mcSize}{\mcSize}}
{\pgfpoint{\mcSize}{\mcSize}}
{
\pgfsetcolor{\tikz@pattern@color}
\pgfsetlinewidth{\mcThickness}
\pgfpathmoveto{\pgfqpoint{0pt}{\mcSize}}
\pgfpathlineto{\pgfpoint{\mcSize+\mcThickness}{-\mcThickness}}
\pgfusepath{stroke}
}}
\makeatother
\tikzset{every picture/.style={line width=0.75pt}} 

\begin{tikzpicture}[x=0.75pt,y=0.75pt,yscale=-1,xscale=1]

\draw  [color={rgb, 255:red, 255; green, 255; blue, 255 }  ,draw opacity=1 ][fill={rgb, 255:red, 74; green, 144; blue, 226 }  ,fill opacity=0.3 ] (117.91,2272.09) -- (27.83,2316.57) -- (141.48,2203.27) -- cycle ;
\draw    (22.6,2322.02) -- (267.57,2321.82) ;
\draw [shift={(269.57,2321.82)}, rotate = 179.95] [color={rgb, 255:red, 0; green, 0; blue, 0 }  ][line width=0.75]    (10.93,-3.29) .. controls (6.95,-1.4) and (3.31,-0.3) .. (0,0) .. controls (3.31,0.3) and (6.95,1.4) .. (10.93,3.29)   ;
\draw    (22.6,2323.24) -- (23.56,2081) ;
\draw [shift={(23.57,2079)}, rotate = 90.23] [color={rgb, 255:red, 0; green, 0; blue, 0 }  ][line width=0.75]    (10.93,-3.29) .. controls (6.95,-1.4) and (3.31,-0.3) .. (0,0) .. controls (3.31,0.3) and (6.95,1.4) .. (10.93,3.29)   ;
\draw    (340.59,2324.06) -- (339.09,2083) ;
\draw [shift={(339.08,2081)}, rotate = 89.64] [color={rgb, 255:red, 0; green, 0; blue, 0 }  ][line width=0.75]    (10.93,-3.29) .. controls (6.95,-1.4) and (3.31,-0.3) .. (0,0) .. controls (3.31,0.3) and (6.95,1.4) .. (10.93,3.29)   ;
\draw    (342.49,2325.06) -- (593.77,2324.95) ;
\draw [shift={(595.77,2324.95)}, rotate = 179.98] [color={rgb, 255:red, 0; green, 0; blue, 0 }  ][line width=0.75]    (10.93,-3.29) .. controls (6.95,-1.4) and (3.31,-0.3) .. (0,0) .. controls (3.31,0.3) and (6.95,1.4) .. (10.93,3.29)   ;
\draw  [color={rgb, 255:red, 233; green, 134; blue, 141 }  ,draw opacity=0.2 ][fill={rgb, 255:red, 208; green, 2; blue, 27 }  ,fill opacity=0.24 ] (498.06,2245.37) -- (337.85,2326.66) -- (573.57,2093.16) -- cycle ;
\draw  [color={rgb, 255:red, 254; green, 19; blue, 23 }  ,draw opacity=0.15 ][pattern=_7c8epya2v,pattern size=6pt,pattern thickness=2.25pt,pattern radius=0pt, pattern color={rgb, 255:red, 144; green, 19; blue, 254}] (578.4,2084) -- (531.7,2180.43) -- (433.92,2279.28) -- (341.78,2323.2) -- cycle ;
\draw [line width=2.25]    (342.49,2324.06) -- (577.88,2086.29) ;
\draw  [fill={rgb, 255:red, 0; green, 0; blue, 0 }  ,fill opacity=1 ] (576.12,2086.53) .. controls (576.12,2085.02) and (577.35,2083.79) .. (578.87,2083.79) .. controls (580.39,2083.79) and (581.62,2085.02) .. (581.62,2086.53) .. controls (581.62,2088.05) and (580.39,2089.27) .. (578.87,2089.27) .. controls (577.35,2089.27) and (576.12,2088.05) .. (576.12,2086.53) -- cycle ;
\draw  [fill={rgb, 255:red, 0; green, 0; blue, 0 }  ,fill opacity=1 ] (343.42,2326.28) .. controls (343.42,2324.34) and (341.71,2322.77) .. (339.6,2322.77) .. controls (337.49,2322.77) and (335.78,2324.34) .. (335.78,2326.28) .. controls (335.78,2328.21) and (337.49,2329.78) .. (339.6,2329.78) .. controls (341.71,2329.78) and (343.42,2328.21) .. (343.42,2326.28) -- cycle ;
\draw   (437.68,2276.73) .. controls (437.68,2274.32) and (439.47,2272.36) .. (441.69,2272.36) .. controls (443.9,2272.36) and (445.7,2274.32) .. (445.7,2276.73) .. controls (445.7,2279.15) and (443.9,2281.11) .. (441.69,2281.11) .. controls (439.47,2281.11) and (437.68,2279.15) .. (437.68,2276.73) -- cycle ;
\draw   (529.41,2181.17) .. controls (529.41,2178.88) and (531.41,2177.02) .. (533.87,2177.02) .. controls (536.34,2177.02) and (538.34,2178.88) .. (538.34,2181.17) .. controls (538.34,2183.47) and (536.34,2185.33) .. (533.87,2185.33) .. controls (531.41,2185.33) and (529.41,2183.47) .. (529.41,2181.17) -- cycle ;
\draw  [dash pattern={on 4.5pt off 4.5pt}]  (257.14,2323.58) -- (257.14,2087.8) ;
\draw  [dash pattern={on 4.5pt off 4.5pt}]  (257.14,2087.8) -- (22.6,2086.92) ;
\draw  [color={rgb, 255:red, 74; green, 144; blue, 226 }  ,draw opacity=0.29 ][pattern=_t8htwyv0t,pattern size=6pt,pattern thickness=2.25pt,pattern radius=0pt, pattern color={rgb, 255:red, 208; green, 2; blue, 27}] (87.62,2286.01) -- (25.88,2320.43) -- (105.07,2242.42) -- cycle ;
\draw  [fill={rgb, 255:red, 0; green, 0; blue, 0 }  ,fill opacity=1 ] (26.5,2323.24) .. controls (26.5,2321.11) and (24.75,2319.38) .. (22.6,2319.38) .. controls (20.45,2319.38) and (18.71,2321.11) .. (18.71,2323.24) .. controls (18.71,2325.37) and (20.45,2327.1) .. (22.6,2327.1) .. controls (24.75,2327.1) and (26.5,2325.37) .. (26.5,2323.24) -- cycle ;
\draw   (108.18,2239.34) .. controls (108.18,2237.64) and (106.79,2236.26) .. (105.07,2236.26) .. controls (103.35,2236.26) and (101.96,2237.64) .. (101.96,2239.34) .. controls (101.96,2241.04) and (103.35,2242.42) .. (105.07,2242.42) .. controls (106.79,2242.42) and (108.18,2241.04) .. (108.18,2239.34) -- cycle ;
\draw   (92.2,2286.72) .. controls (92.2,2285.02) and (90.8,2283.64) .. (89.09,2283.64) .. controls (87.37,2283.64) and (85.98,2285.02) .. (85.98,2286.72) .. controls (85.98,2288.42) and (87.37,2289.8) .. (89.09,2289.8) .. controls (90.8,2289.8) and (92.2,2288.42) .. (92.2,2286.72) -- cycle ;
\draw   (121.02,2274.17) .. controls (121.02,2272.47) and (119.63,2271.09) .. (117.91,2271.09) .. controls (116.19,2271.09) and (114.8,2272.47) .. (114.8,2274.17) .. controls (114.8,2275.87) and (116.19,2277.25) .. (117.91,2277.25) .. controls (119.63,2277.25) and (121.02,2275.87) .. (121.02,2274.17) -- cycle ;
\draw   (144.59,2203.27) .. controls (144.59,2201.57) and (143.2,2200.19) .. (141.48,2200.19) .. controls (139.77,2200.19) and (138.37,2201.57) .. (138.37,2203.27) .. controls (138.37,2204.97) and (139.77,2206.35) .. (141.48,2206.35) .. controls (143.2,2206.35) and (144.59,2204.97) .. (144.59,2203.27) -- cycle ;
\draw [color={rgb, 255:red, 155; green, 155; blue, 155 }  ,draw opacity=1 ][line width=2.25]    (23.06,2321.3) -- (141.48,2202.83) ;
\draw    (87.09,2291.8) -- (68.89,2366.39) ;
\draw [shift={(68.42,2368.33)}, rotate = 283.71] [color={rgb, 255:red, 0; green, 0; blue, 0 }  ][line width=0.75]    (10.93,-3.29) .. controls (6.95,-1.4) and (3.31,-0.3) .. (0,0) .. controls (3.31,0.3) and (6.95,1.4) .. (10.93,3.29)   ;
\draw    (121.26,2279.59) -- (154.9,2338.56) ;
\draw [shift={(155.89,2340.3)}, rotate = 240.29] [color={rgb, 255:red, 0; green, 0; blue, 0 }  ][line width=0.75]    (10.93,-3.29) .. controls (6.95,-1.4) and (3.31,-0.3) .. (0,0) .. controls (3.31,0.3) and (6.95,1.4) .. (10.93,3.29)   ;
\draw  [dash pattern={on 4.5pt off 4.5pt}]  (574.44,2322.19) -- (577.88,2086.53) ;
\draw  [dash pattern={on 4.5pt off 4.5pt}]  (577.88,2088.29) -- (341.27,2088.29) ;
\draw   (498.06,2245.37) .. controls (498.06,2242.93) and (500.09,2240.96) .. (502.6,2240.96) .. controls (505.1,2240.96) and (507.13,2242.93) .. (507.13,2245.37) .. controls (507.13,2247.8) and (505.1,2249.77) .. (502.6,2249.77) .. controls (500.09,2249.77) and (498.06,2247.8) .. (498.06,2245.37) -- cycle ;
\draw    (444.59,2283) -- (464.13,2344.1) ;
\draw [shift={(464.74,2346)}, rotate = 252.26] [color={rgb, 255:red, 0; green, 0; blue, 0 }  ][line width=0.75]    (10.93,-3.29) .. controls (6.95,-1.4) and (3.31,-0.3) .. (0,0) .. controls (3.31,0.3) and (6.95,1.4) .. (10.93,3.29)   ;
\draw    (541.5,2186) -- (588.79,2214.96) ;
\draw [shift={(590.5,2216)}, rotate = 211.48] [color={rgb, 255:red, 0; green, 0; blue, 0 }  ][line width=0.75]    (10.93,-3.29) .. controls (6.95,-1.4) and (3.31,-0.3) .. (0,0) .. controls (3.31,0.3) and (6.95,1.4) .. (10.93,3.29)   ;

\draw (148.49,2334.09) node [anchor=north west][inner sep=0.75pt]    {$\left(\frac{2}{5} ,\frac{1}{5}\right)$};
\draw (55.58,2220.33) node [anchor=north west][inner sep=0.75pt]    {$\left(\frac{1}{d} ,\frac{1}{d}\right)$};
\draw (7.75,2364.01) node [anchor=north west][inner sep=0.75pt]    {$\left(\frac{2}{2d+1} ,\frac{1}{2d+1}\right)$};
\draw (-0.17,2065.47) node [anchor=north west][inner sep=0.75pt]    {$\frac{1}{q}$};
\draw (274.78,2298.26) node [anchor=north west][inner sep=0.75pt]    {$\frac{1}{p}$};
\draw (492.94,2254) node [anchor=north west][inner sep=0.75pt]    {$\left(\frac{2}{3} ,\frac{1}{3}\right)$};
\draw (581.83,2211.84) node [anchor=north west][inner sep=0.75pt]    {$\left(\frac{d}{d+1} ,\frac{d-1}{d+1}\right)$};
\draw (416.84,2341) node [anchor=north west][inner sep=0.75pt]    {$\left(\frac{2}{d+1} ,\frac{1}{d+1}\right)$};
\draw (313.22,2070.36) node [anchor=north west][inner sep=0.75pt]    {$\frac{1}{q}$};
\draw (594.55,2300.79) node [anchor=north west][inner sep=0.75pt]    {$\frac{1}{p}$};
\draw (142.27,2184.09) node [anchor=north west][inner sep=0.75pt]    {$\left(\frac{1}{2} ,\frac{1}{2}\right)$};
\draw (257.75,2076.8) node [anchor=north west][inner sep=0.75pt]    {$( 1,1)$};
\draw (24.15,2321.92) node [anchor=north west][inner sep=0.75pt]    {$( 0,0)$};
\draw (577.8,2083.79) node [anchor=north west][inner sep=0.75pt]    {$( 1,1)$};
\draw (338.53,2329.78) node [anchor=north west][inner sep=0.75pt]    {$( 0,0)$};
\draw (55,2435) node [anchor=north west][inner sep=0.75pt]    {$\textcolor[rgb]{0,0,0}{\Delta }\textcolor[rgb]{0,0,0}{_{2}} \ for\mathbf{\ \textcolor[rgb]{0.29,0.56,0.89}{d\ =2}} \ and\ \textcolor[rgb]{0.82,0.01,0.11}{\mathbf{d\geq 3}}$};
\draw (403,2437) node [anchor=north west][inner sep=0.75pt]    {$\Delta _{3} \ for\ \textcolor[rgb]{0.8,0.36,0.49}{\mathbf{d=2}} \ and\ \textcolor[rgb]{0.56,0.07,1}{\mathbf{d\geq 3}}$};

\end{tikzpicture}

\caption{${\color{black}\Delta_2}\subsetneq {\color{black}\Delta_0}$ for $d\geq3$ and ${\color{black}\Delta_2} = {\color{black}\Delta_0}$ for $d=2$}
\end{figure}

Now we list our main results  as follows.

Let $\phi \in C^{\infty}(I,\mathbb{R})$, where $I$ is a bounded interval containing the origin, and
\begin{equation}\label{phi}
\phi(0)\neq 0; \hspace{0.1cm}\phi^{(j)}(0)= 0, \hspace{0.1cm}j=1,2,\cdots,m-1; \hspace{0.1cm}\phi^{(m)}(0)\neq 0 \hspace{0.1cm}(m\geq 1).
 \end{equation}
We first show the $L^p \rightarrow L^{q}$ estimates for  maximal functions with isotropic dilations along the curves $(x,x^{d}\phi(x) +c)$ of finite type for $d \ge 2$, $d \in \mathbb{N}^{+}$ and $c \in \mathbb{R}$.

\begin{thm}\label{maintheorem1}
Define the averaging operator
\begin{equation}
A_{t}f(y):= \int_{\mathbb{R}} f(y_1-tx,y_2-t(x^d\phi(x) +c))\eta(x)dx,
\end{equation}
where $\eta$ is supported in a sufficiently small neighborhood of the origin. Then we have the following results:

 (1) when $c =0$, for  $(\frac{1}{p},\frac{1}{q}) \in \Delta_{1} = \{(\frac{1}{p},\frac{1}{q}): \frac{1}{2p} < \frac{1}{q} \le \frac{1}{p} , \frac{1}{q}  > \frac{3}{p} -1, \frac{1}{q} > \frac{1}{p} - \frac{1}{d+1}\} \cup \{(0, 0)\}$, there exists a constant
$C_{p,q} >0$ such that $\|\mathop{sup}_{t\in [1,2]}|A_{t}|\|_{L^{p} \rightarrow L^{q} }\le C_{p,q} $;

 (2) when $c \neq 0$, for  $(\frac{1}{p},\frac{1}{q}) \in \Delta_{2} = \{(\frac{1}{p},\frac{1}{q}): \frac{1}{2p} < \frac{1}{q} \le \frac{1}{p} , \frac{1}{q}  > \frac{d+1}{p} -1\} \cup \{(0, 0)\}$, there exists a constant
$C_{p,q} >0$ such that $\|\mathop{sup}_{t\in [1,2]}|A_{t}|\|_{L^{p} \rightarrow L^{q} }\le C_{p,q} $.
\end{thm}

The regions for $p, q$ in Theorem \ref{maintheorem1} are almost sharp. In fact, according to the counterexamples given in Theorem \ref{hnnece}  below, the corresponding  $L^{p} \rightarrow L^{q}$ bound for the local maximal operator $\sup_{t\in [1,2]}|A_{t}|$  can not be finite either $c=0$ and $(\frac{1}{p},\frac{1}{q}) \notin $ closure of the set $\Delta_1$, or $c \neq 0$ and $(\frac{1}{p},\frac{1}{q}) \notin$ closure of the set $\Delta_2$.

Comparing Theorem \ref{maintheorem1} with the corresponding result for the local circular maximal operator specified by $\Delta_0$, we notice that when $c \neq 0$ and $d =2$, $\Delta_{2} = \Delta_{0}$,  but for  $d > 2$, $\Delta_{2} \subsetneqq \Delta_{0}$. When $c =0$, the curves pass through the origin, and the corresponding results (described by $\Delta_1$) behave more regularly than the case $c \neq 0$ (described by $\Delta_2$), since $\Delta_{2} \subsetneqq \Delta_{1}$ whenever $d >2$. However, we still have $\Delta_{1} \subsetneqq \Delta_{0}$ for $d \ge 5$. Therefore, vanishing curvatures lead to worse $L^p\rightarrow L^q$ regularities for the local maximal operators with  isotropic dilations.

Now we briefly explain how to prove Theorem \ref{maintheorem1}. Since the curvature of the curve $(x,x^{d}\phi(x) +c)$ might vanish at the origin, we first decompose the support of $\eta$ dyadically. For each dyadic operator, we make good use of isometric transform on $L^{p}(\mathbb{R}^{2})$ so that the curve will be away from the flat point $x =0$. Then we decompose the Fourier side of $f$ into dyadic annuli, and concentrate on each annulus $\{\xi: |\xi|\approx 2^{j}\}$, $j \gg 1$. By the stationary phase method and Sobolev embedding, we are reduced to obtain the $L^{p} \rightarrow L^{q}$ estimate for a Fourier integral operator whose phase function satisfies the so-called ``cinematic curvature" condition. For $q = p > 2$, we can apply the local smoothing estimate by Mockenhaupt-Seeger-Sogge \cite{mss}; for $(\frac{1}{p} ,\frac{1}{q}) \in \{(\frac{1}{p}, \frac{1}{q}):   \frac{1}{q} \le  \frac{3}{5p}, \frac{3}{q} \le 1-\frac{1}{p} , 0 \le \frac{1}{q} \le \frac{3}{14} \}$, the  $L^{p} \rightarrow L^{q}$  local smoothing estimate by Lee \cite{SL2} can be adopted. Finally, we obtain all the desired $L^{p} \rightarrow L^{q}$ estimate for the corresponding Fourier integral operator by interpolation.

Furthermore, the result in Theorem \ref{maintheorem1} can be generalized to the local maximal functions associated to finite type curves $(x,x^{d}\phi(x) +c)$ with non-isotropic dilation $\delta_{t}(y) = (t^{a_{1}}y_1, t^{a_{2}}y_2)$, $da_{1} \neq a_{2}$.

\begin{thm}\label{hisomaintheorem}
Define the averaging operator
\begin{equation}\label{averagefinite1}
A_{t}f(y):= \int_{\mathbb{R}} f(y_1-t^{a_{1}}x,y_2-t^{a_{2}}(x^d\phi(x) +c))\eta(x)dx, \quad da_{1} \neq a_{2},
\end{equation}
where $\eta$ is supported in a sufficiently small neighborhood of the origin. Then we have the following results:

 (1) when $c =0$, for  $(\frac{1}{p},\frac{1}{q}) \in \Delta_{1} = \{(\frac{1}{p},\frac{1}{q}): \frac{1}{2p} < \frac{1}{q} \le \frac{1}{p} , \frac{1}{q}  > \frac{3}{p} -1, \frac{1}{q} > \frac{1}{p} - \frac{1}{d+1}\} \cup \{(0, 0)\}$, there exists a constant
$C_{p,q} >0$ such that $\|\mathop{sup}_{t\in [1,2]}|A_{t}|\|_{L^{p} \rightarrow L^{q} }\le C_{p,q} $;

 (2) when $c \neq 0$, for  $(\frac{1}{p},\frac{1}{q}) \in \Delta_{2} = \{(\frac{1}{p},\frac{1}{q}): \frac{1}{2p} < \frac{1}{q} \le \frac{1}{p} , \frac{1}{q}  > \frac{d+1}{p} -1\} \cup \{(0, 0)\}$, there exists a constant
$C_{p,q} >0$ such that $\|\mathop{sup}_{t\in [1,2]}|A_{t}|\|_{L^{p} \rightarrow L^{q} }\le C_{p,q} $.
\end{thm}

We would like to emphasize that the case when $da_{1} =a_{2}$ is distinctive from the case when $da_1 \neq a_2$ as described in Theorem \ref{hisomaintheorem}. Since $t \in [1, 2]$, by change of variables, we only need to consider the dilation $\delta_{t}(y)= (ty_1, t^{d}y_2)$. The following theorem concerns the local maximal operator along the homogeneous curve $(x, x^{d})$.

\begin{thm}\label{maitheorem3}
Define the  averaging operator
\begin{equation}\label{averagefinite2}
A_{t}f(y):= \int_{0}^{1}f(y_1-tx,y_2-t^dx^d)dx,
\end{equation}
then for  $(\frac{1}{p},\frac{1}{q}) \in \Delta_{3}= \{(\frac{1}{p},\frac{1}{q}): \frac{1}{2p} < \frac{1}{q} \le \frac{1}{p} , \frac{1}{q}  > \frac{2}{p} -1, \frac{1}{q} > \frac{1}{p} - \frac{1}{d+1}\} \cup \{(0, 0), (1, 1)\}$, there exists a constant
$C_{p,q} >0$ such that $\|\mathop{sup}_{t\in [1,2]}|A_{t}||_{L^{p} \rightarrow L^{q} }\le C_{p,q} $.
\end{thm}

When $d=2$, we have $\Delta_{0} \subsetneqq \Delta_{3}$, indicating that the regularity result in Theorem \ref{maitheorem3} is better than the corresponding result for the local circular maximal operator. It is also clear that the regularity results in Theorem \ref{maitheorem3} are better than the results in Theorem \ref{maintheorem1} for all $d \ge 2$. Theorem \ref{maitheorem3}  is almost sharp.

It is also natural to consider a generalization of the homogeneous curve $(x,x^{d})$, namely  the curves $(x,x^{d}\phi(x))$ near the origin,  which  can be considered as a small perturbation of the homogeneous curve. 
We study the corresponding local maximal functions with the dilation $(t, t^{d})$ and obtain the following result.

\begin{thm}\label{hhisomaintheorem}
Define the averaging operator
\begin{equation}\label{averagefinite3}
A_{t}f(y):= \int_{\mathbb{R}} f(y_1-tx,y_2-t^{d} x^d\phi(x) )\eta(x)dx,
\end{equation}
where $\eta$ is supported in a sufficiently small neighborhood of the origin. Then for  $(\frac{1}{p},\frac{1}{q}) \in \Delta_{1} = \{(\frac{1}{p},\frac{1}{q}): \frac{1}{2p} < \frac{1}{q} \le \frac{1}{p} , \frac{1}{q}  > \frac{3}{p} -1, \frac{1}{q} > \frac{1}{p} - \frac{1}{d+1}\} \cup \{(0, 0)\}$, there exists a constant
$C_{p,q} >0$ such that $\|\mathop{sup}_{t\in [1,2]}|A_{t}|\|_{L^{p} \rightarrow L^{q} }\le C_{p,q} $.
\end{thm}

Theorem \ref{hhisomaintheorem} requests a more involving argument. In particular,
we would like to adopt the similar method used in the proof of Theorem \ref{maintheorem1}. But after employing the stationary phase method and Sobolev embedding, we need to consider a family of  Fourier integral operators which fail to satisfy the ``cinematic curvature" condition uniformly, which means that celebrated local smoothing estimates in the articles \cite{mss,SL2} cannot be directly applied to our problem. This echos with the problems which have arised in the study of the $L^{p}$-estimate for the global maximal functions defined by
 \begin{equation*}
\sup_{t >0}\left|\int_{\mathbb{R}}f(y_1-tx,y_2-t^dx^d\phi(x))\eta(x)dx\right|.
\end{equation*}
In order to overcome the above difficulty, we adopt the following strategy.

(i) We break the support of $\eta$ into dyadic intervals with length $2^{-k}$, $k \gg 1$.  Meanwhile, we decompose the frequency space of $f$ into $\{\xi \in \mathbb{R}^{2}: |\xi| \leq 1 \}$ and $\{\xi \in \mathbb{R}^{2}: |\xi| \approx 2^{j} \}$, $j \ge 1$. The difficulty lies in the case when $j$ is sufficiently large. It can be observed that for each $j$, $j \gg 1$, the principal curvature for the surfaces related to the  corresponding Fourier integral operators vanishes as $k$ tends to infinity. Next, we consider  $j \leq mk$ and $j > mk$ respectively.

(ii) For $j \leq mk$, no  local smoothing estimates can be established in this case. We only get the basic $L^{p} \rightarrow L^{q}$ estimates  for the local maximal operator. Fortunately, we realize that the basic $L^{p} \rightarrow L^{q}$ estimate  is sufficient for us to finish the proof of Theorem \ref{hhisomaintheorem} since $j$ is ``small" here.

(iii) For $j > mk$, by stationary phase method and Sobolev  embedding, we are left to consider a class of Fourier integral operators defined by inequality (\ref{estimate:L4}), in which the principal curvature of the related surfaces is not so``small", since the upper bound of $k$ is dominated by $j/m$. This phenomenon allows us to obtain $L^{p} \rightarrow L^{q}$ estimates for these Fourier integral operators in Theorem \ref{L^(p,q)therom}. In the proof of Theorem \ref{L^(p,q)therom}, we  use  Whitney type decomposition and a bilinear estimate established  by \cite{SL2}.


\subsubsection{Localized maximal functions associated with surfaces in $\mathbb{R}^{3}$}\label{subsection1.2}
Denote by $\delta_t$ the non-isotropic dilations in $\mathbb{R}^3$ given by
\begin{equation}\label{dilation}
\delta_t(y)=(t^{a_1}y_1,t^{a_2}y_2,t^{a_3}y_3).
\end{equation}
We also obtain $L^{p} \rightarrow L^{q}$ estimates for local maximal operators associated with the dilation $\delta_{t}$ of some hypersurfaces $\{(x_1,x_2,\Phi(x_1,x_2)), (x_1,x_2)\in \Omega\subset \mathbb{R}^2\}\subset \mathbb{R}^{3}$. Here,  $\Omega$ is an open neighborhood of the origin, $\Phi: \Omega \rightarrow \mathbb{R}$ is a smooth function.

 We first show $L^{p} \rightarrow L^{q}$ estimates for  maximal functions related to hypersurfaces with at least one non-vanishing principal curvature when $2a_{2} \neq a_{3}$.


\begin{thm}\label{3nonvanish}
Assume that $\Phi(x_1,x_2)\in C^{\infty}(\Omega)$  satisfies
\begin{equation}\label{conditionnonvani}
\partial_2\Phi(0,0)=0,\hspace{0.5cm}\partial_2^2\Phi(0,0)\neq 0,
\end{equation}
and  $2a_2\neq a_3$. Define the  averaging operator  by
\begin{equation}\label{maximal3dimen}
A_{t}f(y):=\int_{\mathbb{R}^2}f(y-\delta_t(x_1,x_2,\Phi(x_1,x_2)))\eta(x)dx,
\end{equation}
where $\eta$ is supported in a sufficiently small neighborhood $U\subset \Omega$ of the origin. For  $(\frac{1}{p},\frac{1}{q}) \in \Delta_{0}= \{(\frac{1}{p},\frac{1}{q}): \frac{1}{2p} < \frac{1}{q} \le \frac{1}{p} , \frac{1}{q}  > \frac{3}{p} -1\} \cup \{(0, 0)\}$, there exists a constant
$C_{p,q} >0$ such that $\|\mathop{sup}_{t\in [1,2]}|A_{t}|\|_{L^{p} \rightarrow L^{q} }\le C_{p,q} $.
\end{thm}

Notice that when $a_{1}=a_{2} = a_{3}=1$, we get the following corollary.

\begin{cor}\label{co1.3}
Assume that $\Phi(x_1,x_2)\in C^{\infty}(\Omega)$  satisfies inequality (\ref{conditionnonvani}). Define the averaging operator
\begin{equation}
A_{t}f(y):= \int_{\mathbb{R}^2}f(y-t(x_1,x_2,\Phi(x_1,x_2)))\eta(x)dx,
\end{equation}
where $\eta$ is supported in a sufficiently small neighborhood $U\subset \Omega$ of the origin. For  $(\frac{1}{p},\frac{1}{q}) \in \Delta_{0}= \{(\frac{1}{p},\frac{1}{q}): \frac{1}{2p} < \frac{1}{q} \le \frac{1}{p} , \frac{1}{q}  > \frac{3}{p} -1\} \cup \{(0, 0)\}$, there exists a constant
$C_{p,q} >0$ such that $\|\mathop{sup}_{t\in [1,2]}|A_{t}|\|_{L^{p} \rightarrow L^{q} }\le C_{p,q} $.
\end{cor}

\begin{rem}
We note that in \cite{WS2}, the ``cinematic curvature" condition in $\mathbb{R}^{3}$ is required  to  establish $L^{p} \rightarrow L^{q}$ estimates for the local maximal functions. While in Corollary \ref{co1.3}, we just need the ``cinematic curvature" condition in $\mathbb{R}^{2}$.
\end{rem}

For hypersurfaces of finite type given by $(x_{1}, x_{2}, x_{2}^{d}\Phi(x_{1}, x_{2}))$ with $\Phi(0,0)\not=0$, we obtain the $L^{p} \rightarrow L^{q}$ estimate for dilations satisfying $da_{2} \neq a_{3}$.

\begin{thm}\label{maintheorem4}
Assume that $\Phi(x_1,x_2)\in C^{\infty}(\Omega)$ satisfies
$\Phi(0,0)\neq 0$ and $da_2\neq a_3$ in (\ref{dilation}), $d\geq 2$. The associated averaging operator is defined by
\begin{equation}\label{averagefinite4}
A_{t}f(y):= \int_{\mathbb{R}^2}f(y-\delta_t(x_1,x_2,c+x_2^d\Phi(x_1,x_2)))\eta(x)dx,
\end{equation}
where $\eta$ is supported in a sufficiently small neighborhood $U$ of the origin.
Then we have the following results:

 (1) when $c =0$, for  $(\frac{1}{p},\frac{1}{q}) \in \Delta_{1} = \{(\frac{1}{p},\frac{1}{q}): \frac{1}{2p} < \frac{1}{q} \le \frac{1}{p} , \frac{1}{q}  > \frac{3}{p} -1, \frac{1}{q} > \frac{1}{p} - \frac{1}{d+1}\} \cup \{(0, 0)\}$, there exists a constant
$C_{p,q} >0$ such that $\|\mathop{sup}_{t\in [1,2]}|A_{t}|\|_{L^{p} \rightarrow L^{q} }\le C_{p,q} $;

 (2) when $c \neq 0$, for  $(\frac{1}{p},\frac{1}{q}) \in \Delta_{2} = \{(\frac{1}{p},\frac{1}{q}): \frac{1}{2p} < \frac{1}{q} \le \frac{1}{p} , \frac{1}{q}  > \frac{d+1}{p} -1\} \cup \{(0, 0)\}$, there exists a constant
$C_{p,q} >0$ such that $\|\mathop{sup}_{t\in [1,2]}|A_{t}|\|_{L^{p} \rightarrow L^{q} }\le C_{p,q} $.
\end{thm}

For dilations with $da_{2} = a_{3}$, the problem is much more difficult and we get the following result.

\begin{thm}\label{theocurvanishi}
Let $\phi\in C^{\infty}(I)$, where $I$ is a bounded interval containing the origin.
Define the averaging operator by
\begin{equation}\label{averagefinite5}
A_{t}f(y):=\int_{\mathbb{R}^2}f(y-\delta_t(x_1,x_2,x_2^{d}\phi(x_d)))\eta(x)dx,
\end{equation}
where $\eta$ is supported in a sufficiently small neighborhood $U$ of the origin.
Assume that $\phi$ satisfies (\ref{phi}),
and $da_2=a_3$ in \eqref{dilation}.  Then for  $(\frac{1}{p},\frac{1}{q}) \in \Delta_{1} = \{(\frac{1}{p},\frac{1}{q}): \frac{1}{2p} < \frac{1}{q} \le \frac{1}{p} , \frac{1}{q}  > \frac{3}{p} -1, \frac{1}{q} > \frac{1}{p} - \frac{1}{d+1}\} \cup \{(0, 0)\}$, there exists a constant
$C_{p,q} >0$ such that $\|\mathop{sup}_{t\in [1,2]}|A_{t}|\|_{L^{p} \rightarrow L^{q} }\le C_{p,q} $.
\end{thm}

We briefly sketch the idea of the proof here. We follow the strategy adopted in \cite{WL}. First we "freeze" the first variable $x_1$ and apply the method of stationary phase to curves in $(x_2, x_3)-$ plane, then by Sobolev embedding, we are reduced to prove the $L^p \rightarrow L^{q}$ estimates for certain Fourier integral operators. Finally the method we used in $\mathbb{R}^{2}$ can be applied here to complete the proof.

\subsubsection{Continuity Lemmas} \label{sec:con:lemma}
Using the above $L^{p} \rightarrow L^{q}$ estimates, we can prove the following continuity lemmas.
\begin{thm}\label{continuty lemma}
There exists a real number $\epsilon >0$ such that for any $z \in \mathbb{R}^{n}$, there holds
\begin{equation}
\biggl\|\sup_{t\in [1,2]} \biggl| A_{t}f(y+z) - A_{t}f(y)  \biggl| \biggl\|_{L^{q}(\mathbb{R}^{n})} \lesssim |z|^{\epsilon} \|f\|_{L^{p}(\mathbb{R}^{n})},
\end{equation} 
provided that

(1) $n=2$, $A_{t}$ is defined by (\ref{averagefinite1}) with $c =0$ or by (\ref{averagefinite3}), and  $(\frac{1}{p},\frac{1}{q}) \in \Delta_{1} \backslash \{(0,0)\} = \{(\frac{1}{p},\frac{1}{q}): \frac{1}{2p} < \frac{1}{q} \le \frac{1}{p} , \frac{1}{q}  > \frac{3}{p} -1, \frac{1}{q} > \frac{1}{p} - \frac{1}{d+1}\}$;

(2) $n=3$, $A_{t}$ is defined by: (\ref{averagefinite4}) with $c=0$ or by (\ref{averagefinite5}), and  $(\frac{1}{p},\frac{1}{q}) \in \Delta_{1} \backslash \{(0,0)\} = \{(\frac{1}{p},\frac{1}{q}): \frac{1}{2p} < \frac{1}{q} \le \frac{1}{p} , \frac{1}{q}  > \frac{3}{p} -1, \frac{1}{q} > \frac{1}{p} - \frac{1}{d+1}\}$;

(3) $n=2$, $A_{t}$ is defined by  (\ref{averagefinite1}) with $c \neq 0$; $n=3$, $A_{t}$ is defined by    (\ref{averagefinite4}) with $c \neq 0$, and   $(\frac{1}{p},\frac{1}{q}) \in \Delta_{2} \backslash \{(0,0)\} =  \{(\frac{1}{p},\frac{1}{q}): \frac{1}{2p} < \frac{1}{q} \le \frac{1}{p} , \frac{1}{q}  > \frac{d+1}{p} -1\} $;

(4) $n=2$, $A_{t}$ is defined by (\ref{averagefinite2}), $(\frac{1}{p},\frac{1}{q}) \in \Delta_{3}\backslash \{(0,0), (1, 1)\} = \{(\frac{1}{p},\frac{1}{q}): \frac{1}{2p} < \frac{1}{q} \le \frac{1}{p} , \frac{1}{q}  > \frac{2}{p} -1, \frac{1}{q} > \frac{1}{p} - \frac{1}{d+1}\} $;

(5) $n=3$,  $A_{t}$ is defined by (\ref{maximal3dimen}), $(\frac{1}{p},\frac{1}{q}) \in \Delta_{0} \backslash \{(0,0)\} =  \{(\frac{1}{p},\frac{1}{q}): \frac{1}{2p} < \frac{1}{q} \le \frac{1}{p} , \frac{1}{q}  > \frac{3}{p} -1\}$.
\end{thm}

We concisely explain how to prove Theorem \ref{continuty lemma}. Notice that the result is trivial when $z =0$, and will follow from the $L^{p} \rightarrow L^{q}$ estimate when $|z| \ge 1$ immediately. So we always assume that $0< |z| < 1$ in what follows.  We need the following result from  [Theorem 2, page 351, \cite{steinbook}]:

Let $S$ be an open subset of a smooth $m$-dimensional submamanifold of $\mathbb{R}^{n}$, and  denote by $d\sigma$ the measure on $S$ induced by the Lebesgue measure on $\mathbb{R}^{n}$.  Taking a function $\psi \in C_{0}^{\infty}(\mathbb{R}^{n})$ whose support intersect $S$ in a compact subset of $S$, we write the finite Borel measure $d\mu = \psi(x) d\sigma$,  if $S$ is of finite type of order $d$ inside the support of $\psi$, then
\begin{equation}\label{fourierdecayoffinite}
|\widehat{d\mu}(\xi)| \lesssim |\xi|^{-\frac{1}{d}},\hspace{0.5cm}\xi\in \mathbb{R}^n\setminus \{0\}.
\end{equation}

According to the result above, for each $z \in \mathbb{R}^{n}$ and $f \in C_{0}^{\infty}(\mathbb{R}^{n})$, we have
\begin{equation}\label{l2continuty}
\|f * d\mu (y+z) -f * d\mu (y)  \|_{L^{2}(\mathbb{R}^{n})}  \le |z|^{\frac{1}{d}} \|f\|_{L^{2}(\mathbb{R}^{n})}.
\end{equation}
Indeed, by Plancherel's theorem, mean value theorem  and inequality (\ref{fourierdecayoffinite}),
\begin{align}
&\|f * d\mu (y+z) -f * d\mu (y)  \|_{L^{2}(\mathbb{R}^{n})} \nonumber\\
 &= \biggl\|   (e^{iz \cdot \xi} - 1) \widehat{d\mu}(\xi) \hat{f }(\xi)\biggl\|_{L^{2}(\mathbb{R}^{n})} \lesssim \biggl\|   (e^{iz \cdot \xi} - 1) (1 + |\xi|)^{-\frac{1}{d}} \hat{f }(\xi)\biggl\|_{L^{2}(\mathbb{R}^{n})} \nonumber\\
 &\lesssim \biggl\| |z| |\xi| (1 + |\xi|)^{-\frac{1}{d}} \hat{f }(\xi)\biggl\|_{L^{2}(\{\xi \in \mathbb{R}^{n}: |\xi| \le |z|^{-1}\} )} + \biggl\|   (1 + |\xi|)^{-\frac{1}{d}} \hat{f }(\xi)\biggl\|_{L^{2}(\{\xi \in \mathbb{R}^{n}: |\xi| > |z|^{-1}\} )} \nonumber\\
 &\lesssim |z|^{\frac{1}{d}}  \|f\|_{L^{2}(\mathbb{R}^{n})}.
\end{align}
Then we arrive at inequality (\ref{l2continuty}).

Notice that all the curves and hypersurfaces considered in Theorem \ref{continuty lemma} are of finite type. Therefore, for each $A_{t}$  in Theorem \ref{continuty lemma} and $t \in [1,2]$,
\begin{equation}
 \|   A_{t}f(y+z) - A_{t}f(y)   \|_{L^{2}(\mathbb{R}^{n})} \lesssim |z|^{\frac{1}{d}} \|f\|_{L^{2}(\mathbb{R}^{n})},
\end{equation}
where the implied constant is independent of $t$. Now we choose a sequence $\mathcal{I} = \{t_{i}\} \subset [1, 2]$ such that $|t_{i}-t_{i+1}| \approx |z|^{\frac{1}{d}}$.  Then the number of the elements in $\mathcal{I}$ is equivalent to $|z|^{ -\frac{1}{d}}$,
and
\begin{equation}\label{interp:cont:lemma}
\biggl\|\sup_{i: t_{i} \in \mathcal{I}} | A_{t_{i}}f(y+z) - A_{t_{i}}f(y)  | \biggl\|_{L^{2}(\mathbb{R}^{n})}
\le \biggl( \sum_{i: t_{i} \in \mathcal{I}} \biggl\|   A_{t_{i}}f(y+z) - A_{t_{i}}f(y)    \biggl\|^{2}_{L^{2}(\mathbb{R}^{n})}  \biggl)^{\frac{1}{2}}
\lesssim |z|^{\frac{1}{2d}} \|f\|_{L^{2}(\mathbb{R}^{n})}.
\end{equation}
By interpolating \eqref{interp:cont:lemma} with the corresponding  $L^{p} \rightarrow L^{q}$ estimate to be proved later, we obtain the following lemma.
\begin{lem}\label{discrete continuty lemma}
Under the conditions of Theorem \ref{continuty lemma}, for $z \in \mathbb{R}^{n}$, $0< |z| < 1$, and  a discrete sequence   $\mathcal{I} = \{t_{i}\} \subset [1, 2]$ with $|t_{i}-t_{i+1}| \approx |z|^{\frac{1}{d}}$, there exists a real number $ \epsilon_1= \epsilon_1(p,q) >0$, such that
\begin{equation}
\biggl\|\sup_{i: t_{i} \in \mathcal{I}} | A_{t_{i}}f(y+z) - A_{t_{i}}f(y)  | \biggl\|_{L^{q}(\mathbb{R}^{n})} \lesssim |z|^{\epsilon_1} \|f\|_{L^{p}(\mathbb{R}^{n})}.
\end{equation}
\end{lem}

We will complete the proof of Theorem \ref{continuty lemma} once we have the following lemma.

\begin{lem}\label{continuous continuty lemma}
Under the conditions of Theorem \ref{continuty lemma}, for $z \in \mathbb{R}^{n}$, $0< |z| < 1$,  there exists a real number $ \epsilon_2= \epsilon_2(p,q) >0$ such that
\begin{equation}
\biggl\|\sup_{t,s\in[1,2]: |t-s| \le |z|^{\frac{1}{d}}} | A_{t}f(y) - A_{s}f(y) | \biggl\|_{L^{q}(\mathbb{R}^{n})} \lesssim |z|^{\epsilon_2} \|f\|_{L^{p}(\mathbb{R}^{n})}.
\end{equation}
\end{lem}

More concretely, if Lemma \ref{continuous continuty lemma} holds true, then for $A_{t}$ and $p, q$ as in Theorem \ref{continuty lemma}, there holds
\begin{align}
&\biggl\|\sup_{t\in [1,2]} | A_{t}f(y+z) - A_{t}f(y)  | \biggl\|_{L^{q}(\mathbb{R}^{n})} \nonumber\\
&\lesssim \biggl\|\sup_{i: t_{i} \in \mathcal{I}} \sup_{t\in [t_{i},t_{i} + |z|^{\frac{1}{d}}]} | A_{t}f(y+z) - A_{t}f(y)  | \biggl\|_{L^{q}(\mathbb{R}^{n})} \nonumber\\
&\lesssim \biggl\|\sup_{i: t_{i} \in \mathcal{I}} \sup_{t\in [t_{i},t_{i} + |z|^{\frac{1}{d}}]} | A_{t}f(y+z) - A_{t_{i}}f(y+z) | \biggl\|_{L^{q}(\mathbb{R}^{n})}  + \biggl\|\sup_{i: t_{i} \in \mathcal{I}} \sup_{t\in [t_{i},t_{i} + |z|^{\frac{1}{d}}]}| A_{t}f(y) - A_{t_{i}}f(y) | \biggl\|_{L^{q}(\mathbb{R}^{n})} \nonumber\\
&\quad + \biggl\|\sup_{i: t_{i} \in \mathcal{I}}  | A_{t_{i}}f(y+z) - A_{t_{i}}f(y)  | \biggl\|_{L^{q}(\mathbb{R}^{n})} \nonumber\\
&\lesssim \biggl\|\sup_{t,s\in[1,2]: |t-s| \le |z|^{\frac{1}{d}}} | A_{t}f(y) - A_{s}f(y) | \biggl\|_{L^{q}(\mathbb{R}^{n})}  + \biggl\|\sup_{i: t_{i} \in \mathcal{I}}   | A_{t_{i}}f(y+z) - A_{t_{i}}f(y)  | \biggl\|_{L^{q}(\mathbb{R}^{n})} \nonumber\\
&\lesssim |z|^{\min\{\epsilon_1, \epsilon_2\}} \|f\|_{L^{p}(\mathbb{R}^{n})}.
\end{align}
Therefore, it remains to prove Lemma \ref{continuous continuty lemma}.

\subsection{Weighted inequalities for global maximal functions associated with curvature} \label{weighted}

As highlighted in the introduction, one important application of the sparse domination is the weighted estimate for the corresponding maximal operator. We will give the definition of the weights associated to $\delta$-cubes.
\begin{dfn}
\begin{enumerate}
\item
A weight $\omega$ is a poistive function defined on $\mathbb{R}^n$ equipped with the Lebesgue measure and the metric defined by
\begin{equation*}
\rho_\delta(x,y) := \max_{1 \leq i \leq d}|x_i - y_i|^{\frac{1}{b_i}},
\end{equation*}
where $\delta$ is the dilation specified in \eqref{dilation:restrict}. We usually denote by $\omega(E) := \int_E w(x) dx$ and $\|f\|_{L^p_\omega} := \left(\int |f(x)|^p w(x) dx\right)^{\frac{1}{p}}$.
\item
A weight $w \in \mathcal{A}_p$ (Muckenhoupt class) for $1 < p < \infty$ if
\begin{equation*}
[\omega]_{ \mathcal{A}_p} := \sup_{Q \in \mathcal{Q}^\delta}\langle \omega \rangle_Q \langle \omega^{1-p'} \rangle_Q^{p-1} < \infty.
\end{equation*}
\item
A weight $w \in RH_p$ (reverse H\"older class) for $1 < p < \infty$ if
\begin{equation*}
[\omega]_{ RH_p} := \sup_{Q \in \mathcal{Q}^\delta}\langle \omega \rangle_Q^{-1} \langle \omega \rangle_{Q,p} < \infty.
\end{equation*}
\end{enumerate}
\end{dfn}

The weighted bound can be obtained for the sparse form as summarized in the following proposition. The proof of the proposition is enclosed in \cite{Bernicot}.

\begin{prop} \label{sparse:weight}
Suppose that $\mathcal{S}$ is a sparse collection of $\delta$-cubes. For any $p < r< q$ with $(\frac{1}{p},\frac{1}{q}) \in \mathcal{L}_n$ and weight $\omega \in A_{\frac{r}{p}} \cap RH_{\left(\frac{q}{r}\right)'}$,
\begin{equation*}
\Lambda_{\mathcal{S},p,q'}(f,g) \lesssim \left([\omega]_{A_{\frac{r}{p}}} [\omega]_{RH_{\left(\frac{q}{r}\right)'}}\right)^\alpha \|f\|_{L^{r}(\omega)}\|g\|_{L^{r'}(\omega^{1-r'})},
\end{equation*}
for $\alpha$ specified in \eqref{weight:exp}.
\end{prop}

The following weighted bound is hence an immediate consequence of the sparse bound stated in Theorem \ref{thm:sparse} and Proposition \ref{sparse:weight}.

\begin{cor} \label{weight}
Suppose that the maximal operator $\mathcal{M}_{\vec{\Phi}}$ satisfies the sparse bound described in Theorem \ref{thm:sparse}. Then for any $p < r < q$ with $(\frac{1}{p},\frac{1}{q}) \in \mathcal{L}_n$ and weight $\omega \in A_{\frac{r}{p}} \cap RH_{\left(\frac{q}{r}\right)'}$ defined on $\delta$-cubes,
\begin{equation} \label{weight:bound}
\|M_{\vec{\Phi}}\|_{L^r(\omega) \rightarrow L^r(\omega)} \lesssim \left([\omega]_{A_{\frac{r}{p}}} [\omega]_{RH_{\left(\frac{q}{r}\right)'}}\right)^\alpha,
\end{equation}
for
\begin{equation} \label{weight:exp}
\alpha := \max \left(\frac{1}{r-p}, \frac{q-1}{q-r} \right).
\end{equation}
\end{cor}


Provided with the generic scheme of reducing the weighted estimates to the local regularity property, we can now summarize the main results on the weighted inequalities for the maximal operators satisfying the local continuity property stated in Theorem \ref{continuty lemma}:
\begin{thm}\label{weighted estimate}
Let $\omega$ be a weight defined on $\delta$ cubes such that  $\omega \in A_{\frac{r}{p}} \cap RH_{\left(\frac{q}{r}\right)'}$. Define the global maximal operator
\[\mathcal{M}f(y):= \sup_{t>0} |A_{t}f(y)|.\]
\begin{enumerate}
\item
When $A_{t}$ is defined in terms of  (\ref{averagefinite1}) with $c =0$ or by (\ref{averagefinite3}), then
for $\mathcal{M}$ and any $p < r < q$ with $(\frac{1}{p},\frac{1}{q}) \in \Delta_1$, \eqref{weight:bound} holds true.
\item
When $A_{t}$ is  defined in terms of  (\ref{averagefinite4}) with $c=0$ or by (\ref{averagefinite5}), then
for $\mathcal{M}$ and any $p < r < q$ with $(\frac{1}{p},\frac{1}{q}) \in \Delta_1$, \eqref{weight:bound} holds true.
\item
When $A_{t}$ is defined in terms of  (\ref{averagefinite1}) with $c \neq 0$ or by (\ref{averagefinite4}) with $c \neq 0$, then
for $\mathcal{M}$ and any $p < r < q$ with $(\frac{1}{p},\frac{1}{q}) \in \Delta_2$, \eqref{weight:bound} holds true.
\item
When $A_{t}$ is defined in terms of (\ref{averagefinite2}), then
for $\mathcal{M}$ and any $p < r < q$ with $(\frac{1}{p},\frac{1}{q}) \in \Delta_3$, \eqref{weight:bound} holds true.
\item
When $A_{t}$ is defined in terms of (\ref{maximal3dimen}), then
for $\mathcal{M}$ and any $p < r < q$ with $(\frac{1}{p},\frac{1}{q}) \in \Delta_0$, \eqref{weight:bound} holds true.
\end{enumerate}
\end{thm}






We notice that when $\omega \equiv 1 $, the results in Theorem \ref{weighted estimate} are consistent with the known unweighted estimates. For the reader's convenience, we enumerate some known unweighted estimates here.  When $c \neq 0$ and $a_{1} = a_{2}$, it can be deduced from \cite{I} that the global maximal operator with $A_{t}$ defined in (\ref{averagefinite1}) is $L^{p}$-bounded if and only if $p>d$. The result in \cite{stein} yields that the global maximal operator with $A_{t}$ defined in (\ref{averagefinite2}) is $L^{p}$-bounded if $p >1$. It follows from [Theorem 1.4, \cite{WL}] that the global maximal operator  with $A_{t}$ defined in (\ref{averagefinite3}) is $L^{p}$-bounded if $p >2$.  For the global maximal operators with $A_{t}$ defined in (\ref{maximal3dimen}), (\ref{averagefinite4}) and (\ref{averagefinite5}) respectively, the corresponding $L^{p}$-bounds were studied in [Theorem 1.6-Theorem 1.10, \cite{WL}].

\vskip .2in
\noindent
\textbf{Acknowledgements.} We would like to express our gratitude to Joris Roos for his insightful and helpful suggestions and comments. 
\section{Proof of  sparse domination bounds}
We will first sketch the proof of Theorem  \ref{thm:sparse} and then we will derive the weighted estimates as a direct corollary. The investigation in this section follows from similar arguments developed in Lacey \cite{Lacey} and Cladek-Ou \cite{cou}. \subsection{Discretization}
The proof of the theorem relies on a stopping-time argument and employs the nested property of the geometrical objects where the average is taken. 
It is therefore natural to define the dyadic $\delta$-cubes which possess the nested property as the standard dyadic cubes:
\begin{dfn}
Let $\vec{s}:= (s_1,\ldots, s_n) \in \{0, \frac{1}{3}, \frac{2}{3}\}^n$. Let $\mathbb{D}^{\delta}_s$ denote the shifted dyadic $\delta$-grid defined as
\begin{equation*}
\mathbb{D}^\delta_s := \{2^{\lceil{kb_1\rceil}}\left(i_1+\frac{s_1}{3}, i_1+1+\frac{s_1}{3} \right) \times \cdots \times 2^{\lceil{kb_n\rceil}}\left(i_n+\frac{s_n}{3}, i_n+1+\frac{s_n}{3} \right): k \in \mathbb{Z}, i_1,\ldots, i_n \in \mathbb{Z}  \}.
\end{equation*}
\end{dfn}
We would like to remark that the shifted dyadic $\delta$-grid satisfies the same nested property as the standard shifted dyadic grid.

We will now start discretizing the maximal function of interest. Suppose that $Q \in \mathbb{D}_s^\delta$ is a dyadic $\delta$-cube with $l_j(Q) = 2^{\lceil kb_j \rceil} (1 \leq j \leq n)$ for some $k \in \mathbb{Z}$ and let $cQ$ denote the generalized cube with the same center as $Q$ with the side-length $c\cdot l_j(Q)$ for $1 \leq j \leq n $. Define
\begin{equation} \label{max:loc}
\tilde{\mathcal{M}}_Q f:= \sup_{2^{k-m-1} \leq t \leq 2^{k-m}}A^{\vec{\Phi}}_{t} (f\chi_{\frac{1}{3}Q}),
\end{equation}
where  $m:= m(b_1, \ldots, b_n)$ is chosen so that $\tilde{\mathcal{M}}_Q f$ is supported on $Q$.

To relate the maximal function with $\tilde{\mathcal{M}}_Q f$, we first realize that
\begin{equation} \label{pre:discretize}
\mathcal{M}_{\vec{\Phi}}  f(y) := \sup_{t >0} A^{\vec{\Phi}}_t f (y) = \sup_{k \in \mathbb{Z}} \sup_{2^{k-m-1} \leq t \leq 2^{k-m}} A^{\vec{\Phi}}_{t} f(y).
\end{equation}
For any fixed $y \in \mathbb{R}^n$ and $k \in \mathbb{Z}$, there exists $Q' \in Q^{\delta}$ satisfying
\begin{enumerate}[(I)]
\item \label{prop:1}
 $l_j(Q')= 2^{\lceil k b_j \rceil -2}$ $1 \leq j \leq n$;
\item \label{prop:2}
for any $x' \in \supp(\eta)$ and any $2^{k-m-1} \leq t \leq 2^{k-m}$,
$$
y- \delta_t(x',\Phi(x')) \in Q'.
$$
\end{enumerate}
By the one-third trick, given $Q'$, there exists a dyadic $\delta$-cube $Q \in \mathbb{D}^\delta_s$ for some $s \in \{0, \frac{1}{3}, \frac{2}{3} \}^n$ such that $Q' \subseteq Q$ and $l_j(Q) \sim l_j(Q')$ for $1 \leq j \leq n$. One can thus find $Q \in \mathbb{D}^\delta_s$ satisfying the conditions \eqref{prop:1} (modulo constant) and \eqref{prop:2}. As a consequence,
$$
 \sup_{2^{k-m-1} \leq t \leq 2^{k-m}} A^{\vec{\Phi}}_{t} f(y) =  \sup_{2^{k-m-1} \leq t \leq 2^{k-m}} A^{\vec{\Phi}}_{t} \left(f \chi_Q \right)(y),
$$
for some $Q \in \mathbb{D}^\delta_s$ depending on $y$.

Next we would like to point out the covering property of a dyadic $\delta$-cube by smaller (shifted) dyadic $\delta$-cubes. We first introduce the following notation: for any interval $I \in \mathbb{R}$, define the translated copy of $I$ by
\begin{align*}
I^{+} := & \text{ translated copy of } I \text{ by shifting } I \text{ to the right with distance } |I|, \\
I^{-} := & \text{ translated copy of } I \text{ by shifting } I \text{ to the left with distance } |I|, \\
I^0 := & I.
\end{align*}
For any $Q = Q_1\times \ldots \times Q_n \subseteq \mathbb{R}^n$, define for $\gamma_1, \ldots, \gamma_n \in \{0, +, -\}$,
$$
Q^{\gamma_1,\ldots, \gamma_n} := Q_1^{\gamma_1} \times \ldots \times Q_n^{\gamma_n}.
$$
We notice the following observations which imply the covering property by dyadic $\delta$-cubes. In particular, for any $Q \in \mathbb{D}^\delta_{s_0}$ with $s_0 \in \{0, \frac{1}{3}, \frac{2}{3} \}^n$,
\begin{enumerate}
\item
$Q = \displaystyle \bigcup_{\gamma_1, \ldots \gamma_n \in \{0, +, -\}} \left(\frac{1}{3}Q\right)^{\gamma_1, \ldots, \gamma_n}$;
\item
for any fixed $(\gamma_1, \ldots, \gamma_n) \in \{0, +, - \}^n$, $\left(\frac{1}{3}Q\right)^{\gamma_1, \ldots, \gamma_n} \in \mathcal{D}^\delta_s$ for some $s \in \{0, \frac{1}{3}, \frac{2}{3}\}^n$.
\end{enumerate}
By applying the above observations, one can rewrite
\begin{align*}
&\sup_{2^{k-m-1} \leq t \leq 2^{k-m}} A^{\vec{\Phi}}_{t}(f \chi_Q) (y) \leq \sum_{\gamma_1, \ldots, \gamma_n \in  \{0, +, - \}} \sup_{2^{k-m-1} \leq t \leq 2^{k-m}} A^{\vec{\Phi}}_{t}(f \chi_{\left(\frac{1}{3}Q\right)^{\gamma_1,\ldots, \gamma_n}} )(y) \\
\leq & \sum_{s \in \{0, \frac{1}{3}, \frac{2}{3}\}^n}\sum_{\substack{Q \in \mathbb{D}_s^\delta \\ l_1(Q) = 2^{\lceil k b_1 \rceil}}} \sup_{2^{k-m-1} \leq t \leq 2^{k-m}} A^{\vec{\Phi}}_{t}(f \chi_{\frac{1}{3}Q} )(y) = \sum_{s \in \{0, \frac{1}{3}, \frac{2}{3}\}^n}\sum_{\substack{Q \in \mathbb{D}_s^\delta \\ l_1(Q) = 2^{\lceil k b_1 \rceil}}} \tilde{\mathcal{M}}_Q f(y),
\end{align*}
where $ \tilde{\mathcal{M}}_Q f$ is defined in \eqref{max:loc}. By recalling the identity \eqref{pre:discretize}, we deduce that
\begin{equation*}
\mathcal{M}_{\vec{\Phi}} f(y) \leq  \sup_{k \in \mathbb{Z}} \sum_{s \in \{0, \frac{1}{3}, \frac{2}{3}\}^n}\sum_{\substack{Q \in \mathbb{D}_s^\delta \\ l_1(Q) = 2^{\lceil k b_1 \rceil}}} \tilde{\mathcal{M}}_Q f(y) \leq \sum_{s \in \{0, \frac{1}{3}, \frac{2}{3}\}^n} \sup_{k \in \mathbb{Z}}\sup_{\substack{Q \in \mathbb{D}_s^\delta \\ l_1(Q) = 2^{\lceil k b_1 \rceil}}} \tilde{\mathcal{M}}_Q f(y),
\end{equation*}
where the last inequality follows from the fact that for any fixed $y$, there exists a unique $Q \in \mathbb{D}_s^\delta$ such that $y \in Q$ and $ l_1(Q) = 2^{\lceil k b_1 \rceil}$.

It is therefore sufficient to prove that the sparse bound holds for
\begin{equation} \label{max:discrete}
\tilde{\mathcal{M}}_{\mathbb{D}_s^\delta}f := \sup_{Q \in \mathbb{D}_s^\delta} \tilde{\mathcal{M}}_Q f
\end{equation}
with any choice of dyadic $\delta$-grid $s \in \{0, \frac{1}{3}, \frac{2}{3} \}^n$. Since such choice does not affect the argument, we will proceed assuming that the dyadic $\delta$-grid is fixed and denote it by $\mathbb{D}^\delta$.

\subsection{Induction}

\subsubsection{Stopping-time algorithm}
We will illustrate the algorithm to select a sparse collection of dyadic $\delta$-cubes which is an iterative procedure.

One simple observation is that to show the sparse bound for \eqref{max:discrete}, it remains to prove the same sparse bound for the operator \eqref{max:discrete} with the supreme taken over an arbitrary finite sub-collection such that the sparse bound is independent of the cardinality of the sub-collection. As a result, we can restrict out focus on a finite sub-collection of dyadic  $\delta$-cubes, denoted by $\mathscr{Q}$, which implies that all the dyadic $\delta$-cubes in $\mathscr{Q}$ are supported on some $Q_0 \in \mathbb{D}^\delta$.

Apriori, the sparse collection denoted by $\mathcal{S}$ is empty. We will start by adding $Q_0$ into the collection $\mathcal{S}$. Next, we denote by $\mathscr{Q}^S_{Q_0}$ a sub-collection of $ \mathscr{Q}$ such that for any $T \in \mathscr{Q}^S_{Q_0}$,
\begin{enumerate}
\item
 $T$ is maximal with respect to containment;
 \item
 $T$ satisfies that
\begin{equation*}
\langle f  \rangle_{T,p}  > C \langle f  \rangle_{Q_0,p} \text{ or } \langle g  \rangle_{T,q'}  > C \langle g  \rangle_{Q_0,q'},
\end{equation*}
for some constant $C \gg 1$.
\end{enumerate}

Now we add all the dyadic $\delta$-cubes in $\mathscr{Q}^S_{Q_0}$ to the collection $\mathcal{S}$. One needs to verity the sparseness of the modified collection:
\begin{enumerate}
\item
for any $T, T' \in \mathscr{Q}^S_{Q_0}$, $T \cap T' = \emptyset$ by the maximality and the nested property of the dyadic $\delta$-grid;
\item
by the disjointness of $\mathscr{Q}^S_{Q_0}$, one has
\begin{equation*}
\big|\bigcup_{T \in \mathscr{Q}^S_{Q_0}} T \cap Q_0 \big| = \sum_{T \in \mathscr{Q}^S_{Q_0}}|T \cap Q_0| \leq \frac{1}{4}|Q_0|,
\end{equation*}
for $C$ sufficiently large. The inequality follows from the stopping condition: suppose that $T$ satisfies $\langle f  \rangle_{T,p}  > C \langle f  \rangle_{Q_0,p}$ (and the same reasoning can be applied to those satisfying the second stopping condition), then
\begin{equation*}
C^p \langle f^p \rangle_{Q_0} \sum_{T \in \mathscr{Q}^S_{Q_0}} |T| < \sum_{T \in \mathscr{Q}^S_{Q_0}} |T| \langle f^p \rangle_{T} \leq \langle f^p \rangle_{Q_0}|Q_0|.
\end{equation*}
\end{enumerate}

We will iterate our algorithm to each $T \in \mathscr{Q}^S_{Q_0}$ to choose dyadic $\delta$-cubes that we will successively add in our collection $\mathcal{S}$. Before we proceed with this iterative procedure, we will elaborate on the analytic implication from each step of the stopping-time algorithm, which can be viewed as an inductive step.

In particular, we define the "good" cubes with respect to the averages over $Q_0$ as
\begin{equation} \label{def:good}
\mathscr{Q}^{good}_{Q_0} := \{Q \in \mathscr{Q}: Q \subseteq Q_0, Q \cap T^c \neq \emptyset \text{ for some } T \in \mathscr{Q}_{Q_0}^S \}.
\end{equation}

\begin{rem} \label{key:prop}
It is not difficult to observe that if $Q \in \mathscr{Q}_{Q_0}^{good}$ and $Q \subseteq Q'$, then $Q' \in \mathscr{Q}_{Q_0}^{good}$.
Moreover, for any $Q \in \mathscr{Q}_{Q_0}^{good}$,
 $$\langle f \rangle_{Q,p} \leq C \langle f \rangle_{Q_0,p} \text{ and } \langle g \rangle_{Q,q'} \leq C \langle g \rangle_{Q_0,q'}$$
 since otherwise $Q \subseteq T$ for some $T \in \mathscr{Q}_{Q_0}^S$, which contradicts the definition \eqref{def:good}.
\end{rem}

Remark \ref{key:prop} infers the key property of the sub-collection $\mathscr{Q}_{Q_0}^{good}$:
\begin{equation} \label{prop:lem}
\sup_{Q \in \mathscr{Q}_{Q_0}^{good}} \sup_{Q': Q \subseteq Q' \subseteq Q_0} \frac{\langle f \rangle_{Q',p}}{\langle f \rangle_{Q_0,p}} + \frac{\langle g \rangle_{Q',q'}}{\langle g \rangle_{Q_0,q'}}   \leq 2C.
\end{equation}

When the supreme is restricted to ``good" cubes, the corresponding maximal function is well-behaved and the bi-linear form can be controlled by the sparse form involving $Q_0$, as highlighted in the following lemma.

\begin{lem} \label{inductive}
Let $(p,q)$ denote the exponents specified in Theorem \ref{thm:sparse}. Let $\mathscr{Q}^{good}_{Q_0}$ be the sub-collection of dyadic $\delta$-cubes defined in \eqref{def:good} which satisfies the property \eqref{prop:lem}. Then
\begin{equation} \label{ineq:induct}
|\langle \sup_{Q \in \mathscr{Q}^{good}_{Q_0}} \tilde{\mathcal{M}}_Q f, g \rangle |  \lesssim |Q_0| \langle f \rangle_{Q_0,p} \langle g \rangle_{Q_0,q'}.
\end{equation}
\end{lem}

We can therefore decompose the original bi-linear form as
\begin{equation} \label{induct:1}
|\langle \sup_{Q \in \mathscr{Q}} \tilde{\mathcal{M}}_Q f, g \rangle |
\leq |\langle \sup_{Q \in \mathscr{Q}^{good}_{Q_0}} \tilde{\mathcal{M}}_Q f, g \rangle | + |\langle \sup_{Q \in \mathscr{Q} \setminus \mathscr{Q}^{good}_{Q_0}} \tilde{\mathcal{M}}_Q f, g \rangle |,
\end{equation}
where the first term in \eqref{induct:1} can be estimated by \eqref{ineq:induct} in Lemma \ref{inductive}.

To estimate the second term, one observes that for $Q \in \mathscr{Q} \setminus \mathscr{Q}^{good}_{Q_0}$, $Q \cap T^c = \emptyset$ for any $T \in  \mathscr{Q}_{Q_0}^S$. By the nested property of the dyadic $\delta$-grid and the disjointness of $\mathscr{Q}_{Q_0}^S$, $Q \subseteq T$ for a unique $T \in \mathscr{Q}_{Q_0}^S$. Denote by such sub-collection localized to $T$ as $\mathscr{Q}_T$.

The same procedure -- selection of the sparse dyadic $\delta$-cubes and the inductive estimate \eqref{ineq:induct} -- can now be applied to each $T \in \mathscr{Q}^S_{Q_0}$, in which case the role of $Q_0$ is played by $T$ and $\mathscr{D}$ by $\mathscr{Q}_T$. Let $\mathscr{Q}^{good}_T \subseteq \mathscr{Q}_T$ denote the sub-collection of ``good cubes" localized on $T$ for each $T \in \mathscr{Q}^S_{Q_0}$. One can now decompose the second term in \eqref{induct:1} as
\begin{equation*}
\sum_{T \in \mathscr{Q}^S_{Q_0}}\big|\langle \sup_{Q \in \mathscr{Q}_T} \tilde{\mathcal{M}}_Q f, g \rangle \big| \leq \sum_{T \in \mathscr{Q}^S_{Q_0}} |\langle \sup_{Q \in \mathscr{Q}^{good}_{T}} \tilde{\mathcal{M}}_Q f, g \rangle | +  \sum_{T \in \mathscr{Q}^S_{Q_0}} |\langle \sup_{Q \in \mathscr{Q}_T \setminus \mathscr{Q}^{good}_{T}} \tilde{\mathcal{M}}_Q f, g \rangle |,
\end{equation*}
where one applies the inductive estimate analogous to \eqref{ineq:induct} to deduce that
\begin{equation*}
 \sum_{T \in \mathscr{Q}^S_{Q_0}} |\langle \sup_{Q \in \mathscr{Q}^{good}_{T}} \tilde{\mathcal{M}}_Q f, g \rangle | \lesssim  \sum_{T \in \mathscr{Q}^S_{Q_0}} \langle f \rangle_{T,p} \langle g \rangle_{T,q'}|T|.
\end{equation*}

Further iterations and inductions can be applied and at each step, the sparseness of the collection and the inductive estimate can be verified using the same argument sketched above. Due to the finiteness of the collection $\mathscr{D}$, the procedure ends in finitely many steps and the desired sparse bound is obtained.

\subsubsection{Proof of the inductive step - Lemma \ref{inductive}}
It remains to prove the inductive estimate \eqref{ineq:induct}. We first linearize the expression on the left hand side of \eqref{ineq:induct}. For any $Q \in \mathscr{Q}_{Q_0}^{good}$, define
\begin{equation*}
F_Q := \{x \in Q_0: \sup_{Q' \in \mathscr{Q}_{Q_0}^{good}} \tilde{\mathcal{M}}_{Q'} f(x) = \tilde{\mathcal{M}}_Qf(x)\}.
\end{equation*}
More precisely, $Q$ is the largest $\delta$-cube contained in $Q_0$ such that the supreme is attained.
By definition, $\{ F_Q\}_{Q \in \mathscr{Q}_{Q_0}^{good}}$ consist of pairwise disjoint sets which cover $Q_0$ and hence the support of $\sup_{Q \in \mathscr{Q}} \mathcal{M}_Q f$.

 The bilinear form can thus be rewritten as
\begin{equation} \label{form:linearize}
\langle \sup_{Q \in \mathscr{Q}^{good}_{Q_0}} \tilde{\mathcal{M}}_Q f, g \rangle = \sum_{Q \in \mathscr{Q}_{Q_0}^{good}} \langle \tilde{\mathcal{M}}_Q f, g \mathbbm{1}_{F_Q} \rangle.
\end{equation}

We now implement the Calder\'on-Zygmund decomposition to $f$ with the threshold $C\langle f \rangle_{Q_0,p}$ so that
\begin{equation} \label{f:CZ}
f = f_{\infty} + f_b ,
\end{equation}
where $f_{\infty}$ is the good part satisfying
\begin{equation} \label{f:good}
\|f_\infty\|_{L^\infty} \lesssim \langle f \rangle_{Q_0,p}.
\end{equation}
Meanwhile the bad part $b$ can be rewritten as
\begin{equation*}
f_b = \sum_{P \in \mathcal{B}} \big(f - \langle f \rangle_{P} \big) \mathbbm{1}_{P} =: \sum_{P \in \mathcal{B}} b_P,
\end{equation*}
where $\mathcal{B} \subseteq \mathbb{D}^\delta$ denotes a disjoint sub-collection of dyadic $\delta$-cubes such that for any $P \in \mathcal{B}$, $P$ is the maximal $\delta$-cube in $\mathscr{Q}$ asatisfying
\begin{equation*}
\langle f \rangle_{P,p} > C\langle f \rangle_{Q_0,p}.
\end{equation*}
Then one notices that $P \subseteq T$ for some unique $T \in \mathscr{D}^S_{Q_0}$.

One important fact imposed from the Calder\'on-Zygmund decomposition is that $b_P$ for $P \in \mathcal{B}$ has the mean-zero property:
\begin{equation} \label{mean:0}
\int b_P =0.
\end{equation}

By applying the decomposition \eqref{f:CZ} to the bilinear form, one obtains
\begin{equation*}
|\eqref{form:linearize}| \leq \sum_{Q \in \mathscr{Q}_{Q_0}^{good}}\big|\langle \tilde{\mathcal{M}}_Q f_\infty, g \mathbbm{1}_{F_Q} \rangle \big| + \bigg| \sum_{Q \in \mathscr{Q}_{Q_0}^{good}} \langle \tilde{\mathcal{M}}_Q f_b, g \mathbbm{1}_{F_Q} \rangle \bigg| =: I + II.
\end{equation*}
\noindent
\textbf{Estimate of $I$.}
By the point-wise control of the good part \eqref{f:good}, one has the trivial estimate
\begin{equation*}
\sum_{Q \in \mathscr{Q}_{Q_0}^{good}} |\langle \tilde{\mathcal{M}}_Q f_\infty, g\rangle | \lesssim \langle f \rangle_{Q_0,p}\sum_{Q \in \mathscr{Q}_{Q_0}^{good}} \|g \mathbbm{1}_{F_Q}\|_{L^1} \leq \langle f \rangle_{Q_0,p} \|g \mathbbm{1}_{Q_0}\|_{L^1} \lesssim   \langle f \rangle_{Q_0,p} \langle g \rangle_{Q_0,q} |Q_0|
\end{equation*}
where the second inequality follows from the disjointness of the sets $F_Q$.
\vskip .2in

\noindent
\textbf{Estimate of $II$.}
The estimation of $II$ relies heavily on the \textit{local continuity property} as we shall see later.

We will further decompose $\mathcal{B}$ in terms of their sizes:
\begin{equation*}
\mathcal{B} = \bigcup_{k \in \mathbb{Z}} \mathcal{B}_k,
\end{equation*}
where the sub-collection $\mathcal{B}_k$ is defined by by
\begin{align*}
\mathcal{B}_k := & \{ P \in \mathcal{B}: l_j(P) = 2^{\lceil kb_j \rceil} \text{ for } 1 \leq j \leq n \}.
\end{align*}

Let $B_k:= \sum_{P \in \mathcal{B}_k} b_P$. Then $II$ can be decomposed as
\begin{align*}
\bigg| \sum_{Q \in \mathscr{Q}_{Q_0}^{good}} \langle \tilde{\mathcal{M}}_Q b, g \mathbbm{1}_{F_Q} \rangle \bigg| \leq  \sum_{Q \in \mathscr{Q}_{Q_0}^{good}}  \sum_{k}\big| \langle \tilde{\mathcal{M}}_Q B_k, g \mathbbm{1}_{F_Q}  \rangle \big|.
\end{align*}

One notices that
 for any $Q \in \mathscr{D}^{good}_{Q_0}$ fixed, $\tilde{\mathcal{M}}_Q b_P \neq 0$ if and only if $P \cap Q \neq \emptyset$ , which implies $T \cap Q \neq \emptyset$ for the appropriate $T \in \mathscr{D}^S_{Q_0}$ with $P \subseteq T$. By the nested property of the dyadic $\delta$-grid and the key property \eqref{prop:lem},
\begin{equation} \label{relation:PQ}
Q \supsetneq T \supseteq P.
\end{equation}
As a consequence, for a fixed $Q \in \mathscr{Q}_{Q_0}^{good}$ with $l_j(Q) = 2^{\lceil lb_j \rceil}$, $1 \leq j \leq n$ and any $P \in \mathcal{B}_k$,
$$\tilde{\mathcal{M}}_Q b_P \neq 0 \text{ if and only if } k < l.$$
One can thus rewrite
\begin{equation*}
\sum_{k}\big| \langle \tilde{\mathcal{M}}_Q B_{k}, g \mathbbm{1}_{F_Q}  \rangle \big| = \sum_{k < l}\langle \tilde{\mathcal{M}}_Q B_{k}, g \mathbbm{1}_{F_Q}  \rangle \big| = \sum_{k \geq 1}\langle \tilde{\mathcal{M}}_Q B_{l-k}, g \mathbbm{1}_{F_Q}  \rangle \big|.
\end{equation*}
For fixed $k \geq 1$, one derives that
\begin{align} \label{pre:mean:0}
 \big| \langle \tilde{\mathcal{M}}_Q B_{l-k}, g \mathbbm{1}_{F_Q}  \rangle \big| =  \big|\langle  \tilde{\mathcal{M}}_Q\left(B_{l-k} \mathbbm{1}_Q\right),  g \mathbbm{1}_{F_Q}  \rangle \big|= \big| \langle B_{l-k} \mathbbm{1}_Q, \tilde{\mathcal{M}}^*_Q (g \mathbbm{1}_{F_Q})  \rangle \big|,
\end{align}
where $ \tilde{\mathcal{M}}^*_Q$ is the adjoint operator of $ \tilde{\mathcal{M}}_Q$. One notices that $ \tilde{\mathcal{M}}_Q$ can be linearized as
\begin{equation*}
 \tilde{\mathcal{M}}_Q f(x) =\sup_{2^{l-m-1} \leq t' \leq 2^{l-m}}A^{\vec{\Phi}}_{t'} \left(f \mathbbm{1}_{\frac{1}{3} Q}\right) (x) = A^{\vec{\Phi}}_{t_Q(x)} \left(f \mathbbm{1}_{\frac{1}{3} Q}\right) (x) ,
\end{equation*}
where $t_Q(x) \in [2^{l-m-1}, 2^{l-m}]$ such that for $x \in Q$ fixed.

Let $\widetilde{A}^{\vec{\Phi}}_t$ be the abbreviation for the adjoint operator of the average operator $A_t^{\vec{\Phi}}$. By invoking the mean-zero property of $b_P$ \eqref{mean:0} and thus of $B_{l-k}$, one deduces
\begin{align}
\eqref{pre:mean:0} =& \sum_{P \in \mathcal{B}_{l-k}} \frac{1}{|P|}\bigg|\int_P\int_{P} B_{l-k}(x) \mathbbm{1}_Q(x)\left(\widetilde{A}^{\vec{\Phi}}_{t_Q(x)} (g \mathbbm{1}_{F_Q})(x) -\widetilde{A}^{\vec{\Phi}}_{t_Q(x)} (g \mathbbm{1}_{F_Q})(x')\right) dxdx' \bigg| \nonumber\\
\leq & \sum_{P \in \mathcal{B}_{l-k}}\frac{1}{|P_0|} \int_{P_0} \bigg| \int_{P}B_{l-k}(x) \mathbbm{1}_Q(x) \left( \widetilde{A}^{\vec{\Phi}}_{t_Q(x)} (g \mathbbm{1}_{F_Q})(x) - \widetilde{A}^{\vec{\Phi}}_{t_Q(x)} (g \mathbbm{1}_{F_Q})(x-y)\right) dx\bigg|dy \nonumber \\
\leq &\frac{1}{|P_0|} \int_{P_0} \bigg| \int B_{l-k}(x) \mathbbm{1}_Q(x) \left(\widetilde{A}^{\vec{\Phi}}_{t_Q(x)} (g \mathbbm{1}_{F_Q})(x) - \widetilde{A}^{\vec{\Phi}}_{t_Q(x)} (g \mathbbm{1}_{F_Q})(x-y)\right) dx\bigg|dy,\label{pre:cont}
\end{align}
where $P_0$ is the dyadic $\delta$-cube centered at $0$ with the same shape as $P$.

We now recall the local continuity property of the maximal operator \eqref{cont:scale1}, which is equivalent to its scale-invariant dual form:
\begin{equation} \label{dual:cont}
\langle \sup_{2^{l} \leq t \leq 2^{l+1}} \big|A_t^{\vec{\Phi}} (f_1 \mathbbm{1}_Q) - \tau_y A_t^{\vec{\Phi}} (f_1 \mathbbm{1}_Q) \big|, f_2 \mathbbm{1}_Q \rangle  \lesssim \left| \left(\frac{y_1}{2^{lb_1}}, \ldots, \frac{y_n}{2^{lb_n }} \right)\right|^\epsilon \langle f_1 \rangle_{Q,p} \langle f_2 \rangle_{Q,q'} |Q|,
\end{equation}
for some $\epsilon=\epsilon(n,p,q') >0$ and the bound is uniform with respect to $y$.

 One can now apply \eqref{dual:cont} to estimate \eqref{pre:cont} as
\begin{align*}
& \frac{1}{|P_0|} \int_{P_0} \int \big| B_{l-k}(x) \mathbbm{1}_Q(x)\big| \sup_{2^{l-m-1} \leq t \leq 2^{l-m} }\big| \widetilde{A}^{\vec{\Phi}}_{t} (g \mathbbm{1}_{F_Q})(x) - \widetilde{A}^{\vec{\Phi}}_{t} (g \mathbbm{1}_{F_Q})(x-y) \big| dxdy \\
\lesssim  & \sup_{y:  |y_j| \lesssim 2^{\lceil (l-k) b_j \rceil}} \left| \left(\frac{y_1}{2^{(l-m-1)b_1}}, \ldots, \frac{y_n}{2^{(l-m-1)b_n }} \right)\right|^\epsilon  \langle B_{l-k} \mathbbm{1}_Q \rangle_{Q,p} \langle g \mathbbm{1}_{F_Q} \rangle_{Q,q'} |Q|
\\
\leq & C(b_1,\ldots, b_n, n)2^{-k\epsilon}\langle B_{l-k} \mathbbm{1}_Q \rangle_{Q,p} \langle g \mathbbm{1}_{F_Q}\rangle_{Q,q'} |Q|,
\end{align*}
where we have also used the fact that $F_Q \subseteq Q$ for each $Q \in \mathscr{Q}_{Q_0}^{good}$.

It remains to show that uniformly in $k \geq 0$,
\begin{equation*}
\sum_{l \in \mathbb{Z}}\sum_{\substack{Q \in \mathscr{Q}_{Q_0}^{good} \\ l_1(Q) = 2^{{\lceil lb_1 \rceil}}}}\langle B_{l-k} \mathbbm{1}_Q \rangle_{Q,p} \langle g \mathbbm{1}_{F_Q}\rangle_{Q,q'} |Q| \lesssim \langle f \rangle_{Q_0,p} \langle g \mathbbm{1}_{F_Q} \rangle_{Q_0,q'} |Q_0|,
\end{equation*}
whose argument follows closely from the investigation in Lacey \cite{Lacey}. The key properties requested in the argument include:
\begin{enumerate}
\item
disjointness of the supports of $B_{l-k}$ when varying $l \in \mathbb{Z}$;
\item
disjointness of the supports of $F_Q$ for all $Q \in \mathscr{Q}_{Q_0}^{good}$;
\item
stopping condition that $\langle f \rangle_P \lesssim \langle f \rangle_{Q_0,p}$ for any $P \in \mathcal{B}$;
\item
the assumption that $f$ is a characteristic function.
\end{enumerate}
Due the reliance of the argument on the last condition, the inductive estimate and thus the sparse bound is proved for the special case when $f$ is a characteristic function. The generic case follows from a standard procedure of expressing $f$ in terms of its level sets, applying the sparse bound for each piece and invoking Carleson embedding inequality, which are all included in Lacey \cite{Lacey}. This completes the proof of Theorem \ref{thm:sparse}.
\vskip .2in

\section{Proofs of  $L^{p} \rightarrow L^{q}$ estimates and continuity lemmas for localized maximal functions}
\subsection{Preliminaries}
In this subsection, we introduce  some lemmas which are useful in the proof of the  $L^{p} \rightarrow L^{q}$ estimates.

We first introduce some local smoothing results for Fourier integral operators whose phase functions satisfy the so-called ``cinematic curvature" condition, which play a key role throughout our proof. Here we follow the notation adopted by Subsection 2.1.1 in \cite{WL}, and we refer the readers to \cite{WL} for more details.

In our case, we consider the Fourier integral operators given by
\begin{equation}\label{localinteg}
\mathcal{F}f(z)=\int_{{\mathbb{R}}^n}e^{i\varphi(z,\eta)}a(z,\eta)\widehat{f}(\eta)d\eta,\hspace{0.3cm}z=(x,t)\in \mathbb{R}^n\times \mathbb{R},\hspace{0.3cm}f\in C_0^{\infty}(\mathbb{R}^n),
\end{equation}
where the phase function $\varphi$  satisfies the ``cinematic curvature" condition:
\begin{equation}
\textrm{rank}\hspace{0.1cm} \varphi_{x,\eta}''\equiv 2,
\end{equation}
and
 \begin{equation}\label{cone:condition}
\textrm{rank}(\frac{\partial^2}{\partial\eta_j\partial\eta_k})\langle\varphi_z'(z_0,\eta),\theta\rangle= 1,
\end{equation} 
where $\pm \theta$ are the unique directions for which $\bigtriangledown_{\eta}\langle\varphi'_z(z_0,\eta),\theta\rangle=0$.
The symbol $a$ of order $\mu$ has small conical support in $\mathbb{R}^{3}\times \mathbb{R}^2$, which means that $a$ vanishes for all $z$ outside a small compact set and for all $\eta=(\eta_1,\eta')$ outside a narrow cone $\{\eta:|\eta'|\leq\varepsilon\eta_1\}$.

We fix $\beta\in C_0^{\infty}(\mathbb{R})$ supported in $[1/2,2]$ and set $a_{\lambda}(x,t,\eta)=\lambda^{-\mu}\beta(|\eta|/\lambda)a(x,t,\eta)$ for fixed $\lambda>1$. Then $a_{\lambda}$ is a symbol of order zero and satisfies the usual symbol estimates uniformly in $\lambda$.  Mockenhaupt, Seeger and Sogge showed the following dyadic $L^{p}$-estimate for the Fourier integral operator $\mathcal{F}$ defined in \eqref{localinteg}.
\begin{thm}\cite{mss}\label{lplocalsmoothing} 
Let $\varphi$ and $a_\lambda$ be defined as above. Then the corresponding Fourier integral operator satisfies
\begin{equation*}
\biggl(\int_{1/2}^4\int_{{\mathbb{R}}^2}\left|\int_{{\mathbb{R}}^2}e^{i\varphi(x,t,\eta)}a_{\lambda}(x,t,\eta)\widehat{f}(\eta)d\eta\right|^pdxdt\biggl)^{1/p}\leq
C_p {\lambda}^{1/2-1/p-\epsilon(p)} \|f\|_{L^{p}(\mathbb{R}^{2})},
\end{equation*}
where $\epsilon(p)=\frac{1}{2p}$, if $4\leq p<\infty$;
$\epsilon(p)=\frac{1}{2}(\frac{1}{2}-\frac{1}{p})$, if $2<p\leq 4$.
\end{thm}

The corresponding $L^{p} \rightarrow L^{q}$ estimate was obtained by Lee (see Corollary 1.5, \cite{SL2}).

\begin{thm}\cite{SL2}\label{lpqlocalsmoothing}
Let $\varphi$ and $a_\lambda$ be defined as above. Then the corresponding Fourier integral operator satisfies
\begin{equation*}
\biggl(\int_{1/2}^4\int_{{\mathbb{R}}^2}\left|\int_{{\mathbb{R}}^2}e^{i\varphi(x,t,\eta)}a_{\lambda}(x,t,\eta)\widehat{f}(\eta)d\eta\right|^qdxdt\biggl)^{1/q}\leq
C_{p,q} \lambda^{(\frac{1}{2} -\frac{3}{q}+ \frac{1}{p} +\epsilon)}  \|f\|_{L^{p}(\mathbb{R}^{2})},
\end{equation*}
for any $\epsilon >0$ and    $(\frac{1}{p}, \frac{1}{q}) \in \{(\frac{1}{p}, \frac{1}{q}):   \frac{1}{q} \le  \frac{3}{5p}, \frac{3}{q} \le 1-\frac{1}{p} , 0 \le \frac{1}{q} \le \frac{3}{14} \}$.
\end{thm}

We will frequently use the following method of stationary phase.
\begin{lem} (Theorem 1.2.1 in \cite{sogge2})\label{lem:Lemma1}
Let S be a smooth hypersurface in $\mathbb{R}^n$ with non-vanishing Gaussian curvature and $d\mu$ be the Lebesgue measure on $S$. Then,
\begin{equation}
|\widehat{d\mu}(\xi)|\leq C(1+|\xi|)^{-\frac{n-1}{2}}.
\end{equation}

Moreover, suppose that $\Gamma\subset \mathbb{R}^n\backslash 0$ is the cone consisting of all $\xi$ which are normal to $S$ at some point $x\in S$ belonging to a fixed relatively compact neighborhood $\mathcal{N}$ of supp $d\mu$. Then
\begin{equation*}
\left|\left(\frac{\partial}{\partial\xi}\right)^{\alpha}\widehat{d\mu}(\xi)\right|=\mathcal{O}(1+|\xi|)^{-N}\hspace{0.2cm}for\hspace{0.1cm} all \hspace{0.1cm}N\in \mathbb{N}, \hspace{0.2cm}if \hspace{0.2cm}\xi\not \in \Gamma,
\end{equation*}
\begin{equation}
\widehat{d\mu}(\xi)=\sum e^{-i\langle x_j,\xi \rangle}a_j(\xi)\hspace{0.5cm}if \hspace{0.2cm}\xi\in \Gamma,
\end{equation}
where the finite sum is taken over all $x_j\in \mathcal{N}$ having $\xi$ as the normal and
\begin{equation}
\left|\left(\frac{\partial}{\partial\xi}\right)^{\alpha}a_j(\xi)\right|\leq C_{\alpha}(1+|\xi|)^{-\frac{n-1}{2}-|\alpha|}.
\end{equation}
\end{lem}

 We also use the following well-known estimate. 
\begin{lem} (Theorem 2.4.2 in \cite{sogge2})\label{lem:Lemma3}
Suppose that $F$ is $C^1(\mathbb{R})$. Then if $p>1$ and $1/p+1/p'=1$,
\begin{equation*}
\sup_{\lambda}|F(\lambda)|^p\leq |F(0)|^p+p\biggl(\int|F(\lambda)|^pd\lambda\biggl)^{1/p'}\biggl(\int|F'(\lambda)|^pd\lambda\biggl)^{1/p}.
\end{equation*}
\end{lem}

\subsection{Localized maximal functions associated with curves in $\mathbb{R}^{2}$}
\subsubsection{The case when $da_{1} \neq a_{2}$: finite type curves}
Tis subsection is devoted to the proof of Theorem \ref{maintheorem1}, and the argument can be easily generalized to prove Theorem \ref{hisomaintheorem}. The necessary conditions for Theorem \ref{maintheorem1} are also discussed.

\textbf{Proof of Theorem \ref{maintheorem1}. } We first consider a simple model, then explain how to prove the general case.

\textbf{A simple model. } Denote
\[A_{t}f(y):= \int_{0}^{1}f(y_{1}-tx, y_{2}- t(x^{d} +c))dx, \quad d \ge 2.\]
Let us consider the $L^{p}\rightarrow L^{q}$ estimate of the maximal operator given by
$\sup_{t \in [1,2]} |A_{t}f(y) |.$

We choose  $\tilde{\rho}\in C_0^{\infty}(\mathbb{R})$ such that  supp $\tilde{\rho}\subset\{x: 1/2 \leq|x|\leq 2\}$  and $\sum_{k\in \mathbb{Z}}\tilde{\rho}(2^kx)=1$ for $x\in \mathbb{R}$. Then we decompose $A_{t}f$ by
\begin{align*}
A_tf(y)=\sum_{k \ge 1}\int_{0}^{1} f(y_1-tx,y_2-t(x^d +c))\tilde{\rho}(2^kx) dx.
\end{align*}
Due to the isometric operator on $L^p(\mathbb{R}^2)$ defined by $T_kf(x_1,x_2)=2^{(d+1)k/p}f(2^kx_1,2^{dk}x_2)$, it suffices to obtain the following estimate:
\[\sum_{k \ge 1}2^{(d+1)k(\frac{1}{p} - \frac{1}{q})-k}\left\|\sup_{t \in [1,2]}|\widetilde{A_t^{k}}|\right\|_{L^p\rightarrow L^q}\leq C_{p,q},\]
 where
\[\widetilde{A_t^{k}}f(y):=\int_{0}^{1} f(y_1-tx,y_2-t(x^d + 2^{dk}c))\tilde{\rho}(x)dx.\]

By means of the Fourier inversion formula, we have
\begin{align*}
\widetilde{A_t^{k}}f(y)&=\frac{1}{(2\pi)^2}\int_{{\mathbb{R}}^2}e^{i\xi\cdot y}\int_{\mathbb{R}}e^{-i(t\xi_1x+t  \xi_2(x^d + 2^{dk}c)) }\tilde{\rho}(x) dx\hat{f}(\xi)d\xi
\\
&=\frac{1}{(2\pi)^2}\int_{{\mathbb{R}}^2} e^{i\xi\cdot y} e^{-it2^{dk}c \xi_{2}} \widehat{d\mu}(t\xi)\hat{f}(\xi)d\xi,
\end{align*}
where
\[\widehat{d\mu}(\xi):=\int_{\mathbb{R}}e^{-i(\xi_1x+\xi_{2}x^d)}\tilde{\rho}(x) dx.\]

We choose a non-negative function $\beta\in C_0^{\infty}(\mathbb{R})$ such that supp $\beta\subset[1/2,2]$ and $\sum_{j\in\mathbb{Z}}\beta(2^{-j}r)=1$ for $r>0$. Define the dyadic operators
\[\widetilde{A_{t,j}^k}f(y)=\frac{1}{(2\pi)^2}\int_{{\mathbb{R}}^2}e^{i\xi\cdot y} e^{-it2^{dk}c \xi_{2}} \widehat{d\mu }(t\xi)\beta(2^{-j}t|\xi|)\hat{f}(\xi)d\xi,\]
and denote by $\widetilde{\mathcal{M}_{j}^k}$ the corresponding maximal operator.  Now we have that
\begin{equation*}
\sup_{t \in [1,2]}|\widetilde{A_{t,j}^k}f(y)|\leq \widetilde{\mathcal{M}^{k,0}}f(y)+\sum_{j\geq 1}\widetilde{\mathcal{M}^{k}_{j}}f(y), \hspace{0.2cm}\textmd{for}\hspace{0.2cm}y\in \mathbb{R}^2,
\end{equation*}
where
\[\widetilde{\mathcal{M}^{k,0}}f(y):=\sup_{t \in [1,2]}|\sum_{j\leq 0}\widetilde{A_{t,j}^k}f(y)|.\]

The main result follows from the estimates (E1)-(E3) below: \\
\textbf{(E1)} For $q \ge p \ge 1$, we have
\begin{align}
\|\widetilde{\mathcal{M}^{k,0}}f\|_{L^{q}({\mathbb{R}}^2)} \lesssim (c2^{\frac{dk}{q}} +1) \|f\|_{L^{p}({\mathbb{R}}^2)}.
\end{align}
\textbf{(E2)} For any $2 < p < \infty$ and $j\geq 1$, there exists  $\epsilon(p) >0$, such that
\begin{equation}
\|\widetilde{\mathcal{M}^{k}_{j}}f\|_{L^{p}({\mathbb{R}}^2)}\leq (c2^{\frac{dk}{p}} +1)  2^{-\epsilon(p)j}\|f\|_{L^p({\mathbb{R}}^2)}.
\end{equation}
\textbf{(E3)} For any $j\geq 1$ and all   $p,q$ satisfying $\frac{1}{2p}< \frac{1}{q} \le  \frac{3}{5p}$, $\frac{3}{q} \le 1-\frac{1}{p} $, we have
\begin{align}
\|\widetilde{\mathcal{M}^{k}_{j}}f\|_{L^q({\mathbb{R}}^2)}
\lesssim (c2^{\frac{dk}{q}} +1) 2^{(\frac{1}{p} -\frac{2}{q}  +\epsilon)j} \|f\|_{L^p({\mathbb{R}}^2)}.
\end{align}

Indeed, if $c=0$, it follows from (E3) that for $p,q$ satisfying $\frac{1}{2p} <\frac{1}{q} \le \frac{3}{5p}$, $\frac{3}{q} \le 1-\frac{1}{p} $,   there exists $\epsilon (p,q)>0$ such that
\begin{equation}\label{hn-pq1}
\|\widetilde{\mathcal{M}^{k}_{j}}f\|_{L^q({\mathbb{R}}^2)} \leq  C_{p,q}  2^{-\epsilon(p,q)j} \|f\|_{L^{p}}.
\end{equation}
Combining this with  (E2) and the Riesz interpolation theorem, we deduce (\ref{hn-pq1}) for all   $p,q$ satisfying $\frac{1}{2p}< \frac{1}{q} \le  \frac{1}{p}$, $\frac{1}{q} > \frac{3}{p}-1 $. As a consequence,
\[\sum_{k \ge 1}2^{\frac{(d+1)k}{p} - \frac{(d+1)k}{q}-k}\left\|\sup_{t \in [1,2]}|\widetilde{A_{t}^{k}}|\right\|_{L^p\rightarrow L^q}\leq \sum_{k \ge 1}2^{(d+1)k(\frac{1}{p} - \frac{1}{q})-k} \sum_{j \ge 0} 2^{-\epsilon(p,q)j} C_{p,q} \le C_{p,q}\]
provided that  $p,q$ satisfy  $\frac{1}{2p}< \frac{1}{q} \le  \frac{1}{p}$, $\frac{1}{q} > \frac{3}{p}-1 $, $\frac{d+1}{p}- \frac{d+1}{q}- 1<0$.

If $c\neq0$, It follows from (E3) that for $p,q$ satisfying $\frac{1}{2p} <\frac{1}{q} \le \frac{3}{5p}$, $\frac{3}{q} \le 1-\frac{1}{p} $,   there exists $\epsilon (p,q)>0$ such that
\begin{equation}\label{hn-pq}
\|\widetilde{\mathcal{M}^{k}_{j}}f\|_{L^q({\mathbb{R}}^2)} \leq  C_{p,q}  2^{\frac{dk}{q}} 2^{-\epsilon(p,q)j} \|f\|_{L^{p}}.
\end{equation}
Then  (E2) and  the Riesz interpolation theorem yield (\ref{hn-pq}) for all  $p,q$ satisfying $\frac{1}{2p}< \frac{1}{q} \le  \frac{1}{p}$, $\frac{1}{q} > \frac{3}{p}-1 $. Therefore, we have
\[\sum_{k \ge 1}2^{\frac{(d+1)k}{p} - \frac{(d+1)k}{q}-k}\left\|\sup_{t \in [1,2]}|\widetilde{A_t^{k}}|\right\|_{L^p\rightarrow L^q}\leq \sum_{k \ge 1}2^{\frac{(d+1)k}{p} - \frac{k}{q}-k} \sum_{j \ge 0} 2^{-\epsilon(p,q)j} C_{p,q} \le C_{p,q}\]
if $p,q$ satisfy  $\frac{1}{2p}< \frac{1}{q} \le  \frac{1}{p}$,   $\frac{d+1}{p}- \frac{1}{q}- 1<0$.

Now it remains to prove (E1), (E2) and (E3). First let us prove (E1). Note that $\widetilde{\mathcal{M}^{k,0}}f(y)=\sup_{t \in [1,2]}|f*K_{ t }(y)|$, where
\[K_{t}(y):=\int_{{\mathbb{R}}^2}e^{i\xi\cdot y} e^{-it2^{dk}c \xi_{2}} \widehat{d\mu}(t\xi)\beta_{0}(t|\xi|)d\xi,\]
where $\beta_{0} \in C_0^{\infty}(\mathbb{R})$ is supported in $[0,2]$. Then Lemma \ref{lem:Lemma1} implies that for a multi-index $\alpha$,
\[\left|\left(\frac{\partial}{\partial\xi}\right)^{\alpha}\widehat{d\mu}(\xi)\right|\leq C_{B,\alpha}(1+|\xi|)^{-1/2}.\]
By integration by parts, we obtain that for each integer $N \ge 1$,
\begin{equation}\label{kernelestimate1}
|K_{t}(y)|\leq C_N (1+|y_{1}| + |y_{2}-c2^{dk}t|)^{-N}.
\end{equation}
Choosing $N$ sufficiently large in (\ref{kernelestimate1}) and applying Lemma \ref{lem:Lemma3}, by Young's inequality,  for $q \ge p \ge1$,  we have
\begin{align}\label{hnlow}
&\|\widetilde{\mathcal{M}^{k,0}}f\|_{L^{q} (\mathbb{R}^{2})} = \|\sup_{t \in [1,2]}|f*K_{ t}|\|_{L^{q}(\mathbb{R}^{2})} \nonumber\\
&\le c2^{\frac{dk}{q}} \biggl\|\frac{C_{N}}{(1+|\cdot -(0,c2^{dk}t) |)^{N}} * |f|\biggl\|^{1-\frac{1}{q}}_{L^{q} (\mathbb{R}^{2} \times [1,2]) } \biggl\|\frac{C_{N}}{(1+|\cdot -(0,c2^{dk}t) |)^{N+1}} * |f|\biggl\|^{\frac{1}{q}}_{L^{q} (\mathbb{R}^{2} \times [1,2]) } \nonumber\\
&\quad +  \biggl\|\frac{C_{N}}{(1+|\cdot -(0,c2^{dk}) |)^{N}} * |f|\biggl\|_{L^{q} (\mathbb{R}^{2}) } \nonumber\\
&\lesssim (c2^{\frac{dk}{q}} +1 )\|f\|_{L^{p}(\mathbb{R}^{2})}.
\end{align}
Then we arrive at (E1).

In order to verify (E2) and (E3), we first consider
\begin{equation}\label{fouriermensre1}
\widehat{d\mu}(t\xi)=\int_{\mathbb{R}}e^{-it\xi_2(-sx+x^d)}\tilde{\rho}(x) dx,
\end{equation}
where
\[s =-\frac{\xi_1}{\xi_2}, \hspace{0.3cm}\textrm{for} \hspace{0.2cm} \xi_2\neq 0.\]


Let
\[\Phi(s,x)=-sx+x^d ,\]
then we have
\begin{equation*}
\partial_x\Phi(s,x)=-s+dx^{d-1}
\end{equation*}
and
\begin{equation*}
\partial_x^2\Phi(s,x)= d(d-1)x^{d-2}.
\end{equation*}
Since $\tilde{\rho}$ is supported in $[1, 2]$, there exists a smooth solution $x_c= (\frac{s}{d})^{1/(d-1)}$ for the equation $\partial_x\Phi(s,x)=0$ provided that $s  \in [d, d2^{d-1}]$. The phase function can be written as
\begin{equation}\label{phase function 1}
-t\xi_2\tilde{\Phi}(s)=(d-1)t\xi_{2} \biggl(-\frac{\xi_{1}}{d\xi_{2}}\biggl)^{d/(d-1)}.
\end{equation}

By applying the method of stationary phase, we have
 \begin{equation}\label{fouriermeansuredecompose1}
\widehat{d\mu}(t\xi)=e^{-it \xi_2\tilde{\Phi}(s)}\chi_{0}(\frac{\xi_1}{\xi_2})
\frac{A_0(t\xi)}{(1+t|\xi|)^{1/2}}+B_0(t\xi),
\end{equation}
where $\chi_0$ is a smooth function supported in the interval $[-d2^{d-1}, -d]$. $A_0(t\xi)$ is contained in a bounded subset of symbols of order zero. More precisely, for arbitrary $t \in [1,2]$,
 \begin{equation}\label{mainsymhol1}
 |D_{\xi}^{\alpha}A_0(t\xi)|\leq C_{\alpha}(1+|\xi|)^{-\alpha},
\end{equation}
where $C_{\alpha}$ is independent of  $t$. Furthermore, $B_0$ is a remainder term and satisfies that for arbitrary $t\in [1,2]$,
\begin{equation}\label{remaindersymbol1}
 |D_{\xi}^{\alpha}B_0(t\xi)|\leq C_{\alpha,N}(1+|\xi|)^{-N},
\end{equation}
where $C_{\alpha,N}$ are admissible constants and again do not depend on $t$.

First, let us consider the remainder part. Set
\[\mathcal{M}_{j}^{k, 0}f(y):=\sup_{t\in[1,2]}\left|\frac{1}{(2\pi)^2}\int_{{\mathbb{R}}^2}e^{i\xi\cdot y}  e^{-it2^{dk}c \xi_{2}} B_0( t\xi)\beta(2^{-j}| t\xi|)\hat{f}(\xi)d\xi \right|.\]
By (\ref{remaindersymbol1}) and integration by parts, we deduce that $$|(B_k\beta(2^{-j}\cdot))^{\vee}(x)|\leq C_N 2^{-jN}(1+|x|)^{-N}.$$ As a result,
\[\mathcal{M}_{j}^{k, 0}f(y) \le \sup_{t\in[1,2]}\frac{C_{N}2^{-jN}}{(2\pi)^2 }\int_{{\mathbb{R}}^2} \frac{|f(x)|}{(1+|y-x -(0,c2^{dk}t)|)^{N}} dx.\]
We use the similar derivation for the inequality (\ref{hnlow}) to obtain
\[ \|\mathcal{M}_{j}^{k, 0}f\|_{L^{q}(\mathbb{R}^{2})} \le (c2^{\frac{dk}{q}} +1 ) 2^{-jN} \|f\|_{L^{p}(\mathbb{R}^{2})} \]
for each $q \ge p \ge 1$.

Next, we focus on the main term. Let
\begin{equation}\label{mainterm1}
A_{t,j}^kf(y):=\frac{1}{(2\pi)^2}\int_{{\mathbb{R}}^2}e^{i(\xi \cdot y -tc2^{dk}\xi_{2}-t\xi_2\tilde{\Phi}(s))}\chi_0(\frac{\xi_1}{\xi_2})
\frac{A_0( t\xi)}{(1+|t\xi|)^{1/2}}\beta(2^{-j}| t \xi|)\hat{f}(\xi)d\xi.
\end{equation}
Denote by $\mathcal{M}_{j}^{k,1}$ the corresponding maximal operator over $[1,2]$.  We choose a non-negative function $\rho_{1}(t)$ which equals to $1$ on $[1,2]$ and supported in $[1/2, 4]$. Then by Lemma \ref{lem:Lemma3},
\begin{align}\label{sobolevembedding1}
&\|\widetilde{\mathcal{M}_{j}^{k,1}}f\|_{L^q} \nonumber\\
&\lesssim     \biggl\|  \sup_{t\in[1/2,4] } \rho_{1}(t) \left| \int_{{\mathbb{R}}^2}e^{i\xi\cdot y}  e^{-it2^{dk}c \xi_{2} -t\xi_2\tilde{\Phi}(s)} \chi_{0}(\frac{\xi_1}{\xi_2})
\frac{A_0(t\xi)}{(1+t|\xi|)^{1/2}}\beta(2^{-j}| t\xi|)\hat{f}(\xi)d\xi \right| \biggl\|_{L^{q}(\mathbb{R}^{2})}     \nonumber\\
&\leq C
2^{\frac{j}{q}  -\frac{j}{2}} (c2^{\frac{dk}{q}} + 1)\nonumber\\
& \quad \times \biggl(\int_{{\mathbb{R}}^2}\int^4_{1/2}\biggl| \int_{{\mathbb{R}}^2}e^{i(\xi \cdot y-it2^{dk}c \xi_{2}-t\xi_2\tilde{\Phi}(s))} \widetilde{A_0}(\xi,t)\beta(2^{-j}| t \xi|)\hat{f}(\xi)d\xi\biggl|^qdtdy \biggl)^{
 1/q},
\end{align}
where $\widetilde{A_0}(\xi,t)$ is a symbol of order zero,  $\xi$ is supported in the cone $\{\xi \in \mathbb{R}^{2}: |\frac{\xi_{1}}{\xi_{2}}| \approx 1 \}$ and $t$ is supported in $[1/2, 4]$.  If  $a(\xi,t):= \widetilde{A}_0(\xi,t)\beta(2^{-j}|t \xi|)$, then $a(\xi,t)$ is a symbol of order zero, i.e. for any $t\in [1/2,4]$, $\alpha\in {\mathbb{N}}^2$,
\[\left|\left(\frac{\partial}{\partial \xi}\right)^{\alpha}a(\xi,t)\right|\leq C_{\alpha}(1+|\xi|)^{-|\alpha|}.\]
Now we are left to estimate the Fourier integral operator defined by
\[F_{j}^{k}f(y,t) := \int_{{\mathbb{R}}^2}e^{i(\xi \cdot y-it2^{dk}c \xi_{2}-t\xi_2\tilde{\Phi}(s))} a(\xi, t)\hat{f}(\xi)d\xi. \]
Changing variables implies that
\[\|F_{j}^{k}f(y,t)\|_{L^{q}(\mathbb{R}^{2} \times [1/2,4])} = \|F_{j}f(y -
(0,c2^{dk}t),t)\|_{L^{q}(\mathbb{R}^{2} \times [1/2,4])} = \|F_{j}f(y,t)\|_{L^{q}(\mathbb{R}^{2} \times [1/2,4])}. \]
Here
\[F_{j} f(y,t) := \int_{{\mathbb{R}}^2}e^{i(\xi \cdot y -t\xi_2\tilde{\Phi}(s))} a(\xi, t)\hat{f}(\xi)d\xi. \]

 We  observe that the operator $F_{j}$ can be localized. Since $\tilde{\Phi}(s)$ is homogeneous of degree zero in $\xi$ and $\frac{\xi_1}{\xi_2}\approx 1$,
then
$$\left|\nabla_{\xi}[\xi \cdot (y-x)-t\xi_2\tilde{\Phi}(s)]\right|\geq C|y-x|$$
provided $|y-x|\geq L$ for some large $L$. By integration by parts, we will see that the kernel of the operator $F_{j}$
\[K(x, y, t)= \int_{{\mathbb{R}}^2}e^{i(\xi \cdot (y-x) -t\xi_2\tilde{\Phi}(s))} a(\xi, t) d\xi\]
is determined by
 \[|K(x,y,t)| \le 2^{-jN}\mathcal{O}(|y-x|^{-N})\]
if $|y-x|\geq L$. Therefore, we only need to estimate
\[G_{j}f(y,t) =\int_{\mathbb{R}^{2}}K(x,y,t) \chi_{B(y,L)}(x)f(x)dx. \]
Here we abuse the notation for simplicity. If we have proved the local estimate
\[\biggl\|G_{j}f(y,t)\biggl\|_{L^{q}(B(0,L) \times [1/2,4])} \lesssim 2^{\epsilon_{1}(j) + \epsilon_{2}(k)} \|f\|_{L^{p}(\mathbb{R}^{2})}\]
for some indexes $\epsilon_{1}(j)$ and $\epsilon_{2}(k)$ related to $j$, $k$, respectively. Then we decompose $\mathbb{R}^{2}=\bigcup_{\mathfrak{l} \in \mathbb{Z}^{2}} B(y_{\mathfrak{l}}, L)$, such that for $\mathfrak{l }\neq \mathfrak{l}^{\prime}$,  $|y_{\mathfrak{l}} - y_{\mathfrak{l}^{\prime}}| \approx L$,  we have
\begin{align}
\biggl\|G_{j}f(y,t)\biggl\|^{q}_{L^{q}(\mathbb{R}^{2} \times [1/2,4])} &= \sum_{\mathfrak{ l }\in \mathbb{Z}^{2}} \biggl\|G_{j}f(y,t)\biggl\|^{q}_{L^{q}(B(y_{\mathfrak{l}}, L) \times [1/2,4])} \nonumber\\
&= \sum_{\mathfrak{ l }\in \mathbb{Z}^{2}} \biggl\| \int_{\mathbb{R}^{2}}K(x,y,t) \chi_{B(y,L)}(x)f(x)dx \biggl\|^{q}_{L^{q}(B(y_{\mathfrak{l}}, L) \times [1/2,4])} \nonumber\\
&= \sum_{\mathfrak{ l }\in \mathbb{Z}^{2}} \biggl\| \int_{\mathbb{R}^{2}}K(x,y,t) \chi_{B(y,L)}(x) \chi_{B(0,2L)}(x)f(x - y_{\mathfrak{l}})dx \biggl\|^{q}_{L^{q}(B(0, L) \times [1/2,4])} \nonumber\\
&\le  2^{q\epsilon_{1}(j) + q\epsilon_{2}(k)}  \sum_{\mathfrak{ l }\in \mathbb{Z}^{2}}  \| f \|^{q}_{L^{p}(B(y_{\mathfrak{l}}, 2L) )} \nonumber\\
&\lesssim  2^{q\epsilon_{1}(j) + q\epsilon_{2}(k)} \|f\|^{q}_{L^{p}(\mathbb{R}^{2})}. \nonumber
\end{align}
Hence it suffices to estimate $\tilde{F}_{j}$ defined by
\begin{equation}
\tilde{F}_{j}f(y,t) = \tilde{\rho_1}(y,t)\int_{{\mathbb{R}}^2}e^{i(\xi \cdot y-t\xi_2\tilde{\Phi}(s))} a(\xi, t)\hat{f}(\xi)d\xi,
\end{equation}
where $\tilde{\rho_1}\in C_0^{\infty}(\mathbb{R}^2\times[\frac{1}{2},4])$.

It is not difficult to check that the phase function satisfies the ``cinematic curvature" condition, so we can apply the local smoothing estimates. We need two estimates for $ \tilde{F}_{j}$. The first one comes from Theorem \ref{lplocalsmoothing}.

\begin{thm}\label{MSS-L^(p,q)therom}
For $2 <p < \infty$, there exists $\epsilon(p) >0$ such that
\[\biggl(\int_{{\mathbb{R}}^2}\int^4_{1/2}|\tilde{F}_{j}f(y,t)|^pdtdy \biggl)^{1/p} \leq C   2^{(\frac{1}{2} -\frac{1}{p}) j} 2^{-\epsilon(p)j} \|f\|_{L^p({\mathbb{R}}^2)},\]
where $\epsilon(p)=\frac{1}{2p}$, if $4\leq p<\infty$; $\epsilon(p)=\frac{1}{2}(\frac{1}{2}-\frac{1}{p})$, if $2<p\leq 4$.
\end{thm}

The second one comes from Theorem \ref{lpqlocalsmoothing}.

\begin{thm}\label{Lee-L^(p,q)therom}
For any $\epsilon >0$ and all   $p,q$ satisfying $\frac{1}{2p}< \frac{1}{q} \le  \frac{3}{5p}$, $\frac{3}{q} \le 1-\frac{1}{p} $, we have
\[\biggl(\int_{{\mathbb{R}}^2}\int^4_{1/2}|\tilde{F}_{j}f(y,t)|^qdtdy \biggl)^{1/q}
\leq C   2^{(\frac{1}{2} -\frac{3}{q}+ \frac{1}{p} +\epsilon)j} \|f\|_{L^p({\mathbb{R}}^2)}.\]
\end{thm}

By inequality (\ref{sobolevembedding1}) and Theorem \ref{MSS-L^(p,q)therom}, we obtain
\begin{align*}
\|\widetilde{\mathcal{M}_{j}^{k,1}}f\|_{L^p({\mathbb{R}}^2)}&\leq C2^{\frac{j}{p}  -\frac{j}{2}} (c2^{\frac{dk}{q}} + 1)  2^{(\frac{1}{2} -\frac{1}{p}) j} 2^{-\epsilon(p)j} \|f\|_{L^p({\mathbb{R}}^2)} \\
&\lesssim  (c2^{\frac{dk}{q}} + 1)  2^{-\epsilon(p)j} \|f\|_{L^p({\mathbb{R}}^2)},
\end{align*}
for all $p > 2$, then (E2) is established.

Inequality (\ref{sobolevembedding1}) and Theorem \ref{Lee-L^(p,q)therom} yield
\begin{align*}
\|\widetilde{\mathcal{M}_{j}^{k,1}}f\|_{L^q({\mathbb{R}}^2)}&\leq C2^{\frac{j}{q}  -\frac{j}{2}} (c2^{\frac{dk}{q}} + 1)   2^{(\frac{1}{2} -\frac{3}{q}+ \frac{1}{p} +\epsilon)j} \|f\|_{L^p({\mathbb{R}}^2)} \\
&\lesssim (c2^{\frac{dk}{q}} + 1)  2^{(\frac{1}{p}-\frac{2}{q} + \epsilon)j} \|f\|_{L^p({\mathbb{R}}^2)},
\end{align*}
provided that $p,q$ satisfy $\frac{1}{2p}< \frac{1}{q} \le  \frac{3}{5p}$, $\frac{3}{q} \le 1-\frac{1}{p} $,  then (E3) is proved.

\textbf{The general case. }The above method can also be applied to estimate the operator
\[A_{t}f(y):= \int_{\mathbb{R}} f(y_1-tx,y_2-t(x^d\phi(x) +c))\eta(x)dx.\]
Similarly, it is sufficient to investigate the operator
\[\widetilde{A_t^{k}}f(y):=\int_{0}^{1} f(y_1-tx,y_2-t(x^d\phi(\frac{x}{2^{k}}) + 2^{dk}c))\tilde{\rho}(x)dx.\]
Notice that $k$ is sufficiently large because $\eta$ is supported in a sufficiently small neighborhood of the origin.

Since $\phi$ satisfies (\ref{phi}), we have $x^{d}\phi(\frac{x}{2^k})= \phi(0) x^{d} + 2^{-mk} \frac{ \phi^{(m)}(0) }{m!} x^{m+d}  + 2^{-(m+1)k}\mathcal{O}(x^{d+m+1})$. Then $x^{d}\phi(\frac{x}{2^k})$ is a small perturbation of $x^{d}$. The remainder terms will not change the results which are obtained by stationary phase method.  In particular, by stationary phase method, we need to estimate the main term
\begin{equation*}
A^k_{t,j}f(y):=\int_{{\mathbb{R}}^2}e^{i(\xi\cdot y -ct2^{dk}\xi_{2}-t \xi_2\tilde{\Phi}(s,\delta))}\chi_{k,d,m}(\frac{\xi_1}{ \xi_2})
\frac{A_{k,d,m}( t\xi)}{(1+| t\xi|)^{1/2}}
\beta(2^{-j}| t\xi|)\hat{f}(\xi)d\xi,
\end{equation*}
where $\chi_{k,d,m}$ is a smooth function supported in the conical region $[c_{k,d,m},\widetilde{c_{k,d,m}}]$, for certain non-zero constants $c_{k,d,m}$ and $\widetilde{c_{k,d,m}}$ dependent only on $k, m$ and $d$. $A_{k,d,m}$ is a symbol of order zero in $\xi$ and $\{A_{k,d,m}(t\xi)\}_k$ is contained in a bounded subset of symbol of order zero. 

We may assume $\phi(0)=1/d$, then the phase function can be written as
\begin{align*}
-t\xi_2\tilde{\Phi}(s,\delta)=(d-1)t\xi_{2} \biggl(-\frac{\xi_{1}}{d\xi_{2}}\biggl)^{d/(d-1)}-t\xi_{2}\frac{\delta^m\phi^{(m)}(0)}{ m!}
\biggl(-\frac{\xi_1}{\xi_2}\biggl)^{\frac{d+m}{d-1}}+ R(t,\xi,\delta,d),
\end{align*}
where $R(t,\xi,\delta,d)$ is homogeneous of degree one in $\xi$ and has at least $m+1$ power of $\delta$. This phase function can be perceived as a small perturbation of $(d-1)t\xi_{2} \biggl(-\frac{\xi_{1}}{d\xi_{2}}\biggl)^{d/(d-1)}$, so the local smoothing results of \cite{mss} and \cite{SL2} can be applied as in the proof of the model case.

Now let us consider the necessary conditions for Theorem  \ref{maintheorem1}.
\begin{thm}\label{hnnece}
The estimate
\begin{equation}\label{neceinequality}
 \biggl\| \sup_{t\in[1,2]} \biggl|\int_{0}^{1}f(y_{1}-tx, y_{2}- t(x^{d} +c))dx\biggl| \biggl\|_{L^{q}(\mathbb{R}^{2})} \lesssim \|f\|_{L^{p}(\mathbb{R}^{2})}
 \end{equation}
holds only if the following conditions are satisfied,

(1) when $c =0$, $p, q$ satisfy (C1) $\frac{1}{q} \le \frac{1}{p}$; (C2) $\frac{1}{q} \ge \frac{1}{2p}$; (C3) $\frac{d+1}{p} - \frac{d+1}{q}-1 \le 0$; (C4) $\frac{1}{q} \ge \frac{3}{p}-1$;

(2) when $c \neq 0$, $p, q$ satisfy (C1)  $\frac{1}{q} \le \frac{1}{p}$; (C2) $\frac{1}{q} \ge \frac{1}{2p}$; (C5) $\frac{1}{q} \ge \frac{d+1}{p}-1$.
\end{thm}
 \begin{proof}
 (1) In the case $c=0$, we first prove (C1). We choose $S_{1}= [-1,1] \times [-2^{k}, 2^{k}]$ for some large $k$, and $D_{1}= [0,1] \times [0, 2^{k}]$. For each $y =(y_{1}, y_{2}) \in D_{1}$ and $x \in [0, 1]$, it is obvious that $(y_{1}-x, y_{2}-x^{d}) \in S_{1}$.

 Let $f = \chi_{S_{1}}$, where $\chi_{S_{1}}$ denotes the characteristic function of $S_{1}$. Then for each $y = (y_{1}, y_{2}) \in D_{1}$,
\[ \int_{0}^{1}\chi_{S_{1}}(y_{1}-x, y_{2}-x^{d})dx \ge 1. \]
Therefore
\[ \biggl\|\int_{0}^{1}\chi_{S_{1}}(y_{1}-x, y_{2}-x^{d})dx \biggl\|_{L^{q}(\mathbb{R}^{2})} \ge \biggl\|\int_{0}^{1}\chi_{S_{1}}(y_{1}-x, y_{2}-x^{d})dx \biggl\|_{L^{q}(D_{1})} \approx 2^{\frac{k}{q}}, \]
also
\[\|f\|_{L^{p}(\mathbb{R}^{2})} \approx 2^{\frac{k}{p}}.\]
Then (\ref{neceinequality}) will yield
\[2^{\frac{k}{q}} \lesssim 2^{\frac{k}{p}},\]
which is true only if (C1) holds since $k$ can be sufficiently large.

To prove (C2), we denote the $2^{-k}$ neighborhood of the curve $\{(-x,-x^{d}): x \in [0, 1]\}$ by $S_{2}$.
For each $y \in D_{2}:= B(0,2^{-k})$ and $x \in [0, 1]$, it is clear that $(y_{1}-x, y_{2}-x^{d}) \in S_{2}$.

Let $f = \chi_{S_{2}}$. Then for each $y  \in D_{2}$,
\[ \int_{0}^{1}\chi_{S_{2}}(y_{1}-x, y_{2}-x^{d})dx \ge 1, \]
which implies that
\[ \biggl\|\int_{0}^{1}\chi_{S_{2}}(y_{1}-x, y_{2}-x^{d})dx \biggl\|_{L^{q}(\mathbb{R}^{2})} \ge \biggl\|\int_{0}^{1}\chi_{S_{2}}(y_{1}-x, y_{2}-x^{d})dx \biggl\|_{L^{q}(D_{2})} \approx 2^{-\frac{2k}{q}}. \]
Meanwhile,
\[\|f\|_{L^{p}(\mathbb{R}^{2})} \approx 2^{-\frac{k}{p}}.\]
Then (\ref{neceinequality}) will imply
\[2^{-\frac{2k}{q}} \lesssim 2^{-\frac{k}{p}},\]
and (C2) follows since  $k$ may tend to infinity.

Now let us turn to prove (C3). We set $S_{3}= [-2^{-k}, 2 ^{-k}] \times [-2^{-dk}, 2 ^{-dk}]$, $D_{3}= [0, 2 ^{-k}] \times [0, 2 ^{-dk}]$. Then for each $y \in D_{3}$ and $x \in [0, 2^{-k}]$, we have $(y_{1}-x, y_{2}-x^{d}) \in S_{3}$.

Define $f = \chi_{S_{3}}$. For each $y \in D_{3}$,
\[  \int_{0}^{1}\chi_{S_{3}}(y_{1}-x, y_{2}-x^{d})dx \ge \int_{0}^{2^{-k}}\chi_{S_{3}}(y_{1}-x, y_{2}-x^{d})dx\ge 2^{-k}. \]
Hence
\[ \biggl\|\int_{0}^{1}\chi_{S_{3}}(y_{1}-x, y_{2}-x^{d})dx\biggl\|_{L^{q}(\mathbb{R}^{2})} \ge \biggl \|\int_{0}^{2^{-k}}\chi_{S_{3}}(y_{1}-x, y_{2}-x^{d})dx \biggl\|_{L^{q}(D_{3})} \approx 2^{-\frac{(d+1)k}{q}}2^{-k}, \]
and
\[\|f\|_{L^{p}(\mathbb{R}^{2})} \approx 2^{-\frac{(d+1)k}{p}}.\]
It follows from (\ref{neceinequality}) that
\[2^{-\frac{(d+1)k}{q}} 2^{-k} \lesssim 2^{-\frac{(d+1)k}{p}},\]
which implies (C3).

Finally we  show (C4) for the case $c =0$. For simplicity, we define
\[\mathcal{M}_{0}f(y): = \sup_{t\in[1,2]} \biggl|\int_{0}^{1}f(y_{1}-tx, y_{2}- tx^{d})dx\biggl|.\]
Denote by $e_{1}= (-\frac{1}{\sqrt{d^{2}+1}}, -\frac{d}{\sqrt{d^{2}+1}})$, $e_{2}= (-\frac{d}{\sqrt{d^{2}+1}}, \frac{1}{\sqrt{d^{2}+1}})$ the unit vectors. We choose $S_{4}= \{y: |y\cdot e_{1}| \le 20 \cdot 2^{-k}, |y\cdot e_{2}| \le 20 \cdot 2^{-2k} \}$, and a subset  $\{t_{i}\} \subset [1,2]$ such that $|t_{i}- t_{i+1}| = 100 2^{-2k}$ for each $i$. So the number of the elements in $\{t_{i}\}$ is equivalent to $2^{2k}$. For each $t_{i}$,  we define $D_{t_{i}}: =\{y: (y-(t_{i},t_{i}))\cdot e_{1} \le 2^{-k}, (y-(t_{i},t_{i}))\cdot e_{2} \le 2^{-2k} \} $.

For each $y \in D_{t_{i}}$ and $x \in [1-(d^{2} + 1)^{-1/2} \cdot 2^{-k}, 1]$, Taylor's expansion implies
\[x^{d}= 1 + d(x-1) + \frac{d(d-1)}{2}(x-1)^{2} + \mathcal{O}(|x-1|^{3}),\]
which yields
\[ \biggl| ((t_{i}x, t_{i}x^{d})-(t_{i}, t_{i})) \cdot e_{1}\biggl| =\biggl| \frac{t_{i}}{\sqrt{d^{2} + 1}} [(x- 1) + d^{2}(x-1) + \mathcal{O}((x-1)^{2}) ]\biggl| \le 10 \cdot 2^{-k},\]
and
\begin{align}
 &\biggl| ((t_{i}x, t_{i}x^{d})-(t_{i}, t_{i})) \cdot e_{2}\biggl| \nonumber\\
 &=\biggl| \frac{t_{i}}{\sqrt{d^{2} + 1}} [-d(x- 1) + d (x-1) + \frac{d(d-1)}{2}(x-1)^{2} + O(|x-1|^{3})) ]\biggl| \le 10 \cdot 2^{-2k}.\nonumber
 \end{align}
By the triangle inequality,
\[ \biggl| ( y - (t_{i}x, t_{i}x^{d}) ) \cdot e_{1} \biggl| \le \biggl| ( y - (t_{i}, t_{i}) ) \cdot e_{1} \biggl| + \biggl| ( (t_{i},t_{i}) - (t_{i}x, t_{i}x^{d}) ) \cdot e_{1} \biggl|  \le 20 \cdot 2^{-k}, \]
and
\[ \biggl| ( y - (t_{i}x, t_{i}x^{d}) ) \cdot e_{2} \biggl| \le \biggl| ( y - (t_{i}, t_{i}) ) \cdot e_{2} \biggl| + \biggl| ( (t_{i},t_{i}) - (t_{i}x, t_{i}x^{d}) ) \cdot e_{2} \biggl|  \le 20 \cdot 2^{-2k}. \]
This means that $ y - (t_{i}x, t_{i}x^{d}) \in S_{4}$. As a result, when $y \in D_{t_{i}}$, we have
\[\mathcal{M}_{0}(\chi_{S_{4}})(y) \ge   \int_{1-(d^{2} + 1)^{-1/2} \cdot 2^{-k} }^{1  }\chi_{S_{4}}(y_{1}-t_{i}x, y_{2}- t_{i}x^{d})dx \ge (d^{2} + 1)^{-1/2} \cdot 2^{-k}. \]
Moreover, it is clear that for $i \neq j$, $D_{t_{i}} \cap D_{t_{j}} = \emptyset.$
 Hence
\[\|\mathcal{M}_{0}(\chi_{S_{4}})\|_{L^{q}(\mathbb{R}^{2})} \ge \biggl( \sum_{i} \|\mathcal{M}_{0}(\chi_{S_{4}})\|^{q}_{L^{q}(D_{t_{i}})}  \biggl)^{1/q}  \gtrsim  2^{-k} 2^{-\frac{k}{q}}. \]
Notice that
\[\|\chi_{S_{4}}\|_{L^{p}(\mathbb{R}^{2})} \approx 2^{-\frac{3k}{p}}.\]
 Inequality (\ref{neceinequality}) thus implies
 \[2^{-k} 2^{-\frac{k}{q}} \lesssim 2^{-\frac{3k}{p}},\]
 then (C4) is obtained since $k$ can be sufficiently large.

(2) In the case $c\not=0$, we prefer to omit the proofs of (C1)-(C2) because they are very similar with  part (1). We only show (C5).
For simplicity, we may choose $c= 1$. Let
 \[\mathcal{M}_{1}f(y): = \sup_{t\in[1,2]} \biggl|\int_{0}^{1}f(y_{1}-tx, y_{2}- t(x^{d} +1))dx\biggl|.\]
 We set $S_{5}= [-10  \cdot 2^{-k}, 10 \cdot 2^{-k}] \times [-10  \cdot 2^{-dk}, 10 \cdot 2^{-dk}]$, and choose a subset  $\{t_{i}\} \subset [1,2]$ such that $|t_{i}- t_{i+1}| = 100 \cdot 2^{-dk}$ for each $i$. So the number of the elements in $\{t_{i}\}$ is equivalent to $2^{dk}$. For each $t_{i}$.  we define $D_{t_{i}}: =[0, 2^{-k}] \times [t_{i}, t_{i} + t_{i}^{1-d}2^{-dk}] $.

 If $y \in D_{t_{i}}$ and $x \in [0, 2^{-k}]$, then
 \[|y_{1}-x| \le 2^{-k}, |y_{2}-t_{i}^{1-d}x^{d}| \le 2^{-dk}, \]
 that is $(y_{1}-x, y_{2}-t_{i}^{1-d}x^{d}) \in S_{5}$.

 Let $f = \chi_{S_{5}}$, then for every $y \in D_{t_{i}}$, we have
 \[\mathcal{M}_{1}(\chi_{S_{5}})(y) \ge \frac{1}{2}  \int_{0}^{\frac{1}{2} }\chi_{S_{5}}(y_{1}- x, y_{2}- t_{i}^{1-d}x^{d})dx \ge  \frac{1}{2}  \int_{0}^{2^{-k} }\chi_{S_{5}}(y_{1}- x, y_{2}- t_{i}^{1-d}x^{d})dx \gtrsim 2^{-k}. \]
Notice that if $i \neq j$, then $D_{t_{i}} \cap D_{t_{j}} = \emptyset.$ Therefore,
\[\|\mathcal{M}_{1}(\chi_{S_{5}})\|_{L^{q}(\mathbb{R}^{2})} \ge \biggl( \sum_{i} \|\mathcal{M}_{1}(\chi_{S_{5}})\|^{q}_{L^{q}(D_{t_{i}})}  \biggl)^{1/q}  \gtrsim  2^{-k} 2^{-\frac{k}{q}}. \]
Also,
\[\|\chi_{S_{5}}\|_{L^{p}(\mathbb{R}^{2})} \approx 2^{-\frac{(d+1)k}{p}}.\]
 Then inequality (\ref{neceinequality})  implies
 \[2^{-k} 2^{-\frac{k}{q}} \lesssim 2^{-\frac{(d+1)k}{p}},\]
 and we arrive at (C5).
\end{proof}

\textbf{Proof of Theorem \ref{hisomaintheorem}.}  The proof of Theorem \ref{hisomaintheorem} is very similar with that of Theorem \ref{maintheorem1}, so we only highlight some key points here. Similarly, we will consider the operator given by
\[\widetilde{A_t^{k}}f(y):=\int_{\mathbb{R}}  f(y_1-t^{a_{1}}x,y_2-t^{a_{2}}(x^d\phi(\frac{x}{2^{k}}) +c))\tilde{\rho}(x)dx.\]
Then we choose a non-negative function $\beta\in C_0^{\infty}(\mathbb{R})$ such that supp $\beta\subset[1/2,2]$ and $\sum_{j\in\mathbb{Z}}\beta(2^{-j}r)=1$ for $r>0$. Define the dyadic operators
\begin{equation}
\widetilde{A_{t,j}^k}f(y)=\frac{1}{(2\pi)^2}\int_{{\mathbb{R}}^2}e^{i\xi\cdot y} e^{ict^{a_{2}} \xi_{3}}  \widehat{d\mu }(\delta_{t}\xi)\beta(2^{-j}|\delta_{t}\xi|)\hat{f}(\xi)d\xi,
\end{equation}
and denote by $\widetilde{\mathcal{M}_{j}^k}$ the corresponding maximal operator. Here $\delta_{t}(\xi) = (t^{a_{1}} \xi_{1}, t ^{a_{2}}\xi_{2})$.

It is straightforward to treat $\sum_{j \le 0} \widetilde{\mathcal{M}_{j}^k}$ in the same spirit of Theorem \ref{maintheorem1}. For $j \ge 1$, by stationary phase method, the main contribution term is given by
\begin{equation*}
A^k_{t,j}f(y):=\int_{{\mathbb{R}}^2}e^{i(\xi\cdot y -ct^{a_{2}}2^{dk}\xi_{2}-t^{a_{2}} \xi_2\tilde{\Phi}(s,\delta))}\chi_{k,d,m}(t^{a_{1}-a_{2}}\frac{\xi_1}{ \xi_2})
\frac{A_{k,d,m}( \delta_{t}\xi)}{(1+| \delta_{t}\xi|)^{1/2}}
\beta(2^{-j}| \delta_{t}\xi|)\hat{f}(\xi)d\xi,
\end{equation*}
where $\delta =2^{-k}$, $\chi_{k,d,m}$ is a smooth function supported in the conical region $[c_{k,d,m},\widetilde{c_{k,d,m}}]$, for certain non-zero constants $c_{k,d,m}$ and $\widetilde{c_{k,d,m}}$ dependent only on $k, m$ and $d$. $A_{k,d,m}$ is a symbol of order zero in $\xi$ and $\{A_{k,d,m}(\delta_t\xi)\}_k$ is contained in a bounded subset of symbol of order zero. We may assume $\phi(0)=1/d$, then the phase function can be written as
\begin{align}\label{mainterm2}
-t^{a_{2}}\xi_2\tilde{\Phi}(s,\delta)=(d-1)t^{\frac{da_{1}-a_{2}}{d-1}}\xi_{2} \biggl(-\frac{\xi_{1}}{d\xi_{2}}\biggl)^{d/(d-1)}-t^{\frac{(a_{1}-a_{2})m}{d-1} + a_{2}} \xi_{2}\frac{\delta^m\phi^{(m)}(0)}{ m!}
\biggl(-\frac{\xi_1}{d\xi_2}\biggl)^{\frac{d+m}{d-1}}+ R(t,\xi,\delta,d),
\end{align}
where $R(t,\xi,\delta,d)$ is homogeneous of degree one in $\xi$ and has at least $m+1$ power of $\delta$. We notice that \eqref{mainterm2} is a small perturbation of $(d-1)t^{\frac{da_{1}-a_{2}}{d-1}}\xi_{2} \biggl(-\frac{\xi_{1}}{d\xi_{2}}\biggl)^{d/(d-1)}$  since $da_{1} \neq a_{2}$. Then we can apply the local smoothing estimates as we did in the proof of Theorem \ref{maintheorem1} to complete the proof of Theorem \ref{hisomaintheorem}.

\subsubsection{The case when $da_{1} = a_{2}$: homogeneous curves}
This subsection is dedicated to the proof of Theorem \ref{maitheorem3}. First of all, we have the following transference theorem.
\begin{thm}\label{transference}
Let the averaging operator $A_t$ be defined in \eqref{averagefinite2} and
\[Tf(y):= \int_{0}^{1}f(y_{1}-x, y_{2}-x^{d})dx, \quad d \ge 2.\]
Then for each $q \ge p \ge 1$, $C_{p,q} \ge 1$,
\begin{equation}\label{maximal operator}
\|\sup_{t \in [1,2]} |A_{t}f|\|_{L^{q}(\mathbb{R}^{2})}\leq  C_{p,q} \|f\|_{L^{p}(\mathbb{R}^{2})}, \hspace{0.5cm}f\in C_0^{\infty}(\mathbb{R}^2),
\end{equation}
if and only if
\begin{equation}\label{convolution operator}
\|Tf\|_{L^{q}(\mathbb{R}^{2})}\leq  C_{p,q} \|f\|_{L^{p}(\mathbb{R}^{2})}, \hspace{0.5cm}f\in C_0^{\infty}(\mathbb{R}^2).
\end{equation}
\end{thm}

\begin{proof}
On one hand, since for each $y \in \mathbb{R}^{2}$, $|Tf(y)| \le \sup_{t \in [1,2]} |A_{t}f(y)|$.
It is obvious that inequality (\ref{convolution operator}) yields inequality (\ref{maximal operator}).

On the other hand, assume that inequality (\ref{convolution operator}) holds true, we will prove inequality (\ref{convolution operator}). Changing variables implies
\begin{align}
\sup_{t \in [1,2]} |A_{t}f(y)| &\le \sup_{t \in [1,2]} \frac{1}{t} \int_{0}^{t}|f|(y_{1}-x, y_{2}-x^{d})dx  \le  \int_{0}^{2}|f|(y_{1}-x, y_{2}-x^{d})dx \nonumber\\
&= 2 \int_{0}^{1}|f|(y_{1}-2x, y_{2}-2^{d}x^{d})dx. \nonumber
\end{align}
Therefore, we have
\begin{align}
\biggl\|\sup_{t \in [1,2]} |A_{t}f|\biggl\|_{L^{q}(\mathbb{R}^{2})}  &\lesssim 2^{\frac{d+1}{q}}  \biggl \|\int_{0}^{1}|f|(2y_{1}-2x, 2^{d}y_{2}-2^{d}x^{d})dx \biggl\|_{L^{q}(\mathbb{R}^{2})} \nonumber\\
&\sim 2^{\frac{d+1}{q}}  \biggl \|T(|f|(2\cdot, 2^{d} \cdot))(y) \biggl\|_{L^{q}(\mathbb{R}^{2})} \nonumber\\
&\lesssim 2^{\frac{d+1}{q}}  C_{p,q}  \| f(2y_1, 2^{d}y_2) \|_{L^{p}(\mathbb{R}^{2})} \nonumber\\
&\sim 2^{\frac{d+1}{q}-\frac{d+1}{p} }  C_{p,q}  \| f \|_{L^{p}(\mathbb{R}^{2})}. \nonumber
\end{align}
Then we arrive at inequality (\ref{maximal operator}).
\end{proof}
\begin{rem}
We would like to point out that the particular choice of dilations enables the change of variables and thus the validation of the transference theorem.
\end{rem}

According to Theorem \ref{transference}, we are reduced to get the $L^{p}\rightarrow L^{q}$ estimate for the operator $T$.
This problem was considered by Iosevich and Sawyer \cite{ios3}. It follows from the results in \cite{ios3} that $\|T\|_{L^{p}\rightarrow L^{q}} < \infty$ provided that $(\frac{2}{p},\frac{1}{q}) \in \Delta_{3}= \{(\frac{2}{p},\frac{1}{q}): \frac{1}{2p} < \frac{1}{q} \le \frac{1}{p} , \frac{1}{q}  > \frac{2}{p} -1, \frac{1}{q} > \frac{1}{p} - \frac{1}{d+1}\} \cup \{(0, 0), (1, 1)\}$. Moreover, according to [Lemma 2, \cite{ios3}], the region for $p, q$ here is almost sharp up to the endpoint. This completes the proof of Theorem \ref{maitheorem3}.

\subsubsection{The case when $da_{1} = a_{2}$: finite type curves}
This subsection is devoted to the proof of Theorem \ref{hhisomaintheorem}. We first prove the theorem for $d = 2$, and then explain how to extend it to $d >2$.

\textbf{The case when $d =2$. }We choose $B>0$ very small and  $\tilde{\rho}\in C_0^{\infty}(\mathbb{R})$ such that  $\supp \tilde{\rho}\subset\{x:B/2\leq|x|\leq 2B\}$  and $\sum_k\tilde{\rho}(2^kx)=1$ for $x\in \mathbb{R}$.

Let
\begin{align*}
A_tf(y):&=\int f(y_1-tx,y_2-t^2x^2\phi(x))\eta(x)dx 
=\sum_kA_t^kf(y),
\end{align*}
where
\begin{equation*}
A_t^kf(y):=\int f(y_1-tx,y_2-t^2x^2\phi(x))\tilde{\rho}(2^kx)\eta(x)dx.
\end{equation*}
Since $\eta$ is supported in a sufficiently small neighborhood of the origin, we only need to consider $k>0$ sufficiently large.

Considering  isometric operator on $L^p(\mathbb{R}^2)$ defined by $T_kf(x_1,x_2)=2^{3k/p}f(2^kx_1,2^{2k}x_2)$, we only need to prove the following estimate
\begin{equation*}
\sum_k2^{3k(\frac{1}{p} - \frac{1}{q})-k}\left\|\sup_{t \in [1,2]}|\widetilde{A_t^k}|\right\|_{L^p\rightarrow L^q}\leq C_{p,q},
\end{equation*}
where
\begin{equation*}
\widetilde{A_t^k}f(y):=\int f(y_1-tx,y_2-t^2x^2\phi(\frac{x}{2^k}))\tilde{\rho}(x)\eta(2^{-k}x)dx.
\end{equation*}

By means of the Fourier inversion formula, we have
\begin{align*}
\widetilde{A_t^k}f(y)&=\frac{1}{(2\pi)^2}\int_{{\mathbb{R}}^2}e^{i\xi\cdot y}\int_{\mathbb{R}}e^{-i(t\xi_1x+t^2\xi_2x^2\phi(\frac{x}{2^k}))}\tilde{\rho}(x)\eta(2^{-k}x)dx\hat{f}(\xi)d\xi
\\
&=\frac{1}{(2\pi)^2}\int_{{\mathbb{R}}^2}e^{i\xi\cdot y}\widehat{d\mu_k}(\delta_t\xi)\hat{f}(\xi)d\xi,
\end{align*}
where
\begin{equation*}
\widehat{d\mu_k}(\xi):=\int_{\mathbb{R}}e^{-i(\xi_1x+\xi_2x^2\phi(\frac{x}{2^k}))}\tilde{\rho}(x)\eta(2^{-k}x)dx.
\end{equation*}

We choose a non-negative function $\beta\in C_0^{\infty}(\mathbb{R})$ such that supp $\beta\subset[1/2,2]$ and $\displaystyle \sum_{j\in\mathbb{Z}}\beta(2^{-j}r)=1$ for $r>0$. Define the dyadic operators
\begin{equation*}
\widetilde{A_{t,j}^k}f(y)=\frac{1}{(2\pi)^2}\int_{{\mathbb{R}}^2}e^{i\xi\cdot y}\widehat{d\mu_k}(\delta_t\xi)\beta(2^{-j}|\delta_t\xi|)\hat{f}(\xi)d\xi,
\end{equation*}
and denote by $\widetilde{\mathcal{M}_{j}^k}$ the corresponding maximal operator.  As a consequence of this decomposition, we have that
\begin{equation*}
\sup_{t \in [1,2]}|\widetilde{A_t^k}f(y)|\leq \widetilde{\mathcal{M}^{k,0}}f(y)+\sum_{j\geq 1}\widetilde{\mathcal{M}_{j}^k}f(y), \hspace{0.2cm}\textmd{for}\hspace{0.2cm}y\in \mathbb{R}^2,
\end{equation*}
where
\begin{equation*}
\widetilde{\mathcal{M}^{k,0}}f(y):=\sup_{t \in [1,2]}|\sum_{j\leq 0}\widetilde{A_{t,j}^k}f(y)|.
\end{equation*}

Theorem \ref{hhisomaintheorem} follows from the estimates (E1)-(E3) below. \\
\textbf{(E1)} For $q \ge p \ge 1$, we have
\begin{align*}
\|\widetilde{\mathcal{M}^{k,0}}f\|_{L^{q}} \lesssim \|f\|_{L^{p}}.
\end{align*}
\textbf{(E2)} Suppose that $\frac{1}{2p} < \frac{1}{q} \le \frac{1}{p}$, $\frac{1}{q} \ge \frac{3}{p}-1$. Then for any $\epsilon >0$, there holds
\begin{equation*}
\|\widetilde{\mathcal{M}_{j}^{k}}f\|_{L^{q}}\leq C_{p,q} 2^{j(1+\epsilon)(1/p-1/q)-(j\wedge mk)/q}\|f\|_{L^p}.
\end{equation*}
\textbf{(E3)} For all $j > mk$, and $p,q$ satisfying $\frac{1}{2p}< \frac{1}{q} \le  \frac{3}{5p}$, $\frac{3}{q} \le 1-\frac{1}{p} $,  we have
\begin{align*}
\|\widetilde{\mathcal{M}_{j}^{k}}f\|_{L^q}
\le C   2^{(\frac{1}{p} -\frac{2}{q}  +\epsilon)j} \|f\|_{L^p}.
\end{align*}

Indeed, if (E1)-(E3) hold true, since $3(\frac{1}{p} - \frac{1}{q}) < 1$ for $p,q$   from  Theorem \ref{hhisomaintheorem} (when $d=2$), then it follows that
\[\sum_k2^{3k(\frac{1}{p} - \frac{1}{q})-k}\|\widetilde{\mathcal{M}^{k,0}}\|_{L^p\rightarrow L^q} \leq C_{p,q}.\]
Next we split the set of $j$ into two  parts, namely $j>mk$ and $1  \le j \le mk$.

When $1 \le j\leq mk$,
by  (E2), if  $\frac{1}{2p} < \frac{1}{q} \le \frac{1}{p}$, $\frac{1}{q} \ge \frac{3}{p}-1$, for  $\epsilon >0$ sufficiently small, we get
\begin{align}
\sum_k 2^{3k(\frac{1}{p} - \frac{1}{q} )-k} \sum_{1 \le j\leq mk}\|\widetilde{\mathcal{M}_{j}^{k}}f\|_{L^q} &\leq C_{p,q} \sum_k 2^{3k(\frac{1}{p} - \frac{1}{q})-k} \sum_{1 \le j \le mk} 2^{j(\frac{1}{p}-\frac{2}{q}) + \epsilon j} \|f\|_{L^p} \nonumber\\
&\leq C_{p,q} \|f\|_{L^p}. \nonumber
\end{align}

When $ j\ge mk$,
 it follows from (E3) that if $p,q$ satisfy $\frac{1}{2p}< \frac{1}{q} \le  \frac{3}{5p}$, $\frac{3}{q} \le 1-\frac{1}{p} $,  $\epsilon >0$ sufficiently small, there holds
\begin{align}
\sum_k 2^{3k(\frac{1}{p} - \frac{1}{q})-k} \sum_{j > mk}\|\widetilde{M_{j}^{k}}f\|_{L^q} &\leq C_{p,q} \sum_k 2^{3k(\frac{1}{p} - \frac{1}{q})-k}   \sum_{j > mk} 2^{(\frac{1}{p} -\frac{2}{q}  +\epsilon)j} \|f\|_{L^p} \nonumber\\
&\leq C_{p,q} \|f\|_{L^p}. \nonumber
\end{align}

These three cases  imply that,  for $p,q$ satisfying $\frac{1}{2p} <\frac{1}{q} \le \frac{3}{5p}$, $\frac{3}{q} \le 1-\frac{1}{p} $, there holds
\begin{equation}\label{equ:planem=1small}
\|\mathcal{M}f\|_{L^{q}}\leq  C_{p,q} \|f\|_{L^{p}}.
\end{equation}
By Theorem 1.1 in \cite{WL}, for each $p>2$,
\begin{equation}\label{diag}
\|\mathcal{M}f\|_{L^{p}}\leq  C_{p} \|f\|_{L^{p}}, \hspace{0.5cm}f\in C_0^{\infty}(\mathbb{R}^2).
\end{equation}
Thanks to the inequalities (\ref{equ:planem=1small}), (\ref{diag}) and the Riesz interpolation theorem, the case $d= 2$ of Theorem \ref{hhisomaintheorem} follows.

It remains to prove (E1), (E2) and (E3). 

\textbf{Proof of (E1).}
Note that $\widetilde{\mathcal{M}^{k,0}}f(y)=\sup_{t \in [1,2]}|f*K_{\delta_{t^{-1}}}(y)|$, where $K_{\delta_{t^{-1}}}(x)=t^{-3}K(\frac{x_1}{t},\frac{x_2}{t^2})$ and
\begin{equation*}
K(y):=\int_{{\mathbb{R}}^2}e^{i\xi\cdot y}\widehat{d\mu_k}(\xi)\rho(|\xi|)d\xi,
 \end{equation*}
where $\rho\in C_0^{\infty}(\mathbb{R})$ is supported in $[0,2]$. Similarly as in the kernel estimate (\ref{kernelestimate1}), we obtain that
\begin{equation*}
|K(y)|\leq C_N (1+|y|)^{-N}.
\end{equation*}
When $q \ge p \ge1$, by Young's inequality, we have
\begin{align}
\|\widetilde{\mathcal{M}^{k,0}}f\|_{L^{q}} &= \|\sup_{t \in [1,2]}|f*K_{\delta_{t^{-1}}}|\|_{L^{q}} \nonumber\\
&\le \biggl\|\frac{C_{N}}{(1+|\cdot|)^{N}} * |f|\biggl\|_{L^{q}} \nonumber\\
&\lesssim \|f\|_{L^{p}}.\nonumber
\end{align}
Then (E1) is proved.

\textbf{Proof of (E2) and (E3).}
We will first consider
\begin{equation*}
\widehat{d\mu_k}(\delta_t\xi)=\int_{\mathbb{R}}e^{-it^2\xi_2(-sx+x^2\phi(\delta x))}\tilde{\rho}(x)\eta(\delta x)dx,
\end{equation*}
where $\delta=2^{-k}$ and
\begin{equation*}
s:=s(\xi,t)=-\frac{\xi_1}{t\xi_2}, \hspace{0.3cm}\textrm{for} \hspace{0.2cm} \xi_2\neq 0.
\end{equation*}

Let
\begin{equation*}
\Phi(s,x,\delta)=-sx+x^2\phi(\delta x),
\end{equation*}
then we have
\begin{equation*}
\partial_x\Phi(s,x,\delta)=-s+2x\phi(\delta x)+x^2\delta\phi'(\delta x)
\end{equation*}
and
\begin{equation*}
\partial_x^2\Phi(s,x,\delta)=2\phi(\delta x)+4x\delta\phi'(\delta x)+x^2\delta^2\phi''(\delta x).
\end{equation*}

Since $k$ is sufficiently large and $\phi(0)\neq 0$, then the implicit function theorem implies that there exists a smooth solution $x_c=\tilde{q}(s,\delta)$ of the equation $\partial_x\Phi(s,x,\delta)=0$. For the sake of simplicity, we may assume $\phi(0)=1/2$. By Taylor's expansion, the phase function can be written as
\begin{equation}\label{Phase}
-t^2\xi_2\tilde{\Phi}(s,\delta)=\frac{\xi_1^2}{2\xi_2}+(-1)^{m+1}\frac{{\phi}^{(m)}(0)}{m!}\delta^{m}\frac{\xi_1^{m+2}}{t^m\xi_2^{m+1}} + R(t,\xi,\delta),
\end{equation}
where $R(t,\xi,\delta)$ is homogeneous of degree one in $\xi$.
Note that $-t^2\xi_2\tilde{\Phi}(s,\delta)$ can be considered as a small perturbation of  \[\frac{\xi_1^2}{2\xi_2}+(-1)^{m+1}\frac{{\phi}^{(m)}(0)}{m!}\delta^{m}\frac{\xi_1^{m+2}}{t^m\xi_2^{m+1}} .\]

By applying the method of stationary phase, we have
 \begin{equation}\label{fouriermeasuredecomposition3}
\widehat{d\mu_{k,m}}(\delta_t\xi)=e^{-it^2\xi_2\tilde{\Phi}(s,\delta)}\chi_{k,m}(\frac{\xi_1}{t\xi_2})
\frac{A_{k,m}(\delta_t\xi)}{(1+|\delta_t\xi|)^{1/2}}+B_{k,m}(\delta_t\xi),
\end{equation}
where  $\chi_{k,m}$ is a smooth function supported in $[c_{k,m},\widetilde{c_{k,m}}]$, for certain non-zero constants $c_{k,m}$ and $\widetilde{c_{k,m}}$ dependent only on $k$ and $m$. $A_{k,m}$ is a symbol of order zero in $\xi$ and $\{A_{k,m}(\delta_t\xi)\}_{k,m}$ is contained in a bounded subset of symbol of order zero. 
More precisely, for arbitrary $t \in [1,2]$,
 \begin{equation*}\label{symbol123}
 |D_{\xi}^{\alpha}A_{k,m}(\delta_t\xi)|\leq C_{\alpha}(1+|\xi|)^{-\alpha},
\end{equation*}
where $C_{\alpha}$ does not depend on $k$ and $m$. Furthermore, $B_{k,m}$ is a remainder term and satisfies for arbitrary $t\in [1,2]$,
\begin{equation*}
 |D_{\xi}^{\alpha}B_{k,m}(\delta_t\xi)|\leq C_{\alpha,N}(1+|\xi|)^{-N},
\end{equation*}
where $C_{\alpha,N}$ are admissible constants that do not depend on $k$ and $m$.

First, let us consider the remainder part of (\ref{fouriermeasuredecomposition3}). Set
\begin{equation*}
\mathcal{M}_{j}^{k,0}f(y):=\sup_{t\in[1,2]}\left|\frac{1}{(2\pi)^2}\int_{{\mathbb{R}}^2}e^{i\xi\cdot y}B_{k,m}(\delta_t\xi)\beta(2^{-j}|\delta_t\xi|)\hat{f}(\xi)d\xi \right|.
\end{equation*}
Integration by parts implies $|(B_{k,m}\beta(2^{-j}\cdot))^{\vee}(x)|\leq C_N 2^{-jN}(1+|x|)^{-N}$. Therefore,
\begin{align}
\mathcal{M}_{j}^{k,0}f(y) &\le \sup_{t\in[1,2]}\frac{C_{N}2^{-jN}}{(2\pi)^2t^{3}}\int_{{\mathbb{R}}^2} \frac{|f(x)|}{(1+|\delta_{t^{-1}}(y-x)|)^{N}} dx \nonumber\\
 &\lesssim 2^{-jN}\int_{{\mathbb{R}}^2} \frac{|f(x)|}{(1+|y-x|)^{N}} dx. \nonumber
\end{align}

Set
\begin{equation}\label{mainterm3}
A_{t,j}^kf(y):=\frac{1}{(2\pi)^2}\int_{{\mathbb{R}}^2}e^{i(\xi \cdot y-t^2\xi_2\tilde{\Phi}(s,\delta))}\chi_{k,m}(\frac{\xi_1}{t\xi_2})
\frac{A_{k,m}(\delta_t\xi)}{(1+|\delta_t\xi|)^{1/2}}\beta(2^{-j}|\delta_t \xi|)\hat{f}(\xi)d\xi.
\end{equation}
Denote by $\mathcal{M}_{j}^{k,1}$ the corresponding maximal operator over $[1,2]$.

 Since $\tilde{\Phi}(s,\delta)$ is homogeneous of degree zero in $\xi$ and $\frac{\xi_1}{\xi_2}\approx 1$,  the operator $M_{j}^{k,1}$ can be localized by the same method that we used in  the proof of Theorem \ref{maintheorem1}.
Hence we can choose $\rho_1\in C_0^{\infty}(\mathbb{R}^2\times[\frac{1}{2},4])$ and define
\begin{equation*}
\widetilde{\mathcal{M}_{j}^{k,1}}f(y):=\sup_{t\in[1,2]}\left|\rho_1(y,t)A_{t,j}^kf(y)\right|.
\end{equation*}

By Lemma 2.12 in \cite{WL}, we have
\begin{equation}\label{p-p}
\|\widetilde{\mathcal{M}_{j}^{k,1}}f\|_{L^p}\leq C_{p} 2^{- (j \wedge mk)/p}\|f\|_{L^p}, \hspace{0.5cm}2\leq p\leq \infty.
\end{equation}
Then (E2) follows from  the Riesz interpolation theorem between (\ref{p-p})  and   Lemma \ref{lemmaL^2} below.

\begin{lem}\label{lemmaL^2}
For any $\epsilon >0$, we have
\begin{equation}\label{1-infty}
\|\widetilde{\mathcal{M}_{j}^{k,1}}f\|_{L^ \infty }\leq C_{\epsilon} 2^{j(1+\epsilon)}\|f\|_{L^1}.
\end{equation}
\end{lem}

\begin{proof}
We introduce the angular decomposition of the set $\{\xi \in \mathbb{R}^{2}: \frac{\xi_1}{\xi_2}\approx 1\}$.  For each positive integer $j$, we consider a roughly equally spaced set of points with grid length $2^{-j/2}$ on the unit circle $S^1$; that is, we fix a collection $\{\kappa_{j}^{\nu}\}_{\nu}$ of real numbers, that satisfy:

$(a)$ $|\kappa_j^{\nu}-\kappa_j^{\nu'}|\geq 2^{-j/2}$, if $\nu\neq\nu'$;

$(b)$ if $\xi\in \{\xi \in \mathbb{R}^{2}: \frac{\xi_1}{\xi_2}\approx 1\}$, then there exists a $\kappa_j^{\nu}$ so that $\biggl|\frac{\xi_1}{\xi_2}-\kappa_j^{\nu}\biggl|<2^{-j/2}$.

Let $\Gamma_j^{\nu}$ denote the corresponding cone in the $\xi$-space
\begin{equation*}
\Gamma_j^{\nu}=\{\xi \in \mathbb{R}^{2}: \biggl|\frac{\xi_{1}}{\xi_{2}}-\kappa_j^{\nu}\biggl|\leq2\cdot 2^{-j/2}\}.
\end{equation*}
We can construct an associated partition of unity:
$\chi_j^{\nu}$ is  homogeneous of degree zero in $\xi$ and supported in $\Gamma_j^{\nu}$, with
\begin{equation*}
\sum_{\nu}\chi_j^{\nu}(\xi)=1 \hspace{0.5cm}\textrm{for}\hspace{0.2cm}\textrm{ all}\hspace{0.2cm} \xi\in \{\xi \in \mathbb{R}^{2}: \frac{\xi_1}{\xi_2}\approx 1\} \hspace{0.2cm}\textrm{and}\hspace{0.2cm} \textrm{all}\hspace{0.2cm} j,
\end{equation*}
and
\begin{equation*}
|\partial_{\xi}^{\alpha}\chi_j^{\nu}(\xi)|\leq A_{\alpha}2^{|\alpha|j/2}|\xi|^{-|\alpha|}.
\end{equation*}

Hence, in order to establish (\ref{1-infty}), notice that
\begin{align}
\|\widetilde{\mathcal{M}_{j}^{k,1}}f\|_{L^{\infty}} &= \biggl \|\sup_{t \in [1,2]}\biggl |\int_{\mathbb{R}^{2}} K_{t}(y,x)f(x)dx \biggl |\biggl \|_{L^{\infty}} \nonumber\\
 &\le \biggl \|\sup_{t \in [1,2]}\int_{\mathbb{R}^{2}} \sum_{\nu}{ |K_t^{\nu}(y,x)|}|f(x)|dx \biggl \|_{L^{\infty}}, \nonumber
\end{align}
where
\begin{equation*}
K_t(y,x)=\rho_1(y,t)\int_{{\mathbb{R}}^2}e^{i(\xi \cdot (y-x)-t^2\xi_2\tilde{\Phi}(s,\delta))}\widetilde{A_{k,m}}(\xi,t)\beta(2^{-j}|\delta_t \xi|)d\xi,
\end{equation*}
and
\begin{equation*}
K_t^{\nu}(y,x)=\rho_1(y,t)\int_{{\mathbb{R}}^2}e^{i(\xi \cdot (y-x)-t^2\xi_2\tilde{\Phi}(s,\delta))}\widetilde{A_{k,m}}(\xi,t)\beta(2^{-j}|\delta_t \xi|)\chi_j^{\nu}(\xi)d\xi,
\end{equation*}
$\widetilde{A_{k,m}}(\xi,t)=\chi_{k,m}(\frac{\xi_1}{t\xi_2})
\frac{A_{k,m}(\delta_t\xi)}{(1+|\delta_t\xi|)^{1/2}}$. If we can show that
for fixed  $y$, $t$ and $\epsilon >0$, we have
\begin{equation}\label{sum bound}
\sum_{\nu}{ |K_t^{\nu}(y,x)|} \le 2^{(1+\epsilon)j},
\end{equation}
uniformly for $x \in \mathbb{R}^{2}$, then inequality (\ref{1-infty}) follows.

In fact, it is not hard to check that for each $\nu$,
\begin{align}\label{kernalL^infty}
|K_t^{\nu}(y,x)| \leq 2^{j} \frac {C _{N}}{(1+ 2^{\frac{j}{2}}|x_{1}-c_{1}(y,t,\kappa_{j}^{\nu},\delta)|  + 2^{j}|x_{2} +\kappa_{j}^{\nu}x_{1}-c_{2}(y,t,\kappa_{j}^{\nu},\delta)|   )^{N}}, \nonumber
\end{align}
where $C$ does not depend on $t$, $j$, $k$ and $\nu$,
\[c_{1}(y,t,\kappa_{j}^{\nu},\delta) = \kappa_{j}^{\nu}+y_{1}+\delta^{m} (\partial_{1}\bar{R})(\kappa_{j}^{\nu},t,\delta),\]
\[c_{2}(y,t,\kappa_{j}^{\nu},\delta) =y_{2} +\kappa_{j}^{\nu}y_{1} +\frac{1}{2}(\kappa_{j}^{\nu})^{2}+\delta^{m} \bar{R}(\kappa_{j}^{\nu},t,\delta),\]
here $\bar{R}(\frac{\xi_{1}}{\xi_{2}},t,\delta) =(-1)^{m+1}\frac{{\phi}^{(m)}(0)}{m!} \frac{\xi_1^{m+2}}{t^m\xi_2^{m+2}} + \frac{\delta}{\xi_{2}} R(t,\xi,\delta)$.

For fixed $y$, $t$, $j$, $k$, if one of
\[|x_{1}-c_{1}(y,t,\kappa_{j}^{\nu},\delta)| \ge 2^{-\frac{j}{2}+\epsilon j}\]
and
\[|x_{2} + \kappa_{j}^{\nu}x_{1}-c_{2}(y,t,\kappa_{j}^{\nu},\delta)| \ge 2^{-j+\epsilon j}\]
holds true for some $\epsilon >0$, then we have
\[|K_t^{\nu}(y,x)| \leq C_{N}2^{-Nj}.\]
 Inequality (\ref{sum bound}) follows  since $N$ can be sufficiently large. Therefore, we only need to consider the case when $(x_{1}, x_{2})$ satisfies
\[|x_{1}-c_{1}(y,t,\kappa_{j}^{\nu},\delta)| \le 2^{-\frac{j}{2}+\epsilon j},\]
\[|x_{2} + \kappa_{j}^{\nu}x_{1}-c_{2}(y,t,\kappa_{j}^{\nu},\delta)| \le 2^{-j+\epsilon j}.\]
It is obvious that for fixed $y$, $t$, $j$, $k$, if $\nu\neq\nu'$,
\[|c_{1}(y,t,\kappa_{j}^{\nu},\delta) -c_{1}(y,t,\kappa_{j}^{\nu'},\delta)| \ge 2^{-j/2} ,\]
which implies inequality (\ref{sum bound}). This completes the proof of Lemma \ref{lemmaL^2}.
\end{proof}

When $j>mk$, by Lemma \ref{lem:Lemma3},
\begin{equation}\label{well}
\begin{aligned}
&\|\widetilde{\mathcal{M}_{j}^{k,1}}f\|_{L^q}\leq C
2^{\frac{j}{q} -\frac{mk}{q} -\frac{j}{2}}
\\
&\quad\times \biggl(\int_{{\mathbb{R}}^2}\int^4_{1/2}\biggl|\rho_1(y,t)\int_{{\mathbb{R}}^2}e^{i(\xi \cdot y-t^2\xi_2\tilde{\Phi}(s,\delta))}2^{j/2}\widetilde{A}_{k,m}(\xi,t)\chi(\frac{\xi_1}{t\xi_2})\beta(2^{-j}|\delta_t \xi|)\hat{f}(\xi)d\xi\biggl|^qdtdy \biggl)^{
1/q'\times 1/q}\\
& \quad \times
\biggl(\int_{{\mathbb{R}}^2}\int^4_{1/2}\biggl|\frac{\partial}{\partial t}(\rho_1(y,t)\int_{{\mathbb{R}}^2}e^{i(\xi \cdot y-t^2\xi_2\tilde{\Phi}(s,\delta))}2^{mk-j/2}\widetilde{A}_{k,m}(\xi,t)\chi(\frac{\xi_1}{t\xi_2})\beta(2^{-j}|\delta_t \xi|)\hat{f}(\xi)d\xi)\biggl|^qdtdy \biggl)^{1/q^{2}}.
\end{aligned}
\end{equation}

In order to simplify the notation, we choose $\tilde{\chi}\in C_0^{\infty}([c_1,c_2])$ so that  $\tilde{\chi}(\frac{\xi_1}{\xi_2})\chi(\frac{\xi_1}{t\xi_2})=\chi(\frac{\xi_1}{t\xi_2})$
 for arbitrary $t\in [1/2,4]$ and $k$ sufficiently large. In a similar way we choose $\rho_0\in C_0^{\infty}((-10,10))$ such that $\rho_0(|\xi|)\beta(|\delta_t \xi|)=\beta(|\delta_t \xi|)$ for arbitrary $t\in [1/2,4]$.

Furthermore, since $A_{k,m}$ satisfies (\ref{symbol123}), if  $a(\xi,t):=2^{j/2}\widetilde{A}_{k,m}(\xi,t)\beta(2^{-j}|\delta_t \xi|)$ for $k$ sufficiently large, then $a(\xi,t)$ is a symbol of order zero, i.e. for any $t\in [1/2,4]$, $\alpha\in {\mathbb{N}}^2$,
\begin{equation}\label{symbola}
\left|\left(\frac{\partial}{\partial \xi}\right)^{\alpha}a(\xi,t)\right|\leq C_{\alpha}(1+|\xi|)^{-|\alpha|}.
\end{equation}

To finish the proof of (E3), we need the following theorem, whose proof will be enclosed later.
\begin{thm}\label{L^(p,q)therom}
For all $j > mk$, and $p,q$ satisfying $\frac{1}{2p}< \frac{1}{q} \le  \frac{3}{5p}$, $\frac{3}{q} \le 1-\frac{1}{p} $, $\frac{1}{q} \ge \frac{1}{2} -\frac{1}{p} $, we have
\begin{equation}
\biggl(\int_{{\mathbb{R}}^2}\int^4_{1/2}|\tilde{F}_{j}^{k}f(y,t)|^qdtdy \biggl)^{1/q}
\leq C 2^{\frac{mk}{2}(1-\frac{1}{p}-\frac{1}{q})} 2^{(\frac{1}{2} -\frac{3}{q}+ \frac{1}{p} +\epsilon)j} \|f\|_{L^p({\mathbb{R}}^2)},\hspace{0.2cm}\textmd{some }\hspace{0.2cm}\epsilon >0,
\end{equation}
where
\begin{equation}\label{estimate:L4}
 \tilde{F}_{j}^{k}f(y,t)=\rho_1(y,t)\int_{{\mathbb{R}}^2}e^{i(\xi \cdot y-t^2\xi_2\tilde{\Phi}(s,\delta))}a(\xi,t)\rho_0(2^{-j}|\xi|)\tilde{\chi}(\frac{\xi_1}{\xi_2})\hat{f}(\xi)d\xi.
\end{equation}
\end{thm}

Notice that when $\frac{1}{2p}< \frac{1}{q} \le  \frac{3}{5p}$, $\frac{3}{q} = 1-\frac{1}{p} $, Theorem \ref{L^(p,q)therom}
implies
\begin{equation*}
\biggl(\int_{{\mathbb{R}}^2}\int^4_{1/2}|\tilde{F}_{j}^{k}f(y,t)|^qdtdy \biggl)^{1/q}
\leq C 2^{\frac{mk}{q}} 2^{(\frac{1}{2} -\frac{3}{q}+ \frac{1}{p} +\epsilon)j} \|f\|_{L^p({\mathbb{R}}^2)},\hspace{0.2cm}\textmd{some }\hspace{0.2cm}\epsilon >0.
\end{equation*}
According to the $L^{p}$-estimate (\ref{p-p}), we also have
\begin{equation*}
\|\tilde{F}_{j}^{k}f\|_{L^{\infty}(\mathbb{R}^{2}\times [1/2, 4])}
\leq C   2^{\frac{j}{2}} \|f\|_{L^\infty({\mathbb{R}}^2)}.
\end{equation*}
Then it follows from the Riesz interpolation that
\begin{equation}\label{final}
\biggl(\int_{{\mathbb{R}}^2}\int^4_{1/2}|\tilde{F}_{j}^{k}f(y,t)|^qdtdy \biggl)^{1/q}
\leq C 2^{\frac{mk}{q}} 2^{(\frac{1}{2} -\frac{3}{q}+ \frac{1}{p} +\epsilon)j} \|f\|_{L^p({\mathbb{R}}^2)},\hspace{0.2cm}\textmd{some }\hspace{0.2cm}\epsilon >0,
\end{equation}
provided that $p, q$ satisfy $\frac{1}{2p}< \frac{1}{q} \le  \frac{3}{5p}$, $\frac{3}{q} \le 1-\frac{1}{p} $.

By inequalities (\ref{well}) and (\ref{final}), we obtain
\begin{align*}
\|\widetilde{\mathcal{M}_{j}^{k,1}}f\|_{L^q({\mathbb{R}}^2)}&\leq C 2^{\frac{j}{q} -\frac{mk}{q} -\frac{j}{2}}2^{\frac{mk}{q}}  2^{(\frac{1}{2} -\frac{3}{q}+ \frac{1}{p} +\epsilon)j} \|f\|_{L^p({\mathbb{R}}^2)}\\
&=C  2^{(\frac{1}{p} -\frac{2}{q}  +\epsilon)j} \|f\|_{L^p({\mathbb{R}}^2)}.
\end{align*}
Then (E3) is established.

Now let's turn to prove Theorem \ref{L^(p,q)therom}.

\textbf{Proof of Theorem \ref{L^(p,q)therom}.} We mainly use  Whitney type decomposition and a bilinear estimate established in \cite{SL2}. By rescaling, we turn to estimate
\begin{equation}\label{estimate:L4+}
\mathcal{F}_{j}^{k}g(y,t)=\rho_1(y,t)\int_{{\mathbb{R}}^2}e^{i2^{j}(\xi \cdot y-t^2\xi_2\tilde{\Phi}(s,\delta))}a(2^{j}\xi,t)\rho_0(|\xi|)\tilde{\chi}(\frac{\xi_1}{\xi_2})\hat{g}(\xi)d\xi,
\end{equation}
since $\tilde{\Phi}(s,\delta)$ defined in (\ref{Phase}) is homogeneous of degree zero in $\xi$. Notice that by finite decomposition, we may assume that $\tilde{\chi}\in C_0^{\infty}([c_1,c_2])$, where
$|c_1-c_2| \le \epsilon_{0}$
with $\epsilon_{0}$ sufficiently small. By Whitney type decomposition, we have
\[(\mathcal{F}_{j}^{k}g)^{2}= \sum_{l:  log1/\epsilon_{0}  \le l \le (j-mk)/2 } \sum_{dist(C_{\theta}^{l}, C_{\theta^{\prime}}^{l}) \approx 2^{-l}} \mathcal{F}_{j}^{k}g_{\theta}^{l} \mathcal{F}_{j}^{k}g_{\theta^{\prime}}^{l},\]
where $\{C_{\theta}^{l}\}_{l:  log1/\epsilon_{0}  \le l \le (j-mk)/2 }$ are intervals with center $\theta$ and length $2^{-l}$,
\[\widehat{g_{\theta}^{l}}(\xi) = \hat{g}(\xi)\chi_{C_{\theta}^{l}}(\frac{\xi_{1}}{\xi_{2}}).\]
Here we abuse the notation by saying $dist(C_{\theta}^{l}, C_{\theta^{\prime}}^{l}) \approx 2^{-l}$ to mean $dist(C_{\theta}^{l}, C_{\theta^{\prime}}^{l}) \le 2^{-l}$ when $l=(j-mk)/2$. By the similar argument as in \cite{SL2}, we can establish the orthogonality that for $q \ge 4$,
 \begin{align}\label{Eq3.2}
\|\mathcal{F}_{j}^{k}g\|_{_{L^{q}(\mathbb{R}^{3})}}
&\lesssim 2^{j \epsilon} \biggl\{\sum_{l:  log1/\epsilon_{0}  \le l \le (j-mk)/2 } \biggl(\sum_{dist(C_{\theta}^{l}, C_{\theta^{\prime}}^{l}) \approx 2^{-l}} \| \mathcal{F}_{j}^{k}g_{\theta}^{l} \mathcal{F}_{j}^{k}g_{\theta^{\prime}}^{l}\|_{_{L^{q/2}(\mathbb{R}^{3})}}^{(q/2)^{\prime}} \biggl)^{1/(q/2)^{\prime}} \biggl \}^{1/2}.
\end{align}

It is sufficient to prove the following two lemmas.

\begin{lem}\label{main}
For each $l$: $2^{l} \ll 2^{(j-mk)/2}$, we have
\begin{equation*}
\| \mathcal{F}_{j}^{k}g_{\theta}^{l} \mathcal{F}_{j}^{k}g_{\theta^{\prime}}^{l}\|_{_{L^{q/p}(\mathbb{R}^{3})}} \leq C 2^{-l+\epsilon j} 2^{(mk + 3l -3j)\frac{p}{q}} \|g_{\theta}^{l}\|_{L^2}\|g_{\theta^{\prime}}^{l}\|_{L^2}.
\end{equation*}
\end{lem}

\begin{lem}\label{basic}
For each $l$: $2^{l} \approx 2^{(j-mk)/2}$, we have
\begin{equation*}
\| \mathcal{F}_{j}^{k}g_{\theta}^{l} \mathcal{F}_{j}^{k}g_{\theta^{\prime}}^{l}\|_{_{L^{q/p}(\mathbb{R}^{3})}} \leq C 2^{-2j\frac{p}{q}} 2^{-\frac{j-mk}{2}(1-\frac{p}{q})} \|g_{\theta}^{l}\|_{L^2}\|g_{\theta^{\prime}}^{l}\|_{L^2}.
\end{equation*}
\end{lem}

Indeed, it can be observed from Lemma 2.12 in \cite{WL} that for each $l: log1/\epsilon_{0}  \le l \le (j-mk)/2 $,
\begin{equation*}
\| \mathcal{F}_{j}^{k}g_{\theta}^{l} \mathcal{F}_{j}^{k}g_{\theta^{\prime}}^{l}\|_{_{L^{\infty}(\mathbb{R}^{3})}} \leq C 2^{-2l+j}  \|g_{\theta}^{l}\|_{L^\infty}\|g_{\theta^{\prime}}^{l}\|_{L^\infty}.
\end{equation*}
If Lemma \ref{main} and Lemma \ref{basic} hold true, then by the Riesz interpolation theorem, for each $l$: $2^{l} \ll 2^{(j-mk)/2}$, we obtain
\begin{equation}\label{Eq3.3}
\| \mathcal{F}_{j}^{k}g_{\theta}^{l} \mathcal{F}_{j}^{k}g_{\theta^{\prime}}^{l}\|_{_{L^{q/2}(\mathbb{R}^{3})}} \leq C 2^{\frac{2mk}{q}} 2^{2(\frac{1}{2} -\frac{3}{q}- \frac{1}{p} +\epsilon)j} 2^{2(-1+\frac{3}{q}+ \frac{1}{p} )l} \|g_{\theta}^{l}\|_{L^p}\|g_{\theta^{\prime}}^{l}\|_{L^p}.
\end{equation}
And for each $l$: $2^{l} \approx 2^{(j-mk)/2}$, we obtain
\begin{equation}\label{Eq3.4}
\| \mathcal{F}_{j}^{k}g_{\theta}^{l} \mathcal{F}_{j}^{k}g_{\theta^{\prime}}^{l}\|_{_{L^{q/2}(\mathbb{R}^{3})}} \leq C 2^{mk(1-\frac{1}{p}-\frac{1}{q})} 2^{(-\frac{3}{q}- \frac{1}{p} )j}  \|g_{\theta}^{l}\|_{L^p}\|g_{\theta^{\prime}}^{l}\|_{L^p}.
\end{equation}
Since $\frac{3}{q} \le 1-\frac{1}{p} $, inequalities (\ref{Eq3.3}) and (\ref{Eq3.4}) imply that for each $l:  log1/\epsilon_{0}  \le l \le (j-mk)/2 $,
\begin{equation}\label{Eq3.5}
\| \mathcal{F}_{j}^{k}g_{\theta}^{l} \mathcal{F}_{j}^{k}g_{\theta^{\prime}}^{l}\|_{_{L^{q/2}(\mathbb{R}^{3})}} \leq C 2^{mk(1-\frac{1}{p}-\frac{1}{q})}  2^{2(\frac{1}{2} -\frac{3}{q}- \frac{1}{p} +\epsilon)j} \|g_{\theta}^{l}\|_{L^p}\|g_{\theta^{\prime}}^{l}\|_{L^p}.
\end{equation}
Since for each $l$, $C_{\theta}^{l}, C_{\theta^{\prime}}^{l}$ are almost disjoint, and the assumption that $\frac{1}{q} \ge \frac{1}{2} -\frac{1}{p} $, we get
\begin{align}\label{Eq3.6}
\sum_{dist(C_{\theta}^{l}, C_{\theta^{\prime}}^{l}) \approx 2^{-l}} \| g_{\theta}^{l}\|_{_{L^{p}}}^{2(q/2)^{\prime}} \lesssim \|g\|_{L^{p}}^{2(q/2)^{\prime}}; \quad
\sum_{dist(C_{\theta}^{l}, C_{\theta^{\prime}}^{l}) \approx 2^{-l}} \| g_{\theta^{\prime}}^{l}\|_{_{L^{p}}}^{2(q/2)^{\prime}} \lesssim \|g\|_{L^{p}}^{2(q/2)^{\prime}}.
\end{align}
It follows from inequalities (\ref{Eq3.2}), (\ref{Eq3.5}), H\"{o}lder's inequality and (\ref{Eq3.6}) that
\begin{align}
\|\mathcal{F}_{j}^{k}g\|_{_{L^{q}(\mathbb{R}^{3})}} \leq C 2^{\frac{mk}{2}(1-\frac{1}{p}-\frac{1}{q})} 2^{(\frac{1}{2} -\frac{3}{q}- \frac{1}{p} +\epsilon)j} \|g\|_{L^p({\mathbb{R}}^2)},
\end{align}
which implies Theorem \ref{L^(p,q)therom} by rescaling.

Let's turn to prove Lemma \ref{main} and Lemma \ref{basic}. 
The proof of Lemma \ref{basic} is quite trivial. In fact, for each $l,j,k$, we have the following basic estimates,
\begin{align*}
\|\mathcal{F}_{j}^{k}g_{\theta}^{l}\|_{_{L^{2}(\mathbb{R}^{3})}} \leq C 2^{-j} \|\widehat{g_{\theta}^{l}}\|_{L^2({\mathbb{R}}^2)},
\end{align*}
\begin{align*}
\|\mathcal{F}_{j}^{k}g_{\theta}^{l}\|_{_{L^{\infty}(\mathbb{R}^{3})}} \leq C 2^{-l} \|\widehat{g_{\theta}^{l}}\|_{L^\infty({\mathbb{R}}^2)},
\end{align*}
\begin{align*}
\|\mathcal{F}_{j}^{k}g_{\theta}^{l}\|_{_{L^{\infty}(\mathbb{R}^{3})}} \leq \|\widehat{g_{\theta}^{l}}\|_{L^1({\mathbb{R}}^2)},
\end{align*}
then Lemma \ref{basic} can be obtained by  the Riesz interpolation theorem and H\"{o}lder's inequality. Therefore, we will prove Lemma \ref{main} in the rest of this subsection.

We first do some reductions. Set
\[\kappa = \frac{1}{2} (\theta + \theta^{\prime}), \:\ \kappa \thickapprox 1,\]
then we define $\mathcal{C}_{\theta}^{l}, \mathcal{C}_{\theta^{\prime}}^{l}$ by
\[\mathcal{C}_{\theta}^{l}:=\biggl\{(\xi_{1},\xi_{2}): \frac{1}{2}2^{-l} \le \frac{\xi_{1}}{\xi_{2}} - \kappa \le \frac{3}{2} 2^{-l} \biggl\},\]
\[\mathcal{C}_{\theta^{\prime}}^{l}:=\biggl\{(\xi_{1},\xi_{2}):-\frac{3}{2}2^{-l} \le \frac{\xi_{1}}{\xi_{2}} - \kappa \le -\frac{1}{2}2^{-l} \biggl\}.\]
By change of variables,
\begin{equation}\label{Eq3.11}
\xi_{1} = 2^{-l}\eta_{1} + \kappa \eta_{2}, \quad \xi_{2} = \eta_{2},
\end{equation}
we get
\[\mathcal{C}_{1}:=\biggl\{(\eta_{1},\eta_{2}): \frac{1}{2} \le \frac{\eta_{1}}{\eta_{2}} \le \frac{3}{2}  \biggl\},\]
\[\mathcal{C}_{2}:=\biggl\{(\eta_{1},\eta_{2}): -\frac{3}{2} \le \frac{\eta_{1}}{\eta_{2}} \le -\frac{1}{2}  \biggl\},\]
and reduce Lemma \ref{main}  to the following lemma.
\begin{lem}\label{last}
For each $l$: $2^{l} \ll 2^{(j-mk)/2}$, and any function $g_{i}$ with supp $\hat{g_{i}} \subset \mathcal{C}_{i}$,
$i=1,2$, we have
\begin{equation}\label{Eqglobal}
\| \mathcal{G}_{j,l}^{k}g_{1} \mathcal{G}_{j,l}^{k}g_{2}\|_{_{L^{q/p}(\mathbb{R}^{3})}} \leq C 2^{\epsilon j} 2^{(mk + 3l -3j)\frac{p}{q}} \|g_{1}\|_{L^2}\|g_{2}\|_{L^2},
\end{equation}
where
\[\mathcal{G}_{j,l}^{k}g_{i}(x,t)=\tilde{\rho}_1(x,t)\int_{{\mathbb{R}}^2}e^{i2^{j}\Psi(x,t,\eta,\delta,2^{-l})} \tilde{a}(2^{j}\eta,t)\tilde{\rho_0}(|\eta|) \widehat{g_{i}}(\eta)d\eta\]
with the phase function
\begin{align}
\Psi(x,t,\eta,\delta,2^{-l}) &= 2^{-l}x_{1}\eta_{1}+x_{2}\eta_{2} +  2^{-2l-mk}\kappa^{m} \frac{\binom{2}{m+2}\phi^{(m)}(0)}{m!t^{m}}\frac{\eta_{1}^{2}}{\eta_{2}}  +2^{-3l-mk} \kappa^{m-1} \frac{\binom{3}{m+2}\phi^{(m)}(0)}{m!t^{m}}\frac{\eta_{1}^{3}}{\eta_{2}^{2}} +   \nonumber\\
&\cdot \cdot \cdot+ 2^{-(m+2)l-mk}  \frac{\binom{m+2}{m+2}\phi^{(m)}(0)}{m!t^{m}}\frac{\eta_{1}^{(m+2)}}{\eta_{2}^{(m+1)}} +2^{-2l-(m+1)k}\eta_2\tilde{R}(\frac{\eta_{1}}{\eta_{2}},t,2^{-l},\delta), \nonumber
\end{align}
where
\[\tilde{R}(\frac{\eta_{1}}{\eta_{2}},t,2^{-l},\delta) =\frac{\eta_{1}^{2}}{2!\eta_{2}^{2}}(\partial _{1}^{2}\bar{R})(\kappa,t,\delta) + O(2^{-l}\frac{\eta_{1}^{3}}{\eta_{2}^{3}}(\partial _{1}^{}\bar{R})(\kappa,t,\delta) ), \quad \bar{R}(\frac{\eta_{1}}{\eta_{2}}, t, \delta) = \frac{1}{\eta_{2}} R(\eta, t, \delta).\]
and $\tilde{\rho_{1}}(x,t) \in C_{0}^{\infty}(\mathbb{R}^{2} \times [1/2,4])$,
\[\tilde{a}(2^{j}\eta,t) = a(2^{j}\tau(\eta),t),\quad \tilde{\rho_0}(|\eta|)=\rho_0(|\tau(\eta)|), \quad \tau: (\eta_{1},\eta_{2}) \rightarrow (2^{-l}\eta_{1} + \kappa \eta_{2},\eta_{2}).\]
\end{lem}

More concretely, if Lemma \ref{last} holds true, under the transformation (\ref{Eq3.11}),
denote
\[\widehat{h_{i}}(\eta) =e^{i2^{j} \frac{(2^{-l}\eta_{1}+\kappa \eta_{2})^{2}}{2\eta_{2}} } 2^{-l}\chi_{i}(\frac{\eta_{1}}{\eta_{2}} ) \widehat{g_i}(\tau(\eta)), \quad i=1,2. \]
By Plancherel,
\begin{align}
\| \mathcal{F}_{j}^{k}g_{\theta}^{l} \mathcal{F}_{j}^{k}g_{\theta^{\prime}}^{l}\|_{_{L^{q/p}(\mathbb{R}^{3})}}
&\leq  \| \mathcal{G}_{j,l}^{k}h_{1} \mathcal{G}_{j,l}^{k} h_{2}  \|_{_{L^{q/p}(\mathbb{R}^{3})}} \nonumber\\
&\leq C 2^{\epsilon j}2^{(mk + 3l -3j)\frac{p}{q}} \Pi_{i=1}^{2} \biggl\|2^{-l} e^{i2^{j} \frac{(2^{-l}\eta_{1}+\kappa \eta_{2})^{2}}{2\eta_{2}} } \chi_{i}(\frac{\eta_{1}}{\eta_{2}}  ) \widehat{g_i}(\tau(\eta)) \biggl\|_{L^2} \nonumber\\
&= C 2^{-l+\epsilon j} 2^{(mk + 3l -3j)\frac{p}{q}} \|g_{\theta}^{l}\|_{L^2}\|g_{\theta^{\prime}}^{l}\|_{L^2}, \nonumber
\end{align}
where
\[\chi_{1}(\frac{\eta_{1}}{\eta_{2}}) = \chi_{C_{\theta}^{l}}(2^{-l}\frac{\eta_{1}}{\eta_{2}} + \kappa), \quad \chi_{2}(\frac{\eta_{1}}{\eta_{2}}) = \chi_{C_{\theta^{\prime}}^{l}} (2^{-l}\frac{\eta_{1}}{\eta_{2}} + \kappa).\]

Now we turn to the proof of Lemma \ref{last}. Notice that the phase function
$\Psi(x,t,\eta,\delta,2^{-l})$
can be considered as a small perturbation of
\[2^{-l}x_{1}\eta_{1}+x_{2}\eta_{2} +2^{-2l-mk}\kappa^{m} \frac{\binom{2}{m+2}\phi^{(m)}(0)}{m!t^{m}}\frac{\eta_{1}^{2}}{\eta_{2}}, \]
 since $k$ and $l$ can be sufficiently large. Let $Q_{0}=[-2^{-l-mk}, 2^{-l-mk}] \times [-2^{-2l-mk}, 2^{-2l-mk}]$,
then
\begin{equation}\label{Eqlocal1}
\| \mathcal{G}_{j,l}^{k}g_{1} \mathcal{G}_{j,l}^{k}g_{2}\|_{_{L^{q/p}(Q_{0} \times \mathbb{R})}} \leq  2^{(-2mk - 3l)\frac{p}{q}} \| Tg_{1} Tg_{2}\|_{_{L^{q/p}(Q(0,1) \times [1/2,4])}},
\end{equation}
where $Q(0,1)$ denotes the unit square and
\[Tg_{i}(x,t)= \int_{{\mathbb{R}}^2}e^{i2^{j-mk-2l}\tilde{\Psi}(x,t,\eta,\delta,2^{-l})} \tilde{a}(2^{j}\eta,t)\tilde{\rho_0}(|\eta|) \widehat{g_{i}}(\eta)d\eta\]
with the phase function
\begin{align}
\tilde{\Psi}(x,t,\eta,\delta,2^{-l}) &= x_{1}\eta_{1}+x_{2}\eta_{2}+ \kappa^{m} \frac{\binom{2}{m+2}\phi^{(m)}(0)}{m!t^{m}}\frac{\eta_{1}^{2}}{\eta_{2}}  +2^{-l} \kappa^{m-1} \frac{\binom{3}{m+2}\phi^{(m)}(0)}{m!t^{m}}\frac{\eta_{1}^{3}}{\eta_{2}^{2}} +   \nonumber\\
&\cdot \cdot \cdot+ 2^{-ml}  \frac{\binom{m+2}{m+2}\phi^{(m)}(0)}{m!t^{m}}\frac{\eta_{1}^{(m+2)}}{\eta_{2}^{(m+1)}} +2^{-k}\eta_2\tilde{R}(\frac{\eta_{1}}{\eta_{2}},t,2^{-l},\delta). \nonumber
\end{align}
It follows from  Theorem 1.2 in \cite{SL2} that
\begin{equation}\label{Eq3.14}
 \| Tg_{1} Tg_{2}\|_{_{L^{q/p}(Q(0,1) \times [1/2,4])}} \leq C 2^{(3mk + 6l-3j + \epsilon j)\frac{p}{q}} \|g_{1}\|_{L^2}\|g_{2}\|_{L^2}.
\end{equation}
The inequalities (\ref{Eqlocal1}) and (\ref{Eq3.14}) imply that
\begin{equation}\label{Eqlocal}
\| \mathcal{G}_{j,l}^{k}g_{1} \mathcal{G}_{j,l}^{k}g_{2}\|_{_{L^{q/p}(Q_{0} \times \mathbb{R})}} \leq C 2^{\epsilon j} 2^{(mk + 3l -3j)\frac{p}{q}} \|g_{1}\|_{L^2}\|g_{2}\|_{L^2},
\end{equation}
and by translation invariance in the $x$-plane, this inequality actually holds true when $Q_{0}$ is replaced by any $2^{-l-mk} \times 2^{-2l-mk}$ rectangles. Now we will show that the global estimate (\ref{Eqglobal}) follows from the local estimate (\ref{Eqlocal}). We have
\begin{align}
\mathcal{G}_{j,l}^{k}g_{i}(x,t)=\tilde{\rho}_1(x,t) \int_{\mathbb{R}^{2}} \int_{{\mathbb{R}}^2}e^{i2^{j}\Psi(x,t,\eta,\delta,2^{-l})-iz \cdot \eta} \tilde{a}(2^{j}\eta,t)\tilde{\rho_0}(|\eta|) \chi_{i}(\frac{\eta_{1}}{\eta_{2}})  d\eta g_{i}(z)dz. \nonumber
\end{align}
Denote
\begin{align}
K_{i}(x,z,t)=  \int_{{\mathbb{R}}^2}e^{i2^{j}\Psi(x,t,\eta,\delta,2^{-l})-iz \cdot \eta} \tilde{a}(2^{j}\eta,t)\tilde{\rho_0}(|\eta|) \chi_{i}(\frac{\eta_{1}}{\eta_{2}}) d\eta. \nonumber
\end{align}
Integration by parts show that for each $t \in [1/2,4]$ and $i=1,2$,
\begin{align}
|K_i(x,z,t)| \leq  \frac {C _{N}}{\biggl(1+ 2^{j-l}\biggl | | x_{1}-\frac{z_{1}}{2^{j-l}}| +\mathcal{O}(2^{-l-mk}) \biggl| + 2^{j}\biggl | | x_{2}-\frac{z_{2}}{2^{j}}|+\mathcal{O}(2^{-2l-mk}) \biggl|   \biggl)^{N}}, \nonumber
\end{align}
which implies that $(x_{1},x_{2})$ can be considered roughly in a $2^{-l-mk} \times 2^{-2l-mk}$ rectangle with the center $(\frac{z_{1}}{2^{j-l}},\frac{z_{2}}{2^{j}})$. Therefore
\begin{align}\label{sum}
&\mathcal{G}_{j,l}^{k}g_{1}(x,t)\mathcal{G}_{j,l}^{k}g_{2}(x,t) \nonumber\\
&=\tilde{\rho}_1(x,t)^{2} \sum_{Q,Q^{\prime}} \Pi_{cQ \cap cQ^{\prime}}(x) \int_{\mathbb{R}^{2}} K_{1}(x,z,t) \Pi_{Q}(\frac{z_{1}}{2^{j-l}},\frac{z_{2}}{2^{j}}) g_{1}(z)dz \nonumber\\
&\quad \times \int_{\mathbb{R}^{2}} K_{2}(x,z,t) \Pi_{Q^{\prime}}(\frac{z_{1}}{2^{j-l}},\frac{z_{2}}{2^{j}}) g_{2}(z)dz
 +C_{2N} 2^{-jN} 2^{-kN} \tilde{\rho}_1(x,t)^{2} \|g_{1}\|_{L^2}\|g_{2}\|_{L^2} \nonumber\\
 &= \sum_{Q,Q^{\prime}} \Pi_{cQ \cap cQ^{\prime}}(x) \mathcal{G}_{j,l}^{k} \biggl( \Pi_{Q}(\frac{z_{1}}{2^{j-l}},\frac{z_{2}}{2^{j}})g_{1}(z) \biggl)(x,t) \mathcal{G}_{j,l}^{k} \biggl( \Pi_{Q^{\prime}}(\frac{z_{1}}{2^{j-l}},\frac{z_{2}}{2^{j}}) g_{2}(z) \biggl)(x,t) \nonumber\\
 &\quad +C_{2N} 2^{-jN} 2^{-kN} \tilde{\rho}_1(x,t)^{2} \|g_{1}\|_{L^2}\|g_{2}\|_{L^2},
\end{align}
in which $Q,Q^{\prime}$ are $2^{-l-mk} \times 2^{-2l-mk}$ rectangles, by $\Pi_{Q}$ and  $\Pi_{Q^{\prime}}$ we mean the restriction on $Q,Q^{\prime}$, respectively.

Notice that by uncertainty principle and the assumption that $2^{l} \ll 2^{(j-mk)/2}$, the Fourier transform of $\Pi_{Q}(\frac{z_{1}}{2^{j-l}},\frac{z_{2}}{2^{j}}) g_{1}(z)$ is supported in a sufficiently small neighborhood of $\mathcal{C}_{1}$. For the same reason, the Fourier transform of $\Pi_{Q^{\prime}}(\frac{z_{1}}{2^{j-l}},\frac{z_{2}}{2^{j}}) g_{2}(z)$ is supported in a sufficiently small neighborhood of $\mathcal{C}_{2}$.

 Then inequalities (\ref{Eqlocal}) and (\ref{sum}) imply that
\begin{align}
&\| \mathcal{G}_{j,l}^{k}g_{1} \mathcal{G}_{j,l}^{k}g_{2}\|_{_{L^{q/p}( \mathbb{R}^{3} )}} \nonumber\\
&\leq C 2^{(mk + 3l -3j+\epsilon j)\frac{p}{q}}\sum_{Q,Q^{\prime}: cQ \cap cQ^{\prime} \neq \emptyset} \biggl\|\Pi_{Q}(\frac{z_{1}}{2^{j-l}},\frac{z_{2}}{2^{j}}) g_{1}(z) \biggl\|_{L^{2}} \biggl\|\Pi_{Q^{\prime}}(\frac{z_{1}}{2^{j-l}},\frac{z_{2}}{2^{j}})g_{2}(z) \biggl\|_{L^{2}}  \nonumber\\
 &\quad +C_{2N} 2^{-jN} 2^{-kN}  \|g_{1}\|_{L^2}\|g_{2}\|_{L^2} \nonumber\\
 & \leq C 2^{\epsilon j} 2^{(mk + 3l -3j+\epsilon j)\frac{p}{q}} \|g_{1}\|_{L^2}\|g_{2}\|_{L^2}.
\end{align}
This completes the proof of Lemma \ref{last}, and also ends the discussion of Theorem \ref{hhisomaintheorem} in the case $d=2$.

\textbf{The case when $d > 2$. }Similarly as in the proof for $d = 2$, we choose $\tilde{\rho}\in C_0^{\infty}(\mathbb{R})$ such that supp $\tilde{\rho}\subset\{x:B/2\leq|x|\leq 2B\}$  and $\sum_k\tilde{\rho}(2^kx)=1$. Since the support of $\eta$ is sufficiently small, we can choose $k$ sufficiently large.

Denote
\begin{equation*}
A_tf(y):=\int_{\mathbb{R}}f(y_1-tx,y_2-t^dx^d\phi(x))\eta(x)dx=\sum_k A_t^k f(y),
\end{equation*}
where $ A_t^k f(y):=\int_{\mathbb{R}}f(y_1-tx,y_2-t^dx^d\phi(x))\eta(x)\tilde{\rho}(2^kx)dx$.

Consider the isometric operator on $L^p(\mathbb{R}^2)$ defined by
\begin{equation*}
Tf(x_1,x_2)=2^{(d+1)k/p}f(2^kx_1,2^{dk}x_2).
\end{equation*}
By the same reasoning in the last section,  it suffices to prove the following estimate:
\begin{equation*}
\sum_k2^{\frac{(d+1)k}{p}-\frac{(d+1)k}{q} -k}\|\sup_{t \in [1, 2]}|\widetilde{A_t^k|}\|_{L^p\rightarrow L^{q}}\leq C_{p,q},
\end{equation*}
where $\widetilde{A_t^k}f(y):=\int_{\mathbb{R}}f(y_1-tx,y_2-t^dx^d\phi(\frac{x}{2^k}))\tilde{\rho}(x)\eta(2^{-k}x)dx$.

By means of the Fourier inversion formula, we can write
\begin{equation*}
\widetilde{ A_t^k}f(y)=\frac{1}{(2\pi)^2}\int_{{\mathbb{R}}^2}e^{i\xi\cdot
  y}\widehat{d\mu_{k,d,m}}(\delta_t\xi)\hat{f}(\xi)d\xi,
\end{equation*}
where
\begin{equation*}
 \widehat{d\mu_{k,d,m}}(\delta_t\xi)=\int_{\mathbb{R}}e^{-i(t\xi_1x+t^d\xi_2x^d\phi(\frac{x}{2^k}))}\tilde{\rho}(x)\eta(2^{-k}x)dx.
\end{equation*}

Choosing a non-negative function $\beta\in C_0^{\infty}(\mathbb{R})$ as before,
set
\begin{equation*}\widetilde{A_{t,j}^k}f(y):=\frac{1}{(2\pi)^2}\int_{{\mathbb{R}}^2}e^{i\xi\cdot
  y}\widehat{d\mu_{k,d,m}}(\delta_t\xi)\beta(2^{-j}|\delta_t\xi|)\hat{f}(\xi)d\xi
\end{equation*}
and denote by $\widetilde{\mathcal{M}_j^k}$ the corresponding maximal operator. Now we have that
\begin{equation*}
\sup_{t \in [1,2]}|\widetilde{A_t^k}f(y)|\leq \widetilde{\mathcal{M}^{k,0}}f(y)+\sum_{j\geq 1}\widetilde{\mathcal{M}_{j}^k}f(y), \hspace{0.2cm}\textmd{for}\hspace{0.2cm}y\in \mathbb{R}^2,
\end{equation*}
where
\begin{equation*}
\widetilde{\mathcal{M}^{k,0}}f(y):=\sup_{t \in [1,2]}|\sum_{j\leq 0}\widetilde{A_{t,j}^k}f(y)|.
\end{equation*}

Similarly as in the case $d = 2$, we can prove the following three estimates. \\
\textbf{(E1)} For $q \ge p \ge 1$, we have
\begin{align*}
\|\widetilde{\mathcal{M}^{k,0}}f\|_{L^{q}} \lesssim \|f\|_{L^{p}}.
\end{align*}
\textbf{(E2)} Suppose that $\frac{1}{2p} < \frac{1}{q} \le \frac{1}{p}$, $\frac{1}{q} \ge \frac{3}{p}-1$. Then for any $\epsilon >0$, there holds
\begin{equation*}
\|\widetilde{\mathcal{M}_{j}^{k}}f\|_{L^{q}}\leq C_{p,q} 2^{j(1+\epsilon)(1/p-1/q)-(j\wedge mk)/q}\|f\|_{L^p}.
\end{equation*}
\textbf{(E3)} For all $j > mk$, and $p,q$ satisfying $\frac{1}{2p}< \frac{1}{q} \le  \frac{3}{5p}$, $\frac{3}{q} \le 1-\frac{1}{p} $,  we have
\begin{align*}
\|\widetilde{\mathcal{M}_{j}^{k}}f\|_{L^q({\mathbb{R}}^2)}
\le C   2^{(\frac{1}{p} -\frac{2}{q}  +\epsilon)j} \|f\|_{L^p({\mathbb{R}}^2)}.
\end{align*}

If $\frac{d+1}{p} - \frac{d+1}{q} - 1<0$, it follows from (E1) that
\[\sum_k2^{(d+1)k(\frac{1}{p} - \frac{1}{q})-k}\|\widetilde{\mathcal{M}^{k,0}}\|_{L^p\rightarrow L^q} \leq C_{p,q}.\]
In order to complete the proof, we split the set of $j$ into two  parts, which are $j>mk$ and $1  \le j \le mk$.

When $1 \le j\leq mk$,
by (E2), if $\frac{1}{2p} < \frac{1}{q} \le \frac{1}{p}$, $\frac{1}{q} \ge \frac{3}{p}-1$, for  $\epsilon >0$ sufficiently small, we get
\begin{align}
\sum_k 2^{(d+1)k(\frac{1}{p} - \frac{1}{q} )-k} \sum_{1 \le j\leq mk}\|\widetilde{\mathcal{M}_{j}^{k}}f\|_{L^q} &\leq C_{p,q} \sum_k 2^{(d+1)k(\frac{1}{p} - \frac{1}{q})-k} \sum_{1 \le j \le mk} 2^{j(\frac{1}{p}-\frac{2}{q}) + \epsilon j} \|f\|_{L^p} \nonumber\\
&\leq C_{p,q} \|f\|_{L^p}, \nonumber
\end{align}
provided that $\frac{d+1}{p} - \frac{d+1}{q} - 1<0$.

When $ j > mk$, notice that according to \cite{WL}, for $2< p < \infty$, there exists $\epsilon(p) >0$ such that
\begin{align*}
\|\widetilde{\mathcal{M}_{j}^{k}}f\|_{L^p({\mathbb{R}}^2)}
\le C   2^{-\epsilon(p) j} \|f\|_{L^p({\mathbb{R}}^2)}.
\end{align*}
By  (E3) and interpolation, for $p,q$ satisfying $\frac{1}{2p}< \frac{1}{q} \le  \frac{1}{p}$, $\frac{1}{q} > \frac{3}{p}-1 $, there exists $\epsilon(p,q) >0$ such that
\begin{align*}
\|\widetilde{\mathcal{M}_{j}^{k}}f\|_{L^q({\mathbb{R}}^2)}
\le C   2^{-\epsilon(p,q) j} \|f\|_{L^p({\mathbb{R}}^2)}.
\end{align*}
Thus if $p,q$ satisfy $\frac{1}{2p}< \frac{1}{q} \le  \frac{1}{p}$, $\frac{1}{q} > \frac{3}{p}-1 $,   and $\frac{d+1}{p} - \frac{d+1}{q} - 1<0$,  we have
\begin{align}
\sum_k 2^{(d+1)k(\frac{1}{p} - \frac{1}{q})-k} \sum_{j > mk}\|\widetilde{M_{j}^{k}}f\|_{L^q} &\leq C_{p,q} \sum_k 2^{(d+1)k(\frac{1}{p} - \frac{1}{q})-k}  \sum_{j > mk} 2^{-\epsilon(p,q) j} \|f\|_{L^p} \nonumber\\
&\leq C_{p,q} \|f\|_{L^p}. \nonumber
\end{align}
Then  we complete the proof of Theorem \ref{hhisomaintheorem}.

Next we simply explain how to  prove (E1), (E2) and (E3). Set
\begin{equation*}
\delta:=2^{-k}, \hspace{0.2cm} s:=s(\xi,t)=-\frac{\xi_1}{t^{d-1}\xi_2}, \hspace{0.3cm}\textrm{for} \hspace{0.2cm} \xi_2\neq 0,
\end{equation*}
and
\begin{equation*}
\Phi(s,x,\delta)=-sx+x^d\phi(\delta x).
\end{equation*}

Similarly as in the proof for the case $d=2$, for fixed $t\in[1,2]$, we will consider the main  term
 \begin{equation}\label{mainterm4}
A^k_{t,j}f(y):=\int_{{\mathbb{R}}^2}e^{i(\xi\cdot y-t^d\xi_2\tilde{\Phi}(s,\delta))}\chi_{k,d,m}(\frac{\xi_1}{t^{d-1}\xi_2})
\frac{A_{k,d,m}(\delta_t\xi)}{(1+|\delta_t\xi|)^{1/2}}
\beta(2^{-j}|\delta_t\xi|)\hat{f}(\xi)d\xi,
\end{equation}
where $\chi_{k,d,m}$ is a smooth function supported in the conical region $[c_{k,d,m},\widetilde{c_{k,d,m}}]$, for certain non-zero constants $c_{k,d,m}$ and $\widetilde{c_{k,d,m}}$ dependent only on $k, m$ and $d$. $A_{k,d,m}$ is a symbol of order zero in $\xi$ and $\{A_{k,d,m}(\delta_t\xi)\}_k$ is contained in a bounded subset of symbol of order zero. 

We may assume that $\phi(0)=1/d$, then the phase function can be written as
\begin{align}\label{phase function 4}
-t^d\xi_2\tilde{\Phi}(s,\delta)=\biggl(\frac{1}{d^{\frac{d}{d-1}}}-\frac{1}{d^{\frac{1}{d-1}}}\biggl)(-\frac{d\xi_1^d}{\xi_2})^{\frac{1}{d-1}}-\frac{\delta^m\phi^{(m)}(0)}{t^mm!}
\biggl(-\frac{\xi_1}{\xi_2^{\frac{m+1}{m+d}}}\biggl)^{\frac{d+m}{d-1}}+ R(t,\xi,\delta,d),
\end{align}
where $R(t,\xi,\delta,d)$ is homogeneous of degree one in $\xi$ and has at least $m+1$ power of $\delta$.

As in the case $d= 2$,  we  choose $\rho_1(y,t)\in C_0^{\infty}(\mathbb{R}^2\times[1/2,4])$ and  it is sufficient to prove (E2) and (E3) for the operator
\begin{equation*}
\widetilde{\mathcal{M}_{j}^{k,1}}f(y):=\sup_{t\in[1,2]}|\rho_1(y,t)A_{t,j}^kf(y)|.
\end{equation*}

Notice that by  the argument of Subsection 2.5  in \cite{WL}, we have
\begin{equation*}
\|\widetilde{\mathcal{M}_{j}^{k,1}}f\|_{L^p}\leq C_{p} 2^{- (j \wedge mk)/p}\|f\|_{L^p}, \hspace{0.5cm}2\leq p\leq \infty.
\end{equation*}
Also, we still have the same regularity estimates as Lemma \ref{lemmaL^2}: for any $\epsilon >0$,
\begin{equation*}
\|\widetilde{\mathcal{M}_{j}^{k,1}}f\|_{L^ \infty }\leq C_{\epsilon} 2^{j(1+\epsilon)}\|f\|_{L^1}.
\end{equation*}
Then  (E2) follows from  the Riesz interpolation theorem.

Moreover, from the inequality (2.82) in \cite{WL},
\begin{equation}\label{h_k_m}
\frac{\partial}{\partial t}(\rho_1(y,t)A^k_{t,j}f(y))=\int_{{\mathbb{R}}^2}e^{i(\xi \cdot y-t^d\xi_2\tilde{\Phi}(s,\delta))}h_{k,m,d}(y,t,\xi,j)\hat{f}(\xi)d\xi,
\end{equation}
where $|h_{k,m}(y,t,\xi,j)|\lesssim  2^{j/2}\delta^m+2^{-j/2}$.

Using the same reasoning as in the case when $d= 2$,  we could complete the proof given the following theorem. We will omit the proof the this theorem since it follows closely from the the argument for Theorem \ref{L^(p,q)therom}.

\begin{thm}
For all $j > mk$, and $p,q$ satisfying $\frac{1}{2p}< \frac{1}{q} \le  \frac{3}{5p}$, $\frac{3}{q} \le 1-\frac{1}{p} $, $\frac{1}{q} \ge \frac{1}{2} -\frac{1}{p} $, we have
\begin{equation}
\biggl(\int_{{\mathbb{R}}^2}\int^4_{1/2}|\tilde{F}_{j}^{k}f(y,t)|^qdtdy \biggl)^{1/q}
\leq C 2^{\frac{mk}{2}(1-\frac{1}{p}-\frac{1}{q})} 2^{(\frac{1}{2} -\frac{3}{q}+ \frac{1}{p} +\epsilon)j} \|f\|_{L^p({\mathbb{R}}^2)},\hspace{0.2cm}\textmd{some }\hspace{0.2cm}\epsilon >0,
\end{equation}
where
\begin{equation}
 \tilde{F}_{j}^{k}f(y,t)=\rho_1(y,t)\int_{{\mathbb{R}}^2}e^{i(\xi \cdot y-t^2\xi_2\tilde{\Phi}(s,\delta))}a(\xi,t)\rho_0(2^{-j}|\xi|)\tilde{\chi}(\frac{\xi_1}{\xi_2})\hat{f}(\xi)d\xi,
\end{equation}
$\tilde{\Phi}(s,\delta)$ was given in  equality (\ref{phase function 4}). Functions $\rho_0$, $a$, and $\tilde{\chi}$ are chosen as in Theorem \ref{L^(p,q)therom}.
\end{thm}


\subsection{Localized maximal functions associated with surfaces in $\mathbb{R}^{3}$}
We notice that some of the details have appeared in \cite{WL}, and we include here only for completeness.

\subsubsection{The case when $2a_{2} \neq a_{3}$: surfaces with non-vanishing principle curvature}
In this subsection we give the proof for Theorem \ref{3nonvanish}.

\textbf{Proof of Theorem \ref{3nonvanish}.}

We can always choose non-negative functions $\eta_1$,
$\eta_2$ $\in C_0^{\infty}({\mathbb{R}})$ so that $\eta(x)\leq \eta_1(x_1)\eta_2(x_2)$.
Since
\begin{equation*}
\left|\int_{\mathbb{R}^2}f(y-\delta_t(x_1,x_2,\Phi(x_1,x_2)))\eta(x)dx\right |\leq \int_{\mathbb{R}^2}|f|(y-\delta_t(x_1,x_2,\Phi(x_1,x_2)))\eta_1(x_1)\eta_2(x_2)dx,
\end{equation*}
we may assume $\eta(x)=\eta_1(x_1)\eta_2(x_2)$ and  $f\geq
0$, $a_1=1$. Set $(y_2, y_3)=y'$ and $(\xi_2,\xi_3)=\xi'$. Denote $1+a_2+a_3$ by $Q$ and $(t^{a_2}\xi_2, t^{a_3}\xi_3)$ by
$\delta'_t\xi'$.

Let
\begin{equation*}
A_tf(y):=\int_{\mathbb{R}^2}f(y-\delta_t(x_1,x_2,\Phi(x_1,x_2)))\eta_1(x_1)\eta_2(x_2)dx.
\end{equation*}

By means of the Fourier inversion formula, we can write
\begin{equation*}
 A_tf(y)=\frac{1}{(2\pi)^3}\int_{{\mathbb{R}}^3}e^{i\xi\cdot
  y}\int_{\mathbb{R}}e^{-it\xi_1x_1}\eta_1(x_1)\widehat{d\mu_{x_1}}(\delta'_t\xi')dx_1\hat{f}(\xi)d\xi,
\end{equation*}
where
\begin{equation*}
 \widehat{d\mu_{x_1}}(\xi')=\int_{\mathbb{R}}e^{-i(\xi_2x_2+\xi_3\Phi(x_1,x_2))}\eta_2(x_2)dx_2.
\end{equation*}

Choose a non-negative function $\beta\in C_0^{\infty}(\mathbb{R})$
such that
\begin{equation*}
\textrm{supp }\hspace{0.2cm}\beta \subset[1/2,2] \hspace{0.5cm}\textrm{and}
\hspace{0.5cm}\sum_{j\in \mathbb{Z}}\beta(2^{-j}r)=1 \hspace{0.5cm}
\textrm{for} \hspace{0.2cm} r>0.
\end{equation*}

Let
\begin{equation*}
A_{t,j}f(y):= \frac{1}{(2\pi)^3}\int_{{\mathbb{R}}^3}e^{i\xi\cdot
  y}\int_{\mathbb{R}}e^{-it\xi_1x_1}\eta_1(x_1)\widehat{d\mu_{x_1}}(\delta'_t\xi')dx_1\beta(2^{-j}|\delta'_t\xi'|)\hat{f}(\xi)d\xi,
\end{equation*}
and
\begin{align*}
 {}A_{t}^0f(y)
:&=A_{t}f(y)-\sum_{j=1}^{\infty}A_{t,j}f(y)
\\
&=\frac{1}{(2\pi)^3}\int_{{\mathbb{R}}^3}e^{i\xi\cdot
  y}\int_{\mathbb{R}}e^{-it\xi_1x_1}\eta_1(x_1)\widehat{d\mu_{x_1}}(\delta'_t\xi')dx_1\rho(|\delta'_t\xi'|)\hat{f}(\xi)d\xi,
\end{align*}
where $\rho$ is supported in a neighborhood of the origin. Since
\begin{align*}
{}&\frac{1}{(2\pi)^3}\int_{{\mathbb{R}}^3}e^{i\xi\cdot
  y}\int_{\mathbb{R}}e^{-it\xi_1x_1}\eta_1(x_1)\widehat{d\mu_{x_1}}(\delta'_t\xi')dx_1\rho(|\delta'_t\xi'|)d\xi
\\
&=\frac{1}{(2\pi)^2}t^{-Q}\eta_1(\frac{y_1}{t})\int_{{\mathbb{R}}^2}e^{i\delta'_{t^{-1}}\xi'\cdot
  y'}\widehat{d\mu_{\frac{y_1}{t}}}(\xi')\rho(|\xi'|)d\xi',
\end{align*}
then  $A_{t}^0f(y)=t^{-Q}f*K(t^{-1}y_1,t^{-a_2}y_2,t^{-a_3}y_3)=f*K_{\delta_{t^{-1}}}(y)$,
where
\begin{equation*}
K(y)=\frac{1}{(2\pi)^2}\eta_1(y_1)\int_{{\mathbb{R}}^2}e^{i\xi'\cdot
  y'}\widehat{d\mu_{y_1}}(\xi')\rho(|\xi'|)d\xi'.
\end{equation*}

Since $\Phi$ satisfies (\ref{conditionnonvani}) and $y_1 \in$ supp $\eta_1$, for $N\in \mathbb{N}$ and multi-index $\alpha$ with $|\alpha|=N$, then Theorem \ref{lem:Lemma1} implies that
\begin{equation*}
|D_{\xi'}^{\alpha}\widehat{d\mu_{y_1}}(\xi')|\leq C_N'(1+|\xi'|)^{-1/2}.
\end{equation*}
By integration by parts in $\xi'$, we obtain
\begin{align*}
|K(y)|&\leq C''_N\frac{1}{(1+|y_1|)^{N}}\frac{1}{(1+|y'|)^{N}}\int_{{\mathbb{R}}^2}\left|D_{\xi'}^{\alpha}\left(\widehat{d\mu_{y_1}}(\xi')\rho(|\xi'|)\right)\right|d\xi'\\
&\leq C_N\prod_{i=1}^3\frac{1}{(1+|y_i|)^N}.
\end{align*}
It follows from Young's inequality that for any $q \ge p \ge 1$, the following estimate holds:
\[\biggl\| \sup_{t\in[1,2]} |{}A_{t}^0f(y)| \biggl\|_{L^{q}(\mathbb{R}^{3})} \le C \|f\|_{L^{p}(\mathbb{R}^{3})}.\]

Hence it suffices to consider the local maximal operator $\sup_{t\in[1,2]}|A_{t,j}f(y)|$.
For fixed $t\in[1,2]$, let us estimate $\widehat{d\mu_{x_1}}(\delta'_t\xi')$. Set
\begin{equation*}
s:=s(\xi',t)=-\frac{t^{a_2}\xi_2}{t^{a_3}\xi_3}, \hspace{0.3cm}\textrm{for} \hspace{0.2cm} \xi_3\neq 0,
\end{equation*}
and
\begin{equation*}
\Psi(x_1,x_2,s):=-sx_2+\Phi(x_1,x_2).
\end{equation*}

We observe that
\begin{equation*}
\partial_2\Psi(0,0,0)=0 \hspace{0.5cm}\textrm{and}\hspace{0.5cm}\partial_2^2\Psi(0,0,0)\neq0,
\end{equation*}
since $\Phi$ satisfies the condition (\ref{conditionnonvani}). The implicit function theorem implies that there must be a smooth solution $x_2=\psi(x_1,s)$ to the equation
\begin{equation*}
\partial_2\Phi(x_1,x_2)=s,
\end{equation*}
where $x_1$ and $s$ are enough small. Here if $t\in [\frac{1}{2},4]$, we can choose $x_1$ and $s$ sufficiently small such that
\begin{equation*}
\partial_2\Phi(\frac{x_1}{t},\psi(\frac{x_1}{t},s))=s.
\end{equation*}

For sufficiently small $x_1$ chosen above, a standard application of the method of stationary phase in $x_2$  yields that
\begin{equation*}
\widehat{d\mu_{x_1}}(\delta'_t\xi')=e^{-it^{a_3}\xi_3\tilde{\Psi}(x_1,s)}
\frac{\chi_{x_1}(t^{a_2}\xi_2/t^{a_3}\xi_3)}{{(1+|\delta'_t\xi'|)^{1/2}}}
A_{x_1}(\delta'_t\xi')+B_{x_1}(\delta'_t\xi'),
\end{equation*}
where $\tilde{\Psi}(x_1,s):=\Psi(x_1,\psi(x_1,s),s)$ and  $\chi_{x_1}$ is a smooth function supported on the set $\{z:|z|<\epsilon_{x_1}\}$, where $\epsilon_{x_1}$ can be controlled by a small positive constant independent on $x_1$.
Moreover, $t^{a_3}\xi_3\tilde{\Psi}(x_1,s)$
is a smooth function which is homogeneous of degree one in $\xi'$.

Meanwhile, $A_{x_1}$ is a
symbol of order zero such that
\begin{equation*}
|D^{\alpha}_{\xi'}A_{x_1}(\xi')|\leq
C_{\alpha}(1+|\xi'|)^{-|\alpha|},
\end{equation*}
where $\alpha$ is a multi-index and $C_{\alpha}$ are admissible
constants.

Furthermore, $B_{x_1}$ is a smooth function and  satisfies
\begin{equation}\label{B1}
|D_{\xi'}^{\alpha}B_{x_1}(\xi')|\leq C_{\alpha,N}(1+|\xi'|)^{-N},
\end{equation}
again with admissible constants $C_{\alpha,N}$ and $N\in \mathbb{N}$.

We first consider the remainder term. For fixed $t\in [1,2]$, let
\begin{align*}
A_{t,j}^0f(y):&=\frac{1}{(2\pi)^3}\int
_{{\mathbb{R}}^3}e^{i\xi\cdot y
}\int_{\mathbb{R}}e^{-it\xi_1x_1}\eta_1(x_1)B_{x_1}(\delta'_t\xi')dx_1\beta(2^{-j}|\delta'_t\xi'|)\hat{f}(\xi)d\xi
\\
&=\mathcal{F}^{-1}\biggl\{\int
_{{\mathbb{R}}^3}e^{i\xi\cdot y
}\biggl(\int_{\mathbb{R}}e^{-it\xi_1x_1}\eta_1(x_1)B_{x_1}(\delta'_t\xi')dx_1\biggl)\beta(2^{-j}|\delta'_t\xi'|)\biggl\}*f(y)= \tilde{K}_{\delta_{t^{-1}}} * f(y),
\end{align*}
where $\tilde{K}(y)=\frac{1}{(2\pi)^2}\eta_1(y_1)\int_{{\mathbb{R}}^2}e^{i\xi'\cdot y'}B_{y_1}(\xi')\beta(2^{-j}|\xi'|)d\xi'$. By (\ref{B1}) and the support of $\beta$, it is easy to deduce
\begin{equation}\label{M2}
\sup_{t\in[1,2]}|A_{t,j}^0f(y)|\leq C_N2^{-jN}\int_{{\mathbb{R}}^3} \frac{|f(x)|}{(1+|y-x|)^{N}} dx.
\end{equation}
It follows from Young's inequality that
\[\biggl\| \sup_{t\in[1,2]} |{}A_{t, j}^0f(y)| \biggl\|_{L^{q}(\mathbb{R}^{3})} \le C2^{-jN} \|f\|_{L^{p}(\mathbb{R}^{3})} \]
holds for each $q \ge p \ge 1$.

Now we focus on the main term.  Set
\begin{align*}
A_{t,j}^1f(y):=&\frac{1}{(2\pi)^3}\int
_{{\mathbb{R}}^3}e^{i\xi\cdot y
}\int_{\mathbb{R}}e^{-it\xi_1x_1}\eta_1(x_1)e^{-it^{a_3}\xi_3\tilde{\Psi}(x_1,s)}E_{x_1}(\delta'_t\xi')dx_1\beta(2^{-j}|\delta'_t\xi'|)\hat{f}(\xi)d\xi
\end{align*}
with
\begin{equation*}
E_{x_1}(\delta'_t\xi'):=
\frac{\chi_{x_1}(t^{a_2}\xi_2/t^{a_3}\xi_3)}{(1+|\delta'_t\xi'|)^{1/2}}
A_{x_1}(\delta'_t\xi'),
\end{equation*}
and denote by $\mathcal{M}_{j}^1$ the corresponding maximal operator.  Notice that
\begin{align*}
\|\mathcal{M}_{j}^1f\|_{L^q}&=\biggl\|\sup_{t\in[1,2]}\frac{1}{(2\pi)^2}\biggl|\int_{\mathbb{R}}\eta_1(x_1)\int
_{{\mathbb{R}}^2}e^{i(\xi'\cdot y'-t^{a_3}\xi_3\tilde{\Psi}(x_1,s))}E_{x_1}(\delta'_t\xi')
\\& \quad\times \beta(2^{-j}|\delta'_t\xi'|)f(y_1-tx_1,\widehat{\xi'})d\xi'dx_1\biggl|\biggl\|_{L^q(dy)}
\\
&\leq C\biggl \|\sup_{t\in[1,2]}\biggl|\int_{\mathbb{R}}\eta_1(\frac{x_1}{t})\int
_{{\mathbb{R}}^2}e^{i(\xi'\cdot y'-t^{a_3}\xi_3\tilde{\Psi}(\frac{x_1}{t},s))}E_{x_1/t}(\delta'_t\xi')
\\& \quad \times \beta(2^{-j}|\delta'_t\xi'|)f(y_1-x_1,\widehat{\xi'})d\xi'dx_1\biggl|\biggl\|_{L^q(dy)}
\\
& =C \biggl\|\sup_{t\in[1,2]}|\widetilde{A_{t,j}^1}f|\biggl\|_{L^q(dy)}=C\|\widetilde{\mathcal{M}_{j}^1}f\|_{L^q},
\end{align*}
where $f(x,\hat{\xi'})$ denotes the partial Fourier transform with respect to the $\xi'$ variables and
\begin{equation}\label{mainterm5}
\widetilde{A_{t,j}^1}f(y):=\int_{\mathbb{R}}\eta_1(\frac{x_1}{t})\int
_{{\mathbb{R}}^2}e^{i(\xi'\cdot y'-t^{a_3}\xi_3\tilde{\Psi}(\frac{x_1}{t},s))}E_{x_1/t}(\delta'_t\xi')
\beta(2^{-j}|\delta'_t\xi'|)f(y_1-x_1,\widehat{\xi'})d\xi'dx_1,
\end{equation}
and
\begin{equation*}
\widetilde{\mathcal{M}_{j}^1}f(y):=\sup_{t\in[1,2]}|\widetilde{A_{t,j}^1}f(y)|.
\end{equation*}

Set
\begin{equation*}
Q_{x_1}(y',t,\xi')=\xi'\cdot y'-t^{a_3}\xi_3\tilde{\Psi}(\frac{x_1}{t},s).
\end{equation*}

Using the routine reasoning detailed before, it suffices to consider the local maximal operator
\[ \widetilde{\mathcal{M}_{j,loc}^1}f(y):=  \sup_{t\in[1/2,4]} |\tilde{\rho}(t)\widetilde{A_{t,j}^1}f(y)|,\]
where $\tilde{\rho} \in C_0^{\infty}(\mathbb{R})$ denotes a bump function supported in $[1/2,4]$ such that $\tilde{\rho}(t)=1$ if $1\leq t\leq 2$.

 By Lemma \ref{lem:Lemma3}, we have
\begin{equation}\label{wellknoweest}
\begin{aligned}
\|\widetilde{\mathcal{M}_{j,loc}^1}f(y)\|^q_{L^q}&\leq C_p\biggl(\int
_{{\mathbb{R}}^3}\int_{1/2}^4\left|\tilde{\rho}(t)\widetilde{A_{t,j}^1}f(y)\right|^q dtdy\biggl )^{1/q'}\\
&\quad \times \biggl(\int
_{{\mathbb{R}}^3}\int_{1/2}^4\left|\frac{\partial}{\partial t}\left(\tilde{\rho}(t)\widetilde{A_{t,j}^1}f(y)\right)\right|^qdtdy\biggl)^{1/q}.
\end{aligned}
\end{equation}

Moreover, we can choose $\tilde{\eta}_1\in C_0^{\infty}(\mathbb{R})$ such that $\tilde{\eta}_1=1$  on the support of $\eta_1$, then we have
\begin{equation*}
\frac{\partial}{\partial t}\widetilde{A_{t,j}^1}f(y)  =\int_{\mathbb{R}}\tilde{\eta}_1(\frac{x_1}{t})\int
_{{\mathbb{R}}^2}e^{iQ_{x_1}(y',t,\xi')}h(t,j,x_1,\xi')f(y_1-x_1,\widehat{\xi'})d\xi'dx_1,
\end{equation*}
where
\begin{align*}
h(t,j,x_1,\xi')&=\left(-t^{-2}x_1\eta_1'(\frac{x_1}{t})+\frac{\partial}{\partial t}Q_{x_1}(y',t,\xi')\right)E_{x_1/t}(\delta'_t\xi')\beta(2^{-j}|\delta'_t\xi'|)
\\
& +\eta_1(\frac{x_1}{t})\frac{\partial}{\partial t}E_{x_1/t}(\delta'_t\xi')\beta(2^{-j}|\delta'_t\xi'|)
+\eta_1(\frac{x_1}{t})E_{x_1/t}(\delta'_t\xi')\frac{\partial}{\partial t}\beta(2^{-j}|\delta'_t\xi'|),
\end{align*}
and
\begin{equation*}
\left|\frac{\partial}{\partial t}Q_{x_1}(y',t,\xi'))\right|=\left|\xi_3\frac{\partial}{\partial t}(t^{a_3}\Psi(\frac{x_1}{t},\psi(\frac{x_1}{t},s),s))\right|\leq C|\xi_3|,
\end{equation*}
since $t\approx 1$, $x_1$ and the support of $\chi_{x_1}$ are sufficiently small.  Clearly we only need to show the $L^p \rightarrow L^{q}$-boundedness of the operator $\widetilde{A_{t,j}^1}$.

Furthermore, let $\eta_0\in C_0^{\infty}(\mathbb{R})$ denote a non-negative function so that for arbitrary $t\in [1/2,4]$, $\eta_1(\frac{x_1}{t})\leq \eta_0(x_1)$. Then
\begin{align*}
&\biggl(\int
_{{\mathbb{R}}^3}\int_{1/2}^4|\tilde{\rho}(t)\widetilde{A_{t,j}^1}f(y)|^qdtdy\biggl)^{1/q}
\\
&\lesssim\Biggl\|\int_{\mathbb{R}}\eta_0(x_1)\biggl\|\tilde{\rho}(t)\int
_{{\mathbb{R}}^2}e^{iQ_{x_1}(y',t,\xi')}E_{x_1/t}(\delta'_t\xi')\beta(2^{-j}|\delta'_t\xi'|)\\
&\quad\quad\quad\times
f(y_1-x_1,\widehat{\xi'})d\xi'
\biggl\|_{L^q([1/2,4]\times \mathbb{R}^2,dtdy')}dx_1
\Biggl\|_{L^q(\mathbb{R}, dy_1)}.
\end{align*}

Hence we can apply Theorem \ref{lpqlocalsmoothing} for $ \frac{1}{2p} < \frac{1}{q} \le  \frac{3}{5p}, \frac{3}{q} \le 1-\frac{1}{p} $,  supp $\hat{f} \subset \{\xi \in \mathbb{R}^{2}: |\xi| \sim 2^{j}\}$, and obtain
\begin{align*}
&\biggl(\int
_{{\mathbb{R}}^3}\int_{1/2}^4|\tilde{\rho}(t)\widetilde{A_{t,j}^1}f(y)|^qdtdy\biggl)^{1/q}\\
&\leq C_{p,q} 2^{(-\frac{3}{q} + \frac{1}{p})j+\epsilon j}
\biggl\|\int_{\mathbb{R}}\eta_0(x_1)\|f(y_1-x_1,y')\|_{L^p({\mathbb{R}}^2,dy')}dx_1\biggl\|_{L^p(\mathbb{R},dy_1)}
\\
&\leq C_{p,q} 2^{(-\frac{3}{q} + \frac{1}{p})j+\epsilon j}\|f\|_{L^p}\|\eta_0\|_{L^1}.
\end{align*}
Thus by the inequality (\ref{wellknoweest}), we deduce
\begin{align*}
&\|\widetilde{\mathcal{M}_{j,loc}^1}f(y)\|_{L^q}^q\\
&\leq C_{p,q} \biggl( 2^{(-\frac{3}{q} + \frac{1}{p})j+\epsilon j}\|f\|_{L^p}\|\eta_0\|_{L^1}\biggl)^{q-1}2^{(1-\frac{3}{q} + \frac{1}{p})j+\epsilon j} \|f\|_{L^p}\|\eta_0\|_{L^1}
\\
&=C_{p,q} 2^{(\frac{q}{p}-2)j + q \epsilon j}\|f\|_{L^p}^q\|\eta_0\|_{L^1}^q.
\end{align*}
Since $\frac{1}{q} > \frac{1}{2p}$, we have proved  Theorem \ref{3nonvanish} for $p, q$ satisfying $ \frac{1}{2p} < \frac{1}{q} \le  \frac{3}{5p}, \frac{3}{q} \le 1-\frac{1}{p} $. The proof of the theorem is completed by interpolating with the $L^{p}$-estimate from Theorem 1.6 in \cite{WL}.


\subsubsection{The case when $da_{2} \neq a_{3}$: surfaces of finite type}
In this subsection, we will provide the proof of Theorem \ref{maintheorem4}.

First we may assume $\eta(x)=\eta_1(x_1)\eta_2(x_2)$ with non-negative functions $\eta_1$,
$\eta_2$ $\in C_0^{\infty}({\mathbb{R}})$. Let $(y_2, y_3)=y'$ and $(\xi_2,\xi_3)=\xi'$. Denote   $(t^{a_{2}}\xi_2, t^{a_{3}}\xi_3)$ by
$\tilde{\delta}_t\xi'$. We choose $B>0$ very small and  $\tilde{\rho}\in C_0^{\infty}(\mathbb{R})$ such that supp $\tilde{\rho}\subset\{x\in \mathbb{R}:B/2\leq|x|\leq 2B\}$  and $\sum_k\tilde{\rho}(2^kx)=1$.

We decompose the averaging operator as follows:
\begin{equation*}
A_tf(y) =\sum_k\widetilde{A_t^k}f(y),
\end{equation*}
where $\widetilde{A_t^k}f(y)=\int_{\mathbb{R}^3}f(y_1-tx_1,y_2-t^{a_{2}}x_2,y_3-t^{a_{3}}(x_2^d\Phi(x_{1}, x_2)+c))\eta_1(x_1)\eta_2(x_2)\tilde{\rho}(2^kx_2)dx$.

One observes that it suffices to prove the following estimate
\begin{equation*}
\sum_k2^{(d+1)k(\frac{1}{p}-\frac{1}{q})-k}\|\sup_{t \in [1, 2]}|A_t^k|\|_{L^p\rightarrow L^q}\leq C_{p,q},
\end{equation*}
where
\begin{equation}
A_t^kf(y):=\int_{\mathbb{R}^2}f(y_1-tx_1,y_2-t^{a_{2}}x_2,y_3-t^{a_{3}}(x_2^d\Phi(x_{1},\frac{x_2}{2^k}) +c2^{dk}))\eta_1(x_1)\eta_2(\frac{x_2}{2^k})\rho(x_2)dx.
\end{equation}

By means of the Fourier inversion formula, we can write
\begin{equation*}
 A_t^kf(y)=\frac{1}{(2\pi)^3}\int_{{\mathbb{R}}^3}e^{i\xi\cdot
  y} e^{-ic2^{dk}t^{a_{3}}\xi_{3}} \widehat{\eta_1}(t\xi_1)\widehat{d\mu_{k,m}}(\tilde{\delta}_t\xi')\hat{f}(\xi)d\xi,
\end{equation*}
where
\begin{align*}
 \widehat{d\mu_{k,m}}(\tilde{\delta}_t\xi')&=\int_{\mathbb{R}}e^{-i(t^{a_{2}}\xi_2x_2+t^{a_{3}}\xi_3x_2^d(\Phi(x_{1}, \frac{x_2}{2^k}))}\tilde{\rho}(x_2)\eta_2(2^{-k}x_2)dx_2.
\end{align*}
Set  $A_{t}^{k,0}f(y):=\sum_{j\leq 0}^{\infty}A_{t,j}^kf(y)$ and $\mathcal {M}^{k,0}f(y):=\sup_{t \in [1, 2]}|A_{t}^{k,0}f(y)|$. One can easily verify that
\begin{equation*}
A_{t}^{k,0}f(y)=f*K^{\sigma}_{\delta_{t^{-1}}}(y),
\end{equation*} 
where $\sigma=(0,0,c2^{dk})$ and
\begin{align*}
& K^{\sigma}_{\delta_{t^{-1}}}(y):=t^{-Q}K^{\sigma}(t^{-1}y_1,t^{-a_2}y_2,t^{-a_3}y_3), \\
& K^{\sigma}(y)=K(y-\sigma),\\
& K(y):=\frac{1}{(2\pi)^2}\eta_1(y_1)\int_{{\mathbb{R}}^2}e^{i\xi'\cdot
  y'}\widehat{d\mu_{k,y_1,d}}(\xi')\rho(|\xi'|)d\xi.
\end{align*}

Due to integration by parts,
\begin{equation}\label{K2}
|K(y)|\leq C_N(1+|y|)^{-N}, N\in \mathbb{N}.
\end{equation}
Now, we choose $N$ large enough in (\ref{K2}) and invoke Lemma \ref{lem:Lemma3} to obtain
\begin{align}
&\|\mathcal{M}^{k,0}f\|_{L^{q}(\mathbb{R}^{3})}= \|\sup_{t \in [1,2]}|f*K^{\sigma}_{\delta_{t^{-1}}}|\|_{L^{q}(\mathbb{R}^{3})} 
\lesssim (c2^{\frac{dk}{q}} +1 )\|f\|_{L^{p}(\mathbb{R}^{3})}.\nonumber
\end{align}

Hence, it suffices to consider the maximal operator $\sup_{t \in [1,2]}|A_{t,j}^kf(y)|$. Furthermore, a standard application of the method of stationary
phase requires us to obtain the $L^p \rightarrow L^{q}$-boundedness of the operator given by
\begin{equation*}
\begin{aligned}
\widetilde{\mathcal{M}_{j}^{k,1}}f(y)&:=\sup_{t\in[1,2]}|\widetilde{A_{t,j}^{k,1}}f(y)|\\
&=\sup_{t\in[1,2]}\biggl|\int_{\mathbb{R}}\eta_1(\frac{x_1}{t})\int
_{{\mathbb{R}}^2}e^{iQ_{k,x_1,d}(y',t,\xi')}E_{k,x_1/t,d}
(\delta'_t\xi')dx_1\\
&\quad\times\beta(2^{-j}|\delta'_t\xi'|)f(y_1-x_1,\widehat{\xi'})d\xi'dx_1\biggl|,
\end{aligned}
\end{equation*}
where
\begin{equation*}
Q_{k,x_1,d}(y',t,\xi'):=\langle\xi', (y_2,y_3-ct^{a_3}2^{dk})\rangle-t^{a_3}\xi_3\tilde{\Psi}(\frac{x_1}{t},s,\delta),
\end{equation*}
\begin{equation*}
\tilde{\Psi}(\frac{x_1}{t},s,\delta):=\Psi(\frac{x_1}{t},\psi(\frac{x_1}{t},s,\delta),s,\delta),
\end{equation*}
and
\begin{equation*}
\Psi(x_1,x_2,s,\delta):=-sx_2+x_2^d\Phi(x_1,\delta x_2).
\end{equation*}

By Lemma \ref{lem:Lemma3}, we have
\begin{equation*}
\|\widetilde{\mathcal{M}_{j}^{k,1}}f\|^q_{L^q}\leq C_p\biggl(\int
_{{\mathbb{R}}^3}|\rho(t)\widetilde{A_{t,j}^{k,1}}f(y)|^q\biggl)^{1/q'}\biggl(\int
_{{\mathbb{R}}^3}|\frac{\partial}{\partial t}(\rho(t)\widetilde{A_{t,j}^{k,1}}f(y))|^q\biggl)^{1/q}.
\end{equation*}

Since
\begin{align*}
\left|\frac{\partial}{\partial t}Q_{k,x_1,d}(y',t,\xi'))\right|&=\left|-ca_3t^{a_3-1}\xi_32^{kd}+\xi_3\frac{\partial}{\partial t}(t^{a_3}\tilde{\Psi}(\frac{x_1}{t},s,\delta))\right|\leq C  (c2^{dk} +1) 2^j,
\end{align*}
then
$\frac{\partial}{\partial t}(\rho(t)\widetilde{A_{t,j}^{k,1}})$ behaves like $(c2^{dk} +1) 2^j\widetilde{A_{t,j}^{k,1}}$. Clearly we  only consider the $L^p \rightarrow L^q$-estimate for the operator $\widetilde{A_{t,j}^{k,1}}$.

Since $x_2\approx 1$ here, in addition to $\Phi(0,0)\neq 0$, then we can  proceed similarly as in the proof of Theorem \ref{maintheorem1}.
Finally for $p,q$ satisfying  $ \frac{1}{2p} < \frac{1}{q} \le  \frac{1}{p}, \frac{1}{q} \ge \frac{3}{p} -1 $,  supp $\hat{f} \subset \{\xi \in \mathbb{R}^{2}: |\xi| \sim 2^{j}\}$, we obtain
\begin{align*}
&\sum_k2^{(d+1)k(\frac{1}{p}-\frac{1}{q})-k}\sum_j\|\widetilde{\mathcal{M}_{j}^{k,1}}\|_{L^p\rightarrow L^q}\\
&= C_{p,q}\sum_k  2^{(d+1)k(\frac{1}{p}-\frac{1}{q})-k} (c2^{\frac{dk}{q}} + 1) \sum_j2^{- \epsilon(p,q) j}\|\eta_0\|_{L^1}.
\end{align*}
We have thus finished the proof of Theorem \ref{maintheorem4}.

\subsubsection{The case when $da_{2} = a_{3}$: surfaces of finite type}

In this subsection, we will prove Theorem \ref{theocurvanishi}. We only prove it for $d =2$ and the case when $d > 2$ will follow similarly.

First we may assume $\eta(x)=\eta_1(x_1)\eta_2(x_2)$ with non-negative functions $\eta_1$,
$\eta_2$ $\in C_0^{\infty}({\mathbb{R}})$, and $f\geq
0$, $a_2=1$, $a_3=2$. Let $(y_2, y_3)=y'$ and $(\xi_2,\xi_3)=\xi'$. Denote $a_1+3$ by $Q$ and $(t\xi_2, t^2\xi_3)$ by
$\tilde{\delta}_t\xi'$. We choose $B>0$ very small and  $\tilde{\rho}\in C_0^{\infty}(\mathbb{R})$ such that supp $\tilde{\rho}\subset\{x\in \mathbb{R}:B/2\leq|x|\leq 2B\}$  and $\sum_k\tilde{\rho}(2^kx)=1$.

Set
\begin{equation*}
A_tf(y):=\int_{\mathbb{R}^2}f(y_1-t^{a_1}x_1,y_2-tx_2,y_3-t^2x_2^2\phi(x_2))\eta(x)dx=\sum_k\widetilde{A_t^k}f(y),
\end{equation*}
where $\widetilde{A_t^k}f(y)=\int_{\mathbb{R}^2}f(y_1-t^{a_1}x_1,y_2-tx_2,y_3-t^2x_2^2\phi(x_2))\eta_1(x_1)\eta_2(x_2)\tilde{\rho}(2^kx_2)dx$.

It suffices to prove the following estimate
\begin{equation*}
\sum_k2^{3k(\frac{1}{p}-\frac{1}{q})-k}\|\sup_{t \in [1, 2]}|A_t^k|\|_{L^p\rightarrow L^q}\leq C_{p ,q},
\end{equation*}
where
\begin{equation*}
A_t^kf(y):=\int_{\mathbb{R}^2}f(y_1-t^{a_1}x_1,y_2-tx_2,y_3-t^2x_2^2\phi(\frac{x_2}{2^k}))\eta_1(x_1)\eta_2(\frac{x_2}{2^k})\rho(x_2)dx.
\end{equation*}

By means of the Fourier inversion formula, we can write
\begin{equation*}
 A_t^kf(y)=\frac{1}{(2\pi)^3}\int_{{\mathbb{R}}^3}e^{i\xi\cdot
  y}\widehat{\eta_1}(t^{a_1}\xi_1)\widehat{d\mu_{k,m}}(\tilde{\delta}_t\xi')\hat{f}(\xi)d\xi,
\end{equation*}
where
\begin{align*}
 \widehat{d\mu_{k,m}}(\tilde{\delta}_t\xi')&=\int_{\mathbb{R}}e^{-i(t\xi_2x_2+t^2\xi_3x_2^2\phi(\frac{x_2}{2^k}))}\tilde{\rho}(x_2)\eta_2(2^{-k}x_2)dx_2.
\end{align*}

Choose a non-negative function $\beta\in C_0^{\infty}(\mathbb{R})$
such that
\begin{equation*}
\hspace{0.2cm}\textrm{supp }\hspace{0.2cm} \beta \subset[1/2,2] \hspace{0.5cm}\textrm{and}
\hspace{0.5cm}\sum_{j\in \mathbb{Z}}\beta(2^{-j}r)=1 \hspace{0.5cm}
\textrm{for} \hspace{0.5cm} r>0,
\end{equation*}
and set
\begin{equation*}
A_{t,j}^kf(y):=\frac{1}{(2\pi)^3}\int_{{\mathbb{R}}^3}e^{i\xi\cdot y}\widehat{\eta_1}(t^{a_1}\xi_1)\widehat{d\mu_{k,m}}(\tilde{\delta}_t\xi')\beta(2^{-j}|\tilde{\delta}_t\xi'|)\hat{f}(\xi)d\xi.
\end{equation*}
As in the proof of Theorem \ref{3nonvanish}, it is sufficient to obtain the $L^{p} \rightarrow L^{q}$ estimate for $\sup_{t\in[1,2]}|A_{t,j}^kf(y)|$.

For fixed $t\in[1,2]$, we will first estimate $\widehat{d\mu_{k,m}}(\tilde{\delta}_t\xi')$.
Set
\begin{equation*}
\delta:=2^{-k},\hspace{0.3cm}s:=s(\xi',t)=-\frac{\xi_2}{t\xi_3},\hspace{0.2cm}\textrm{for}\hspace{0.2cm}\xi_3\neq 0,
\end{equation*}
\begin{equation*}
\Phi(s,x_2,\delta):=-sx_2+x_2^2\phi(\delta x_2).
\end{equation*}

A similar argument as in the proof of Theorem \ref{3nonvanish} and Theorem \ref{hhisomaintheorem} shows that we only need to consider   the operator $\widetilde{A_{t,j}^k}$ given by
 \begin{align*}
\widetilde{A_{t,j}^k}f(y)&:=\int_{\mathbb{R}}\eta_1(\frac{x_1}{t^{a_1}})\int
_{{\mathbb{R}}^2}e^{i(\xi'\cdot y'-t^2\xi_3\tilde{\Phi}(s,\delta))}\chi_{k,m}(\frac{\xi_2}{t\xi_3})
\frac{A_{k,m}(\tilde{\delta}_t\xi')}{(1+|\tilde{\delta}_t\xi'|)^{1/2}}
\beta(2^{-j}|\tilde{\delta}_t\xi'|)\\
&\quad \times f(y_1-x_1,\widehat{\xi'})d\xi'dx_1,
\end{align*}
where  $\tilde{\Phi}(s,\delta):=\Phi(s,\tilde{q}(s,\delta),\delta)$ and $x_2=\tilde{q}(s,\delta)$ is the solution of the equation
$\partial_2\Phi(s,x_2,\delta)=0$ and smoothly converges to the solution $\tilde{q}(s,0)=s$ of the equation $\partial_2\Phi(s,x,0)=0$  if we assume $\phi(0)=1/2$.

The phase function can be written as
\begin{equation*}
-t^2\xi_3\tilde{\Phi}(s,\delta):=\frac{\xi_2^2}{2\xi_3}+(-1)^{m+1}\delta^m\frac{\phi^{(m)}(0)}{m!}\frac{\xi_2^{m+2}}{t^m\xi_3^{m+1}}+R(t,\xi',\delta),
\end{equation*}
which is homogeneous of degree one and can be considered as a small perturbation of  $\frac{\xi_2^2}{2\xi_3}+(-1)^{m+1}\delta^m\frac{\phi^{(m)}(0)}{m!}\frac{\xi_2^{m+2}}{t^m\xi_3^{m+1}}$.

Notice that $\chi_{k,m}$ is a smooth function supported in the interval $[c_{k,m},\tilde{c}_{k,m}]$, for certain non-zero constants $c_{k,m}$ and $\tilde{c}_{k,m}$ dependent only on $k$ and $m$. $A_{k,m}$ is a symbol of order zero and $\{A_{k,m}(\tilde{\delta}_t\xi')\}_k$ is contained in a bounded subset of symbols of order zero. 

We will again focus on the local maximal operator
\[\widetilde{\mathcal{M}_{j,loc}^k}f(y):= \sup_{t\in[1/2, 4]} |\rho_{1}(y^{\prime},t)\widetilde{A_{t,j}^k}f(y)|,\]
where $\rho_1 \in C_0^{\infty}(\mathbb{R}^2\times [1/2,4])$ is a bump function.

By Lemma \ref{lem:Lemma3}, we have
\begin{equation*}
\begin{aligned}
\|\widetilde{\mathcal{M}_{j,loc}^k}f\|^q_{L^q}&\leq C_p\biggl(\int
_{{\mathbb{R}}^3}\int_{1/2}^4|\rho_1(y',t)\widetilde{A_{t,j}^k}f(y)|^qdtdy\biggl)^{1/q'}\\
&\quad \times \biggl(\int
_{{\mathbb{R}}^3}\int_{1/2}^4|\frac{\partial}{\partial t}(\rho_1(y',t)\widetilde{A_{t,j}^k}f(y))|^qdtdy\biggl)^{1/q}.
\end{aligned}
\end{equation*}

Moreover, we choose a non-negative function $\tilde{\eta}_1\in C_0^{\infty}(\mathbb{R})$ such that $\eta_1=1$ on the support of $\eta_1$, then
\begin{equation*}
\frac{\partial}{\partial t}\left(\widetilde{A_{t,j}^k}f(y)\right)=\int_{\mathbb{R}}\tilde{\eta}_1(\frac{x_1}{t})\int
_{{\mathbb{R}}^2}e^{i(\xi'\cdot y'-t^2\xi_3\tilde{\Phi}(s,\delta))}h_{k,m}(y,t,j,\xi')
f(y_1-x_1,\widehat{\xi'})d\xi'dx_1,
\end{equation*}
where
\begin{equation*}
|h_{k,m}(y,t,j,\xi')|\lesssim 2^{j/2}\delta^m+2^{-j/2}.
\end{equation*}
Now it is easy to see that
$(2^{j/2}\delta^m+2^{-j/2})^{-1}\frac{\partial}{\partial t}(\rho_1(y',t)\widetilde{A_{t,j}^k}f)$ behaves like $2^{j/2}\widetilde{A_{t,j}^k}f$. It is sufficient to estimate $\|\rho_1(y',t)\widetilde{A_{t,j}^k}f(y)\|_{L^q(\mathbb{R}^3\times[1/2,4],dt dy)}$.

Furthermore, choosing a function $\eta_0\in C_0^{\infty}(\mathbb{R})$, non-negative, such that for arbitrary $t\in [1,2]$, $a_1>0$, $\eta_1(\frac{x_1}{t^{a_1}})\leq \eta_0(x_1)$,  we get
 \begin{align*}
&\left|\int_{\mathbb{R}}\eta_1(\frac{x_1}{t^{a_1}})\int
_{{\mathbb{R}}^2}e^{i(\xi'\cdot y'-t^2\xi_3\tilde{\Phi}(s,\delta))}\chi_{k,m}(\frac{\xi_2}{t\xi_3})
\frac{A_{k,m}(\tilde{\delta}_t\xi')}{(1+|\tilde{\delta}_t\xi'|)^{1/2}}
\beta(2^{-j}|\tilde{\delta}_t\xi'|)f(y_1-x_1,\widehat{\xi'})d\xi'dx_1\right|\\
&\leq \int_{\mathbb{R}}\eta_0(x_1)\biggl|\int
_{{\mathbb{R}}^2}e^{i(\xi'\cdot y'-t^2\xi_3\tilde{\Phi}(s,\delta))}\chi_{k,m}(\frac{\xi_2}{t\xi_3})
\frac{A_{k,m}(\tilde{\delta}_t\xi')}{(1+|\tilde{\delta}_t\xi'|)^{1/2}}
\beta(2^{-j}|\tilde{\delta}_t\xi'|)f(y_1-x_1,\widehat{\xi'})d\xi'\biggl|dx_1.
\end{align*}

In order to apply Lemma \ref{lemmaL^2} for $j\leq km$, the local smoothing estimate Theorem \ref{L^(p,q)therom} for $j>km$ of the Fourier integral operators not satisfying the ``cinematic curvature" condition  uniformly, we freeze $x_1$. In fact, by  Minkowski's and Young's inequalities, we have
\begin{align*}
&\|\rho_1(y',t)\widetilde{A_{t,j}^k}f(y)\|_{L^q(\mathbb{R}^3\times[1/2,4],dtdy)}
\\
&\leq C\Bigg(\int
_{{\mathbb{R}}^3}\int\biggl(\rho_1(y',t)\int_{\mathbb{R}}\eta_0(x_1)\biggl|\int
_{{\mathbb{R}}^2}e^{i(\xi'\cdot y'-t^2\xi_3\tilde{\Phi}(s,\delta))}\chi_{k,m}(\frac{\xi_2}{t\xi_3})
\frac{A_{k,m}(\tilde{\delta}_t\xi')}{(1+|\tilde{\delta}_t\xi'|)^{1/2}}
\beta(2^{-j}|\tilde{\delta}_t\xi'|)\\
&\quad\times f(y_1-x_1,\widehat{\xi'})d\xi'\biggl|dx_1\biggl)^qdtdy\Biggl)^{1/q}
\\
& \leq C\|\eta_0\|_{L^1}\Biggl\|\biggl\|\rho_1(y',t)\int
_{{\mathbb{R}}^2}e^{i(\xi'\cdot y'-t^2\xi_3\tilde{\Phi}(s,\delta))}\chi_{k,m}(\frac{\xi_2}{t\xi_3})
\frac{A_{k,m}(\tilde{\delta}_t\xi')}{(1+|\tilde{\delta}_t\xi'|)^{1/2}}
\beta(2^{-j}|\tilde{\delta}_t\xi'|)\\
&\quad \times f(y_1,\widehat{\xi'})d\xi'
\biggl\|_{L^q({\mathbb{R}}^2\times[\frac{1}{2},4],dtdy')}\Biggl\|_{L^q(\mathbb{R},dy_1)}.
\end{align*}

Finally, together with the arguments from  the proof of Theorem \ref{3nonvanish} and Theorem \ref{hhisomaintheorem}, we finish the proof.

\subsection{Proof of continuity lemmas}
According to the arguments in the introduction, we only need to show Lemma \ref{continuous continuty lemma}. Notice that
\[ \biggl\|\sup_{t,s\in[1,2]: |t-s| \le |z|^{\frac{1}{d}}}  | A_{t} - A_{s}  |  \biggl\|_{L^{p} \rightarrow L^{q} } \approx \biggl\|\sup_{t\in[1,2], h \in [0, |z|^{\frac{1}{d}}]} | A_{t+h}  - A_{t}   | \biggl\|_{L^{p} \rightarrow L^{q}}.\]
Next we will prove
\begin{align}\label{goal}
\biggl\|\sup_{t\in[1,2], h \in [0, |z|^{\frac{1}{d}}]} | A_{t+h}f(y) - A_{t}f(y) |  \biggl\|_{L^{q}(\mathbb{R}^{n})} \lesssim |z|^{\epsilon_2} \|f\|_{L^{p}(\mathbb{R}^{n})}
\end{align}
for $p, q$ as in Lemma \ref{continuous continuty lemma}. It follows from the $L^{p} \rightarrow L^{q}$ estimate of $\sup_{t \in [1, 2]} |A_{t}|$ that
\begin{align*}
\biggl\|\sup_{t\in[1,2], h \in [0, |z|^{\frac{1}{d}}]}  | A_{t+h}f(y) - A_{t}f(y) |  \biggl\|_{L^{q}(\mathbb{R}^{n})} \lesssim   \|f\|_{L^{p}(\mathbb{R}^{n})}.
\end{align*}
Then we only  prove (\ref{goal}) for some $p_{0}, q_{0}$, and obtain the estimate for more $p, q$'s by interpolation.
In what follows, we only provide some details for the proof of averaging operators defined by (\ref{averagefinite1}), (\ref{averagefinite2}), (\ref{averagefinite3}), and (\ref{maximal3dimen}). The other cases can be proved similarly.

\textbf{Case 1.} For $A_{t}$  defined by (\ref{averagefinite1}), we decompose the support of $\eta(x)$ as in the proof of Theorem \ref{maintheorem1}. It suffices to show
\begin{equation}
\sum_{k \ge 1}2^{(d+1)k(\frac{1}{p} - \frac{1}{q})-k}\left\|\sup_{t\in[1,2], h \in [0, |z|^{\frac{1}{d}}]} |\widetilde{A_{t+h}^{k}} - \widetilde{A_{t}^{k}} |\right\|_{L^p\rightarrow L^q}\leq C_{p,q},
\end{equation}
 where
\[\widetilde{A_t^{k}}f(y):=\int_{\mathbb{R}}  f(y_1-t^{a_{1}}x,y_2-t^{a_{2}}(x^d\phi(\frac{x}{2^{k}}) + 2^{dk}c))\tilde{\rho}(x)dx.\]
Then we decompose $\xi$ as in the proof of Theorem \ref{maintheorem1}, we only need to consider the main term
\begin{equation*}
A^k_{t,j}f(y):=\int_{{\mathbb{R}}^2}e^{i(\xi\cdot y -ct^{a_{2}}2^{dk}\xi_{2}-t^{a_{2}} \xi_2\tilde{\Phi}(s,\delta))}\chi_{k,d,m}(t^{a_{1}-a_{2}}\frac{\xi_1}{ \xi_2})
\frac{A_{k,d,m}( \delta_{t}\xi)}{(1+| \delta_{t}\xi|)^{1/2}}
\beta(2^{-j}| \delta_{t}\xi|)\hat{f}(\xi)d\xi.
\end{equation*}
We recall the notations  for the reader's convenience. Here $\chi_{k,d,m}$ is a smooth function supported in the conical region $[c_{k,d,m},\widetilde{c_{k,d,m}}]$, for certain non-zero constant $c_{k,d,m}$ and $\widetilde{c_{k,d,m}}$ dependent only on $k, m$ and $d$. $A_{k,d,m}$ is a symbol of order zero in $\xi$ and $\{A_{k,d,m}(\delta_t\xi)\}_k$ is contained in a bounded subset of symbol of order zero.

We assume $\phi(0)=1/d$, then the phase function can be written as
\begin{align*}
-t^{a_{2}}\xi_2\tilde{\Phi}(s,\delta)=(d-1)t^{\frac{da_{1}-a_{2}}{d-1}}\xi_{2} \biggl(-\frac{\xi_{1}}{d\xi_{2}}\biggl)^{d/(d-1)}-t^{\frac{(a_{1}-a_{2})m}{d-1} + a_{2}} \xi_{2}\frac{\delta^m\phi^{(m)}(0)}{ m!}
\biggl(-\frac{\xi_1}{d\xi_2}\biggl)^{\frac{d+m}{d-1}}+ R(t,\xi,\delta,d),
\end{align*}
where $R(t,\xi,\delta,d)$ is homogeneous of degree one in $\xi$ and has at least $m+1$ power of $\delta$. The phase function is a small perturbation of $(d-1)t^{\frac{da_{1}-a_{2}}{d-1}}\xi_{2} \biggl(-\frac{\xi_{1}}{d\xi_{2}}\biggl)^{d/(d-1)}$.

On one hand, by Lemma \ref{lem:Lemma3} and the Plancherel's theorem,
\begin{align}
\left\|\sup_{t\in[1,2], h \in [0, |z|^{\frac{1}{d}}]} |\widetilde{A_{t+h, j}^{k}}f - \widetilde{A_{t, j}^{k}}f |\right\|_{L^{2}(\mathbb{R}^{2})} \lesssim |z|^{\frac{1}{2d}} (c2^{dk} + 1) 2^{\frac{j}{2}} \|f\|_{L^{2}(\mathbb{R}^{2})}.
\end{align}
On the other hand, from the proof of Theorem \ref{maintheorem1},
\begin{align}
\left\|\sup_{t\in[1,2], h \in [0, |z|^{\frac{1}{d}}]} |\widetilde{A_{t+h, j}^{k}}f - \widetilde{A_{t, j}^{k}}f |\right\|_{L^{4d}(\mathbb{R}^{2})} \lesssim (c2^{\frac{k}{4}} + 1) 2^{-\frac{j}{8d}} \|f\|_{L^{4d}(\mathbb{R}^{2})}.
\end{align}
By interpolation, we have
\begin{align}
&\left\|\sup_{t\in[1,2], h \in [0, |z|^{\frac{1}{d}}]} |\widetilde{A_{t+h, j}^{k}}f - \widetilde{A_{t, j}^{k}}f |\right\|_{L^{2d}(\mathbb{R}^{2})} \nonumber\\
&\lesssim |z|^{\frac{1}{2d(2d-1)}} (c2^{dk} + 1)^{\frac{1}{2d-1}} (c2^{\frac{k}{4}} + 1)^{\frac{2d-2}{2d-1}} 2^{-\frac{d-1}{4d(2d-1)}j} \|f\|_{L^{2d}(\mathbb{R}^{2})}. \nonumber
\end{align}
It follows that when the curve passes through the origin or equivalently when $c=0$,
\begin{align}
&\sum_{k \ge 1}2^{(d+1)k(\frac{1}{2d} - \frac{1}{2d})-k}\left\|\sup_{t\in[1,2], h \in [0, |z|^{\frac{1}{d}}]} |\widetilde{A_{t+h}^{k}} - \widetilde{A_{t}^{k}} |\right\|_{L^{2d}\rightarrow L^{2d}} \nonumber\\
&\lesssim |z|^{\frac{1}{2d(2d-1)}}  \sum_{k} 2^{-k} \sum_{j} 2^{-\frac{d-1}{4d(2d-1)}j}  \nonumber\\
&\lesssim |z|^{\frac{1}{2d(2d-1)}}; \nonumber
\end{align}

when the curve doesn't pass through the origin, i.e. when $c \neq 0$,
\begin{align}
&\sum_{k \ge 1}2^{(d+1)k(\frac{1}{2d} - \frac{1}{2d})-k}\left\|\sup_{t\in[1,2], h \in [0, |z|^{\frac{1}{d}}]} |\widetilde{A_{t+h}^{k}} - \widetilde{A_{t}^{k}} |\right\|_{L^{2d}\rightarrow L^{2d}} \nonumber\\
&\lesssim |z|^{\frac{1}{2d(2d-1)}}  \sum_{k} 2^{-\frac{d-1}{2(2d-1)}k} \sum_{j} 2^{-\frac{d-1}{4d(2d-1)}j} \nonumber\\
&\lesssim |z|^{\frac{1}{2d(2d-1)}}. \nonumber
\end{align}
Therefore,
\begin{align*}
\biggl\|\sup_{t \in[1,2], h \in [0, |z|^{\frac{1}{d}}]}  | A_{t+h}f(y) - A_{t}f(y)  |  \biggl\|_{L^{2d}(\mathbb{R}^{2})} \lesssim |z|^{\frac{1}{2d(2d-1)}} \|f\|_{L^{2d}(\mathbb{R}^{2})}.
\end{align*}
Interpolating with the corresponding $L^{p} \rightarrow L^{q}$ estimate, we complete the proof for $A_{t}$ defined by (\ref{averagefinite1}).

\textbf{Case 2.} Let us consider the averaging  operator given by (\ref{averagefinite2}). As for \textbf{case 1}, we only need to consider the main term
\begin{equation*}
A_{t,j}f(y):=\int_{{\mathbb{R}}^2}e^{i(\xi\cdot y -  (d-1) \xi_{2}  (-\xi_{1}/d\xi_{2} )^{d/(d-1)})}\chi_{d}(t^{1-d}\frac{\xi_1}{ \xi_2})
\frac{A_{d}( \delta_{t}\xi)}{(1+| \delta_{t}\xi|)^{1/2}}
\beta(2^{-j}| \delta_{t}\xi|)\hat{f}(\xi)d\xi,
\end{equation*}
where $\chi_{d}$ is a smooth function supported in the conical region $[c_{d},\widetilde{c_{d}}]$, for certain non-zero constants $c_{d}$ and $\widetilde{c_{d}}$ dependent only on  $d$. $A_{d}$ is a symbol of order zero in $\xi$.

By Lemma \ref{lem:Lemma3} and the Plancherel's theorem,
\begin{align*}
\left\|\sup_{t\in[1,2], h \in [0, |z|^{\frac{1}{d}}]} |\widetilde{A_{t+h, j}}f - \widetilde{A_{t, j}}f |\right\|_{L^{2}(\mathbb{R}^{2})} \lesssim |z|^{\frac{1}{2d}} 2^{-\frac{j}{2}}  \|f\|_{L^{2}(\mathbb{R}^{2})}.
\end{align*}
Hence,
\begin{align*}
\biggl\|\sup_{t\in[1,2], h \in [0, |z|^{\frac{1}{d}}]}  | A_{t+h}f(y) - A_{t}f(y)  |  \biggl\|_{L^{2}(\mathbb{R}^{2})} \lesssim |z|^{\frac{1}{2d}}  \|f\|_{L^{2}(\mathbb{R}^{2})},
\end{align*}
then we complete the proof by interpolating with the corresponding $L^{p} \rightarrow L^{q}$ estimate.

\textbf{Case 3.} By the proof of Theorem \ref{hhisomaintheorem}, the main term  for the averaging operator in (\ref{averagefinite3}) is given by
\begin{equation*}
A_{t,j}^kf(y):=\frac{1}{(2\pi)^2}\int_{{\mathbb{R}}^2}e^{i(\xi \cdot y-t^d\xi_2\tilde{\Phi}(s,\delta))}\chi_{k,m,d}(\frac{\xi_1}{t\xi_2})
\frac{A_{k,m,d}(\delta_t\xi)}{(1+|\delta_t\xi|)^{1/2}}\beta(2^{-j}|\delta_t \xi|)\hat{f}(\xi)d\xi.
\end{equation*}
Here
\begin{align*}
-t^d\xi_2\tilde{\Phi}(s,\delta)=\biggl(\frac{1}{d^{\frac{d}{d-1}}}-\frac{1}{d^{\frac{1}{d-1}}}\biggl)(-\frac{d\xi_1^d}{\xi_2})^{\frac{1}{d-1}}-\frac{\delta^m\phi^{(m)}(0)}{t^mm!}
\biggl(-\frac{\xi_1}{\xi_2^{\frac{m+1}{m+d}}}\biggl)^{\frac{d+m}{d-1}}+ R(t,\xi,\delta,d),
\end{align*}
where $R(t,\xi,\delta,d)$ is homogeneous of degree one in $\xi$ and has at least $m+1$ power of $\delta$.

By Lemma \ref{lem:Lemma3} and the Plancherel's theorem, when $1 \le j \le mk$,
\begin{align*}
\left\|\sup_{t\in[1,2], h \in [0, |z|^{\frac{1}{d}}]} | A_{t+h, j}^{k} f - A_{t, j}^{k} f |\right\|_{L^{2}(\mathbb{R}^{2})} \lesssim |z|^{\frac{1}{2d}}  2^{-\frac{j}{2}} \|f\|_{L^{2}(\mathbb{R}^{2})};
\end{align*}
when $j>mk$,
\begin{align*}
\left\|\sup_{t\in[1,2], h \in [0, |z|^{\frac{1}{d}}]} | A_{t+h, j}^{k} f - A_{t, j}^{k} f |\right\|_{L^{2}(\mathbb{R}^{2})} \lesssim |z|^{\frac{1}{2d}}  2^{\frac{j}{2}-mk} \|f\|_{L^{2}(\mathbb{R}^{2})}.
\end{align*}
Also, it follows from \cite{WL} that for $j\geq 1$,
\begin{align*}
\left\|\sup_{t\in[1,2], h \in [0, |z|^{\frac{1}{d}}]} | A_{t+h, j}^{k} f - A_{t, j}^{k} f |\right\|_{L^{4}(\mathbb{R}^{2})} \lesssim  2^{-\frac{1}{72}j + \epsilon j} \|f\|_{L^{4}(\mathbb{R}^{2})}.
\end{align*}
Interpolation yields for $j \ge 1$,
\begin{align*}
\left\|\sup_{t\in[1,2], h \in [0, |z|^{\frac{1}{d}}]} | A_{t+h, j}^{k} f - A_{t, j}^{k} f |\right\|_{L^{\frac{146}{37}}(\mathbb{R}^{2})} \lesssim |z|^{\frac{1}{146d}}  2^{-\frac{j}{146} + \epsilon j} \|f\|_{L^{\frac{146}{37}}(\mathbb{R}^{2})}.
\end{align*}
This implies
\begin{align*}
\biggl\|\sup_{t\in[1,2], h \in [0, |z|^{\frac{1}{d}}]}  | A_{t+h}f(y) - A_{t}f(y)  |  \biggl\|_{L^{\frac{146}{37}}(\mathbb{R}^{2})} \lesssim |z|^{\frac{1}{146d}} \|f\|_{L^{\frac{146}{37}}(\mathbb{R}^{2})},
\end{align*}
which completes the proof for $A_{t}$ defined by (\ref{averagefinite3}).

\textbf{Case 4.} For $A_{t}$ defined by (\ref{maximal3dimen}), the main term is given by
\begin{equation*}
\widetilde{A_{t,j}^1}f(y):=\int_{\mathbb{R}}\eta_1(\frac{x_1}{t})\int
_{{\mathbb{R}}^2}e^{i(\xi'\cdot y'-t^{a_3}\xi_3\tilde{\Psi}(\frac{x_1}{t},s))}E_{x_1/t}(\delta'_t\xi')
\beta(2^{-j}|\delta'_t\xi'|)f(y_1-x_1,\widehat{\xi'})d\xi'dx_1,\
\end{equation*}
where the phase function is homogeneous of degree one and $E_{x_1/t}$ is a symbol of order $-1/2$.

Combining with  the Plancherel's theorem and Young's inequality, we derive
\begin{align*}
\left\|\sup_{t\in[1,2], h \in [0, |z|^{\frac{1}{d}}]} |\widetilde{A_{t+h, j} }f - \widetilde{A_{t, j} }f |\right\|_{L^{2}(\mathbb{R}^{3})} \lesssim |z|^{\frac{1}{2d}}  2^{\frac{j}{2}} \|f\|_{L^{2}(\mathbb{R}^{3})}.
\end{align*}
It follows from the proof of the $L^{p} \rightarrow L^{q}$ estimate that
\begin{align*}
\left\|\sup_{t\in[1,2], h \in [0, |z|^{\frac{1}{d}}]} |\widetilde{A_{t+h, j} }f - \widetilde{A_{t, j} }f |\right\|_{L^{4}(\mathbb{R}^{3})} \lesssim    2^{-\frac{j}{8}} \|f\|_{L^{4}(\mathbb{R}^{3})}.
\end{align*}
The interpolation theorem implies
\begin{equation*}
\left\|\sup_{t\in[1,2], h \in [0, |z|^{\frac{1}{d}}]} |\widetilde{A_{t+h, j} }f - \widetilde{A_{t, j} }f |\right\|_{L^{\frac{24}{7}}(\mathbb{R}^{3})} \lesssim |z|^{\frac{1}{12d}}  2^{-\frac{j}{48}} \|f\|_{L^{\frac{24}{7}}(\mathbb{R}^{3})}.
\end{equation*}
Therefore, we have
\begin{equation*}
\left\|\sup_{t\in[1,2], h \in [0, |z|^{\frac{1}{d}}]} | A_{t+h} f -A_{t} f |\right\|_{L^{\frac{24}{7}}(\mathbb{R}^{3})} \lesssim |z|^{\frac{1}{12d}}   \|f\|_{L^{\frac{24}{7}}(\mathbb{R}^{3})},
\end{equation*}
and finish the proof for $A_{t}$ defined by (\ref{maximal3dimen}).


\begin{flushleft}
\vspace{0.3cm}\textsc{Wenjuan Li\\School of Mathematics and Statistics\\Northwestern Polytechnical University\\710129\\Xi'an, People's Republic of China}

\vspace{0.3cm}\textsc{Huiju Wang\\School of Mathematics Sciences\\University of Chinese Academy of Sciences\\100049\\Beijing, People's Republic of China}

\vspace{0.3cm}\textsc{Yujia Zhai\\Laboratoire de Math\'{e}matiques Jean Leray\\Nantes Universit\'{e} \\44322\\Nantes, France}

\end{flushleft}

\end{document}